\documentclass{amsart}
\usepackage[dvipsnames]{xcolor}
\usepackage[left=1in,top=1in,right=1in,bottom=1in]{geometry}
\usepackage{wasysym, mathtools, amssymb, tikz, mathrsfs, tikz-cd, enumerate, appendix, amsmath}

\usepackage[T1]{fontenc}
\usepackage{comment}

\usepackage[colorlinks=true, linkcolor=magenta, urlcolor=violet, citecolor=MidnightBlue]{hyperref}
\usetikzlibrary{matrix,arrows}

\input xy
\xyoption{all}

\newtheorem{theorem}{Theorem}[section]
\newtheorem{lemma}[theorem]{Lemma}
\newtheorem{corollary}[theorem]{Corollary}
\newtheorem{proposition}[theorem]{Proposition}
\newtheorem{conjecture}[theorem]{Conjecture}

\theoremstyle{definition}
\newtheorem{definition}[theorem]{Definition}
\newtheorem{example}[theorem]{Example}

\theoremstyle{remark}
\newtheorem{remark}[theorem]{Remark}

\def\owlset@endmark{\ensuremath{\scriptscriptstyle\blacksquare}}

\newcommand{\Aff}{{\mathbb A}}
\newcommand{\C}{{\mathbb C}}
\newcommand{\F}{{\mathbb F}}

\newcommand{\PP}{{\mathbb P}}
\newcommand{\bP}{{\mathbb P}}
\newcommand{\Q}{{\mathbb Q}}

\newcommand{\Z}{{\mathbb Z}}

\newcommand{\bZ}{\mathbb{Z}}
\newcommand{\bQ}{\mathbb{Q}}

\def\bbar#1{\setbox0=\hbox{$#1$}\dimen0=.2\ht0 \kern\dimen0
\overline{\kern-\dimen0 #1}}

\newcommand{\calD}{{\mathcal D}}

\newcommand{\calM}{{\mathcal M}}

\newcommand{\calO}{{\mathcal O}}
\newcommand{\calP}{{\mathcal P}}

\newcommand{\calT}{{\mathcal T}}

\newcommand{\calX}{{\mathcal X}}

\newcommand{\Diff}{\mathrm{Diff}}
\newcommand{\s}{\text{-sm}}

\DeclareMathOperator{\Supp}{Supp}

\DeclareMathOperator{\Aut}{Aut}

\DeclareMathOperator{\Pic}{Pic}

\DeclareMathOperator{\Spec}{Spec}

\newcommand{\he}{H-}
\newcommand{\Zpq}{(Z,p,q)}
\newcommand{\Zpqs}{(Z^\mathrm{sm},p,q)}
\newcommand{\dfour}{(\mathbb{P}^3,d,4)}
\newcommand{\dfours}{(\mathbb{P}^{3,\mathrm{sm}},d,4)}

\title{Moduli of surfaces in $\PP^3$}

\author{Kristin DeVleming}
\email{kdevleming@umass.edu}
\address{Department of Mathematics, University of Massachusetts, Amherst, MA 01003, USA.}

\begin{document}

\begin{abstract}
The main goal of this paper is to construct a compactification of the moduli space of degree $d \ge 5$ surfaces in $\mathbb{P}^3_{\C}$, i.e. a parameter space whose interior points correspond to (equivalence classes of) smooth surfaces in $\mathbb{P}^3$ and whose boundary points correspond to degenerations of such surfaces.  We consider a divisor $D$ on a Fano variety $Z$ as a pair $(Z, D)$ satisfying certain properties.  We find a modular compactification of such pairs and, in the case of $Z = \PP^3$ and $D$ a surface, use their properties to classify the pairs on the boundary of the moduli space. 

\end{abstract}

\maketitle

\section{Introduction}

The aim of this paper is to construct a compactification of the moduli space of degree $d$ surfaces in $\PP^3$.  In the case of plane curves, using the Hilbert scheme or GIT techniques, one can find a compactification of the space parameterizing smooth degree $d$ curves in $\PP^2$, but the boundary does not always have a good modular interpretation.  For instance, there are points on the boundary that correspond to several different limits of families of plane curves.  In \cite{hassett} for the degree $4$ case and \cite{hacking} for general degree, instead of studying curves $C$, the authors worked with pairs $(\PP^2, C)$ and certain allowable degenerations.  Remembering the embedding of $C$ in $\PP^2$ and extracting certain properties yielded a compactification with a modular interpretation.  

This paper stems from the natural generalization of Hacking's work: find a good compactification of the moduli space of degree $d$ surfaces $S$ in $\PP^3$ using pairs.  In fact, because the framework is not specific to $\PP^3$, we will solve a more general problem and find a modular compactification of anticanonical divisors $D_Z$ on a fixed Fano variety $Z$ using degenerations of the pair $(Z,D_Z)$.  

For $Z = \PP^3$ and $D$ a surface of degree $d = 5$, the moduli space of smooth surfaces has been understood by \cite{horikawa}.  However, even for smooth surfaces, this moduli space is not irreducible.  When fixing the numerical invariants $K_S^5 = 5$, $p_g = 4$, and $q = 0$ of quintic surfaces, even in the smooth case, one obtains a moduli space with two components.  These details will be further explored in Section \ref{sec:quintics}, but we mention it here to indicate the increase in complexity when passing from curves to surfaces and motivate our focus on smoothable pairs. 

To find a meaningful compactification of the moduli space of degree $d$ surfaces in $\PP^3$, and more generally, divisors $D_Z$ that are a rational multiple of the anticanonical divisor on a fixed Fano variety $Z$, we will follow Hacking's approach and study pairs $(X,D)$ that arise as limits of pairs $(Z,D_Z)$.  In fact, we will only consider what we call $\Zpqs$ \he stable pairs $(X,D)$ (see Definition \ref{def:stable}) which have prescribed singularities and satisfy the relationship $pK_X + qD \sim 0$ and admit a smoothing to $Z$.  As in Hacking's work, this class of pairs is particularly well-behaved.  In the case of $Z = \PP^3$, remembering the embedding of $D$ into $\PP^3$ allows us to not only have a modular compactification of a space parameterizing $(\PP^3, D)$ but also to classify the pairs appearing on the boundary of the moduli space. The first result is that the class of pairs defined does actually give a compactification of the moduli space of pairs $(Z, D_Z)$.   

\begin{theorem}\label{theorem:stack}
The moduli space of $\Zpqs$ \he stable pairs is a proper Deligne-Mumford stack.
\end{theorem}

While the smoothability assumption restricts us to one component of the moduli space, we can remove the assumption in the case of odd degree surfaces in $\PP^3$.  However, this moduli space is not irreducible. 

\begin{theorem}\label{theorem:stackP3}
For odd degree $d$, the moduli space of three-dimensional $\dfour$ \he stable pairs is a proper Deligne-Mumford stack.
\end{theorem}

The reason for requiring oddness of degree is that, without the smoothability assumption, the class of pairs defined is not obviously a bounded family.  If boundedness was known, Hacking's arguments in the plane curve case \cite{hacking} would imply that the moduli space is a Deligne-Mumford stack and properness would follow from a relatively standard arguments.  Hacon, M\textsuperscript{c}Kernan, and Xu proved a strong result about boundedness of families of certain pairs $(X,D)$.  It requires the coefficients of the divisors appearing in $D$ to belong to a DCC set.  Here, in the definition of a $\Zpq$ \he stable pair, one requires that $(X, (\frac{q}{p} + \epsilon) D)$ is slc for $\epsilon$ sufficiently small.  However, $\epsilon$ is not bounded from below, so results on boundedness like those in \cite{hmx} do not directly apply.  If $\epsilon$ was required to belong to a DCC set, \cite[Theorem 1.1]{hmx} would apply to show the given pairs belong to a bounded family. 

In a different direction, an initial goal of this project was to classify the singular pairs appearing on the boundary of the moduli space of $\dfours$ \he stable pairs.  In working on this problem, the classification results gave enough control of the singularities of the boundary of the moduli space for odd degree $d$ to obtain boundedness even for non-smoothable pairs.  In other words, regardless of what set $\epsilon$ lives in, the classification results for odd degree $d$ actually imply boundedness.  Therefore, for $\dfour$ \he stable pairs, the following theorems serve two purposes: explicit classification of singular threefolds appearing in the moduli space and a means to achieve boundedness without carefully studying the numbers $\epsilon$ that appear in \he stable pairs. The first result is about ambient spaces $X$ with mild singularities.  Because $D \sim_{\mathbb{Q}} -\frac{d}{4} K_{X}$, classifying all the possible threefolds $X$ appearing is the first step toward understanding $D$. 

\begin{theorem}\label{thm:canonical}
Given a $\dfours$ \he stable pair $(X,D)$, if $d$ is odd and $X$ has canonical singularities, then either 
\begin{enumerate}[(a)]
	\item $X \cong \PP^3$,
	\item $X$ is isomorphic to the cone over the anticanonical embedding of the quadric surface $\PP^1 \times \PP^1$, or
	\item $X \cong \PP(1,1,2,4)$, the cone over the anticanonical embedding of the singular quadric surface
\end{enumerate}

\noindent and $D \in \left| \mathcal{O}_{X}(-\frac{d}{4}K_X) \right|$ such that $(X, \frac{4}{d}D)$ is log terminal.
\end{theorem}

Certainly there are other examples of threefolds $X$ with canonical singularities and $-K_X$ ample that do not appear in the previous list, even some appearing in the boundary for even degree.  However, in the odd degree case, the boundary is very special.  For any degree, the previous theorem gives a divisorial component of the boundary of the moduli space. 

\begin{theorem}
 Let $X$ be the cone over the anticanonical embedding of $\PP^1 \times \PP^1$.  For any degree $d > 4$, (the closure of) the locus of $\dfours$ \he stable pairs $(X,D)$ in the moduli space forms a divisor $\calD_Q$ in $\calM_{\dfours}$ and the moduli space is smooth at the generic point of $\calD_Q$.  
\end{theorem}

More generally, in the moduli space of $(d,n+1)$ $\PP^n$-smoothable \he stable pairs, the locus parameterizing divisors on a cone over a quadric always forms a divisor. 

\begin{theorem}
Let $X \subset \PP(1^{n+1}, 2)$ be the cone over the degree $2$ embedding of a smooth quadric $Q \in | \calO_{\PP^n}(2) |$.  For any degree $d > n+1$, (the closure of) the locus of $(d, n+1)$ $\PP^n$-smoothable \he stable pairs $(X,D)$ forms a divisor $\calD_Q$ in the moduli space $\calM_{\PP^{n\s}, (d,n+1)}$.
\end{theorem}

In the case of $\PP^3$-smoothable \he stable pairs, we also study more complicated singularities.  In order to guarantee a proper moduli space, we consider semi log canonical (slc) pairs.  In Section \ref{sec:classification}, we show that pairs with strictly slc singularities can only appear in the moduli space of pairs with even degree $d$.  Therefore, for odd degree, we only need to consider semi log terminal pairs to construct the moduli space.   

\begin{theorem}\label{thm:mainresult}
Given a $\dfour$ \he stable pair with $d$ odd, $(X,\frac{4}{d}D)$, each component of the normalization of $(X,\frac{4}{d}D)$ is divisorial log terminal.
\end{theorem}

Considering for a moment the \he stable pairs $(X,D)$ with $X$ normal, the canonical threefolds $X$ appearing are described in Theorem \ref{thm:canonical}.  From Theorem \ref{thm:mainresult},  it remains to classify the (non-canonical) log terminal threefolds appearing as \he stable pairs.  Partial progress toward this goal is obtained in Section \ref{sec:lt}.  We hope to continue this study and classify non-normal threefolds (via a similar analysis of their normalization) in future work.

\subsection{A map of this paper.}
In the first half of the paper, we study $\Zpq$ \he stable pairs for an arbitrary smooth Fano variety $Z$.  We begin in Section \ref{sec:pairs} with a brief explanation of motivating ideas for this paper, we define \he stable pairs, and use the existence of minimal models to prove that a family of pairs over a punctured curve can be extended in an essentially unique way, justifying the definition. In Section \ref{sec:moduli}, we study deformation theory of \he stable pairs and further analyze the moduli space $\calM_{\Zpqs}$, proving that it is a proper Deligne-Mumford stack.  In the case of $Z = \PP^3$, we remove the smoothability assumption for odd degree $d$ and prove that $\calM_{\dfour}$ is a proper Deligne-Mumford stack. In the second half of the paper, we focus primarily on the case $Z = \PP^3$ to obtain finer classification results, generalizing to $Z = \PP^n$ when possible.  In Section \ref{sec:classification}, we prove a number of technical lemmas about extremal contractions in the minimal model program, use them to understand log canonical singularities appearing in degree $d$ $\PP^3$ \he stable pairs, and build up the necessary machinery to prove Theorem \ref{thm:mainresult}.  We prove a number of related results on the structure of slc Fano varieties and their connections to boundedness of odd degree $\PP^3$ \he stable pairs.  We also discuss canonical and log terminal Fano degenerations of $\PP^n$ as a step toward classifying all $\PP^n$-smoothable \he stable pairs.   Finally, in Section \ref{sec:quintics}, we explicitly study $\calM_{\PP^3\s, (5,4)}$, the moduli space of quintic surfaces in $\PP^3$.  We compare this to existing compactifications and discuss the differences between $\calM_{\dfour}$ and $\calM_{\dfours}$.

\subsection{Acknowledgments.}

This paper is based on my PhD thesis.  I would like to thank my advisor, S\'andor Kov\'acs, for suggesting this problem and for his years of generous guidance, encouragement, and insight.  I also benefited from conversations with Max Lieblich, Charles Godfrey, Siddharth Mathur, and Kenneth Ascher.  I would also like to thank Paul Hacking for comments on an earlier version of this paper.  I would like to thank J\'anos Koll\'ar for comments on an earlier draft and the crucial ideas for extending the proof of Theorem \ref{theorem:stack} to arbitrary smooth Fano varieties.  Finally, I thank the referee for their detailed and careful comments greatly improving the exposition of this paper.

\section{\he stable pairs}\label{sec:pairs}

Throughout, we will work with varieties over $\C$.  A pair $(X,D)$ is a variety $X$ with a divisor $D = \sum a_iD_i$ that is a formal linear combination of prime divisors.  We will restrict to the case $a_i \in \bQ$.  We use the standard definitions for (log) terminal and (log) canonical singularities found, for example, in \cite[Section 2.3]{km98}.

\subsection{Previous work}

We will construct a compactification of the moduli space of degree $d$ surfaces in $\PP^3$ using ideas first introduced in \cite{ksb} and mention here the relationship between this compactification and related ones.  The KSB (Koll\'ar, Shepherd-Barron) compactification is a construction of a parameter space for stable objects with slc singularities.  Our precise definition of stable objects will appear in Section \ref{sec:pairs}.  One can consider these compactifications as higher dimensional analogues of $\overline{M}_{g,n}$, the moduli space of stable genus $g$ curves with $n$ marked points \cite{alexeev}.  

The framework motivating this definition of the compactification comes from earlier work of Hassett and Hacking, studying moduli spaces of degree $d$ plane curves by considering them as pairs $(\PP^2, C)$.  A compactification of the moduli space of the space of pairs $(\PP^2, C_4)$, where $C_4$ has degree $4$ was studied in \cite{hassett} and, more generally, for any degree $d \ge 4$, a compactification of pairs $(\PP^2, C_d)$ was constructed in \cite{hacking}.  The work in \cite{hacking} provides much of the foundation for this paper as the study of pairs $(\PP^3, S_d)$ is a natural generalization.  In a different direction, one could generalize \cite{hacking} to study pairs $(S,D)$ for other del Pezzo surfaces.  In particular, a compactification of the space of pairs $(\PP^1 \times \PP^1, C_{3,3})$ using similar machinery is described in \cite{deohan}.  In general, Koll\'ar and Xu show that one can make a coarse moduli space of \textit{polarized log Calabi Yau pairs} where each irreducible component is projective (\cite{KX}).  

There is another approach to the study of plane curves and surfaces in $\PP^3$ using the tools provided by GIT.  However, as discussed in \cite{wangxu}, GIT begins to fail for higher degree $d$.  In the GIT construction, the moduli space depends on the power $r$ of $\omega_S$ (or other ample line bundle) being used to embed smooth surfaces into projective space and it is shown that the moduli spaces in this construction do not stabilize.  In particular, in \cite[Theorem 1 (2)]{wangxu}, there are families of degree $d > 30$ smooth surfaces over a punctured base whose limit does not stabilize as $r$ increases.  Therefore, it benefits us to approach the problem for general degree $d$ surfaces using the framework of stable pairs instead of GIT.  Some comparison to the GIT case for degree $d = 5$ will be given in Section \ref{sec:quintics}.

\subsection{Definitions}

We are interested in studying the moduli space of hypersurfaces $S$ (of a fixed degree) in $\mathbb{P}^3$.  As motivated in the introduction, instead of studying moduli of such $S$ directly, we study moduli of pairs $(X, D)$ where $X$ is a degeneration of $\mathbb{P}^3$ and $D$ is a degeneration of $S$. The motivating work for plane curves in \cite{hacking} considered moduli of pairs $(X,D)$ where $X$ was a slc surface that smoothed to $\mathbb{P}^2$ and $D$ was a divisor such that $d K_X + 3 D \sim 0$ and $K_X + (\frac{3}{d} + \epsilon) D$ was ample for some (and hence all) $\epsilon$ sufficiently small.  Hacking was able to show that, for $d$ not a multiple of 3, this moduli stack is proper, separated, and smooth, and was also able to explicitly determine the surfaces $X$ (and thus the divisors $D$) appearing on the boundary of the moduli space. 

Now, consider the direct generalization of \cite{hacking}: a compactification of the moduli space of degree $d$ hypersurfaces in $\mathbb{P}^3$.  We present the following definitions.  This paper will focus on dimension 3, but the definitions make sense for general Fano varieties and divisors of arbitrary dimension.  First, we define a $\Zpq$ stable pair. 

\begin{definition}\label{def:stable}
For a given Fano variety $Z$, let $n = \dim Z$, and fix integers $p,q> 0$ such that $\frac{q}{p} <1$.  A pair $(X,D)$, where $X$ is an $n$-fold and $D$ is an effective $\mathbb{Q}$-Cartier $\bZ$-divisor, is said to be $\Zpq$ stable in the sense of Hacking, or $\Zpq$ \he stable, if
\begin{itemize}
\item The pair $(X, (\frac{q}{p}+\epsilon)D)$ is semi log canonical and the divisor $K_X + (\frac{q}{p} + \epsilon)D$ is ample for some $\epsilon > 0$.
\item The divisor $pK_X + qD$ is linearly equivalent to zero.
\item $K_X^n = K_{Z}^n $ and $X$ is Cohen-Macaulay.
\end{itemize}
\end{definition}

Note that this definition does not require that $Z$ admits a deformation to $X$. 

Next, we define $Z$-smoothable pairs.  Note that the last condition implies the last condition in the non-smoothable definition by Proposition \ref{prop:constantvolume}.

\begin{definition}\label{def:stablesmoothable}
Fix a smooth Fano variety $Z$ and integers $p,q> 0$ such that $\frac{q}{p} <1$.  A pair $(X,D)$, where $X$ is an $n$-fold and $D$ is an effective $\mathbb{Q}$-Cartier $\bZ$-divisor, is said to be $\Zpqs$ stable in the sense of Hacking, or $\Zpqs$ \he stable if
\begin{itemize}
\item The pair $(X, (\frac{q}{p}+\epsilon)D)$ is semi log canonical and the divisor $K_X + (\frac{q}{p} + \epsilon)D$ is ample for some $\epsilon > 0$.
\item The divisor $pK_X + qD$ is linearly equivalent to zero.
\item There is a deformation $(\mathcal{X}, \mathcal{D})/T$ of $(X,D)$ over the germ of a curve such that the general fiber $\mathcal{X}_t$ is smooth, $\mathcal{X}_t$ admits a smooth deformation to $Z$, and the divisors $K_{\mathcal{X}/T}$ and $\mathcal{D}$ are $\mathbb{Q}$-Cartier.
\end{itemize}
\end{definition}

Note that the last condition in Definition \ref{def:stablesmoothable} serves two purposes: one, to ensure that $X$ is smoothable (restricting to certain components of the moduli space) but two, to ensure that $X$ is smoothable to a fixed component of the moduli space.  For example, the Fano surfaces $\bP^1 \times \bP^1 $ and $\mathbb{F}_1$ have the same volume, but are not deformation equivalent, and correspond to two different components of a moduli space.  If, for example, $Z = \bP^1 \times \bP^1$, we would require that $X$ admits a smoothing to $\bP^1 \times \bP^1$.  If $Z$ itself has moduli (for example, $Z$ a cubic surface in $\bP^3$), we require that $X$ admit a smoothing to a smooth cubic surface, deformation equivalent to $Z$.  

Our goal is to construct a moduli space $\calM_{\Zpqs}$ of these pairs. 

\begin{remark}
Note that the relationship $pK_X + qD \sim 0$ implies that, if $K_X^n$ is fixed, so is the volume of the pair $(K_X + D)^n$.  The condition that $X$ is Cohen-Macaulay potentially restricts us to some components of the general moduli space, but by \cite[Corollary 7.13]{dubois}, the components parameterizing non-Cohen Macaulay pairs are disconnected from these components. 
\end{remark}

\begin{remark}
In \cite{deohan}, for del Pezzo surfaces $S$ and curves $C$, the authors make the natural generalization of Hacking's definition and refer to pairs $(S,C)$ as \textit{almost K3 stable}.  This definition coincides with their definition.
\end{remark}

There are trivial advantages to studying $\Zpq$ (and $\Zpqs$) \he stable pairs, summarized in the following lemma.  These are the primary reasons $\dfours$ \he stable pairs are classifiable. 

\begin{lemma}
If $(X,D)$ is an \he stable pair, the following hold:
    \begin{enumerate}[(a)]
    \item $K_X$ is anti-ample.
    \item $D$ is ample.
    \item Both $K_X$ and $D$ are $\Q$-Cartier.
    \item If $X$ is strictly slc, the strictly slc locus of $X$ is not contained in the support of $D$.
    \end{enumerate}

\end{lemma}

\subsection{Limits of \he stable pairs exist}

In recent work \cite{KX}, J\'anos Koll\'ar and Chenyang Xu prove the following theorem:

\begin{theorem}{\cite[Theorem 2]{KX}}\label{thm:KX}
The irreducible components of the coarse moduli space of $\Zpq$ \he stable pairs are projective. 
\end{theorem}

We are very grateful to J\'anos Koll\'ar for pointing our attention to this result and the crucial steps of the proof.  In particular, this implies that each component of the moduli stack $\calM_{\Zpq}$ is proper:

\begin{theorem}\label{thm:proper}
Let $0 \in T$ be the germ of a curve and write $T^\times = T - 0$.  Let $X$ be a smooth Fano variety and $\mathcal{D}^\times \subset X^\times \times T^\times$ be a family of smooth hypersurfaces over $T^\times$ such that $pK_X + qD \sim 0$, $0 < \frac{q}{p} < 1$.  Then, there exists a finite surjective base change $T' \to T$ and a family $(\mathcal{X}, \mathcal{D})/T'$ of \he stable pairs extending the pullback of the family $(X \times T^\times, \mathcal{D}^\times)/T^\times$ such that the divisors $K_\mathcal{X}$ and $\mathcal{D}$ are $\mathbb{Q}$-Cartier.  Any two such families become isomorphic after a further finite surjective base change.
\end{theorem}

\begin{proof}(Sketch.)
    This is a consequence of Koll\'ar and Xu's Theorem, although we provide a brief sketch of the proof for convenience of the reader.  By \cite[Theorem 0.3]{abramovichkaru}, we can complete $(X \times T^\times, \mathcal{D}^\times)/T^\times$ to a locally stable family $(\mathcal{X}, \mathcal{D})/T'$.  Let $r = \frac{q}{p}$.  Choosing an appropriate ample divisor $L$ on $\mathcal{X}$, one can run a relative minimal model program with scaling of $L + \mathcal{D}$ to obtain a minimal family over a smooth base, which is locally stable by \cite[Corollary 10]{KNX}.  Furthermore, the total space is $\mathbb{Q}$-factorial and one can show that $K_{\mathcal{X}^\mathrm{min}} + r\mathcal{D}^\mathrm{min}$ is relatively trivial and $\mathcal{D}^{\mathrm{min}}$ is relatively big, so by \cite[Theorem 1.1]{haconxu} one can find a relative canonical model $(\mathcal{X}^c, (r+\epsilon)\mathcal{D}^c)$.  This is a stable family and one can show that $K_{\mathcal{X}^c} + r\mathcal{D}^c$ is relatively semiample and trivial on the generic fiber, so trivial on the central fiber.  Because both $K_{\mathcal{X}^c} + r\mathcal{D}^c$ and $K_{\mathcal{X}^c} + (r + \epsilon) \mathcal{D}^c$ are $\mathbb{Q}$-Cartier, we have $K_{\mathcal{X}^c}$ and $\mathcal{D}^c$ are both $\mathbb{Q}$-Cartier.  Therefore, $(\mathcal{X}^c, \mathcal{D}^c)$ is the desired family of \he stable pairs.  Uniqueness follows from uniqueness of canonical models.  
\end{proof}

\section{The moduli space}\label{sec:moduli}

\subsection{Boundedness}

First, we show the moduli functor is bounded.  This argument is inspired by that in \cite[Theorem 1.1]{karu} and is similar to the proof of Koll\'ar and Xu's Theorem (Theorem \ref{thm:KX}).  As before, this is also a consequence of Theorem \ref{thm:KX}.  

\begin{theorem}
For a given Fano variety $Z$, the set of $\Zpqs$ \he stable pairs $(X,D)$ form a bounded family.  
\end{theorem}

\begin{proof}
    By a theorem of Matsusaka \cite[Theorem 2.4]{matsusaka}, there is a fixed integer $m$ such that all smooth Fano varieties $X$ with fixed anticanonical Hilbert polynomial $H_X(l) = \chi(X,-lK_X) = \chi(Z,-lK_Z)$ have $-mK_X$ very ample with no higher cohomology, hence a fixed projective space $\PP^n$ containing all such $X$.  In particular, $Z$ is one such $X$, as is any smooth variety $W$ deformation equivalent to $Z$.  Let $r = \frac{q}{p}$.  Similarly, we can choose $m$ such that $-mK_W|_D$ is very ample with no higher cohomology for any smooth $D \sim_{\mathbb{Q}} -r^{-1}K_W$.  We can find a product $\mathrm{Hilb}$ of Hilbert schemes parameterizing embeddings $D \subset W \subset \PP^n$.  Let $B_0 \subset \mathrm{Hilb}$ be the subscheme parameterizing such pairs.  Taking its closure $B_0 \subset B \subset \mathrm{Hilb}$, there is a universal family $(\mathcal{X}, \mathcal{D}) \to B$ whose general fiber is deformation equivalent to the smooth Fano variety $Z$ with a smooth divisor $D$.  By \cite[Theorem 0.3]{abramovichkaru}, we can apply weak semistable reduction to obtain a locally stable family $(\mathcal{Y}, \mathcal{D}_{\mathcal{Y}}) \to B'$, where $B' \to B$ is a modification of $B$ and $\mathcal{Y} \to \mathcal{X}' = \mathcal{X} \times_B B'$ is a modification of $\mathcal{X}'$.  We can apply the techniques of proof of Theorem \ref{thm:KX} and Theorem \ref{thm:proper} to find the $K_{\mathcal{Y}} + (r + \epsilon) \mathcal{D}_{\mathcal{Y}}$ canonical model over $B'$ for any $\epsilon$ sufficiently small; call it $(\mathcal{Y}^c, \mathcal{D}^c)$.  It is straightforward to show that every fiber satisfies the conditions of being an \he stable pair; now it suffices to show that every \he stable pair appears as a fiber.  
    
    Because the weak semistable reduction potentially modifies even smooth fibers, we first need to show the generic smooth fibers are isomorphic to the generic fibers of the canonical model.  Let $(\mathcal{X}^0, \mathcal{D}^0)$ be the universal family over $B_0$ and let $({\mathcal{X}^0}', {\mathcal{D}^0}')$ be the family over the preimage $B_0' \subset B'$.  Let $(\mathcal{Y}^{0,c}, \mathcal{D}^{0,c}) \to B_0'$ be the fibers of the canonical model over $B_0'$.  Because the fibers of $({\mathcal{X}^0}', {\mathcal{D}^0}') \to B_0'$ are smooth Fano varieties with smooth divisors $D$ over a smooth base, it follows that $({\mathcal{X}^0}', {\mathcal{D}^0}')$ is its own relative canonical model, and birational to $(\mathcal{Y}^{0,c}, \mathcal{D}^{0,c})$ by construction.  Therefore, by uniqueness of canonical models, we must have $({\mathcal{X}^0}', {\mathcal{D}^0}') \cong (\mathcal{Y}^{0,c}, \mathcal{D}^{0,c}). $
    
    Now, let $(Y,D_Y)$ be any \he stable pair.  By the smoothability assumption, $(Y,D_Y)$ admits a smoothing over the germ of a curve to $(W,D_W)$ for some $D_W$, so $(Y, D_Y)$ is the central fiber of some family $(\mathcal{X}, \mathcal{D}) \to T$, where $T$ is a DVR.  Up to a finite base change, the generic fiber $(\mathcal{X}^\times, \mathcal{D}^\times) \to T^\times$ gives an embedding $T^\times \to B'$ which can be completed to a curve $T' \subset B'$.  The family over $T'$ has the same generic fiber as $(\mathcal{X}, \mathcal{D}) \to T$, but up to a finite base change, both families are canonical models of a common semistable resolution.  Therefore, their central fibers are isomorphic.
\end{proof}

As mentioned in Section \ref{sec:pairs}, one could define an alternate class of pairs, omitting the condition that each pair admits a smoothing.  Indeed, even for quintic surfaces in $\PP^3$, there should be a second component of the moduli space of smooth surfaces as discussed in Section \ref{sec:quintics}, so it seems natural to relax this condition.   While we do not have boundedness of non-smoothable pairs in general, we can prove boundedness of $\dfour$ \he stable pairs when $d$ is odd.  This will be further discussed in Section \ref{sec:classification} and proven in Theorem \ref{thm:boundednessP3}.  

\subsection{Algebraicity}

Using recent results, we can further analyze the structure of the moduli space of \he stable pairs.  We study \textit{Koll\'ar families} of pairs, $\mathbb{Q}$-Gorenstein families where $\omega_{\mathcal{X}}^{[n]}$ commutes with base change for all $n$.  In fact, we will require a stronger condition that both $\omega_{\mathcal{X}}^{[n]}$ and $\calO_{\mathcal{X}}(nD)$ commute with base change for all $n$. Using \cite[c.f. Theorem 4.4]{hacking} and following his work, one can show that the moduli space is indeed an algebraic stack.

\begin{definition}
Let $p \in X$ be a germ of an slc variety.  Define the index of $p$ in $X$ to be the minimal $N > 0$ such that $NK_X$ is Cartier.  Let $V \to X$ be the canonical covering $V = \text{Spec}_X \mathcal{O} \oplus \mathcal{O}(K_X) \oplus \dots \oplus \mathcal{O}((N-1)K_X)$, a $\mu_N$ quotient (c.f. \cite{ypg}).  A deformation $\mathcal{X}/S$ of $X$ is $\mathbb{Q}$-Gorenstein if there is a $\mu_N$-equivariant deformation $\mathcal{V}/S$ of $V$ whose quotient is $\mathcal{X}/S$.
\end{definition}

\begin{definition}
Let $\mathcal{X}/S$ be a flat family of slc varieties.  We say that $\mathcal{X}$ is weakly $\mathbb{Q}$-Gorenstein if, for some $N > 0$, $\omega_{\mathcal{X}/S}^{[N]}$ is invertible.  The minimal such $N$ is called the index of $\mathcal{X}$.
\end{definition}

The following lemmas show that $\mathbb{Q}$-Gorenstein implies weakly $\mathbb{Q}$-Gorenstein and that the conditions are equivalent if the general fiber is canonical and the base is a curve.

\begin{lemma}
Let $p \in X$ be a germ of an slc variety.  A $\mathbb{Q}$-Gorenstein deformation $\mathcal{X}/S$ of $X$ of index $N$ is weakly $\mathbb{Q}$-Gorenstein of index $N$.
\end{lemma}

\begin{proof} This follows directly from \cite[Lemma 3.3]{hacking}. \end{proof}

\begin{lemma}
Let $\mathcal{X}/T$ be a flat family of slc varieties over the germ of a curve.  If the general fiber has canonical singularities and $K_{\mathcal{X}}$ is $\mathbb{Q}$-Cartier, then $\mathcal{X}/S$ is a $\mathbb{Q}$-Gorenstein deformation of $\mathcal{X}_0$.
\end{lemma}

\begin{proof} Using the stronger inversion of adjunction result in \cite[Lemma 2.10]{zsolt}, the proof of Lemma 3.4 in \cite{hacking} applies directly. \end{proof}

Many properties of $\mathbb{Q}$-Gorenstein deformations are collected in \cite[Section 3]{hacking}.  In fact, $\mathbb{Q}$-Gorenstein deformations $X$ are exactly the deformations $\calX$ of $X$ satisfying the Koll\'ar condition that $\omega_\calX^{[n]}$ commutes with base change for all $n$ \cite[2.4]{hackinggentype}. However, the presence of the divisor $D$ can cause further obstructions to deforming $X$.  In particular, taking the canonical cover $V$ of $X$, the associated divisor $D_V$ does not have to be a Cartier divisor.  

\begin{example}\label{ex:Dnotcartier}
    Let $X$ be the cone over the anticanonical embedding of $\PP^1 \times \PP^1$, a section of $\calO_W(2)$ in $W = \PP(1,1,1,1,2)$.  This is a $\mathbb{Q}$-Gorenstein deformation of $\PP^3$ as computed in Example \ref{ex:O(2,2)} and appears in $\dfours$ \he stable pairs.  Note that $X$ has canonical singularities, so the canonical cover of $X$ is $X$ itself and $\calO_X(K_X) = \calO_W(-4)|_X$.  A general surface $D$ such that $dK_X + 4D \sim 0$ is an element of $\calO_W(d)|_X$.  If $d$ is odd, $D$ is not a Cartier divisor.  
\end{example}

This contrasts the picture for plane curves, $(d,3)$ $\PP^2$ \he stable pairs: in \cite[Theorem 3.12]{hacking}, it is shown that $D_V$, the induced divisor on the canonical cover, is always Cartier.  In order to avoid obstructions coming from the divisor $D$, we consider higher index covers.  By considering the canonical covering of slightly higher index, we can show that studying deformations of the pair $(X,D)$ amounts to studying deformations of $X$ because the presence of the divisor $D$ does not add any further obstructions.  In fact, this can be done by taking the canonical covering $V$ corresponding to $N'K_X$, where $N'$ is the index of the $\mathbb{Q}$-divisor $\frac{1}{q}K_X$: the relationship $pK_X + qD \sim 0$ implies that $D_V$ is Cartier on $V$.  We will call such a cover a $q$-canonical cover and such a deformation a $q$-$\mathbb{Q}$-Gorenstein deformation.  By the same proof as in \cite[2.4]{hackinggentype}, these deformations are exactly the ones such that both $K_{\mathcal{X}}$ and $\mathcal{D}$ commute with base change. 

\begin{theorem}
Let $(\mathcal{X}, \mathcal{D})/A$ be a $\mathbb{Q}$-Gorenstein family \he stable pairs.  Let $A' \to A$ be an infinitesimal extension and $\mathcal{X}' \to A'$ a $q$-$\mathbb{Q}$-Gorenstein deformation of $\mathcal{X}/A$.  Then, there exists a $\mathbb{Q}$-Gorenstein deformation $(\mathcal{X}',\mathcal{D}')/A'$ of $(\mathcal{X}, \mathcal{D})/A$.
\end{theorem}

\begin{proof} Using the following lemma in place of Lemma 3.14 in the proof of Theorem 3.12 in \cite{hacking}, the same proof holds. \end{proof}

\begin{lemma}\label{lem:H1(D)}
Let $(X,D)$ be an \he stable pair.  Then, $H^1(X, \mathcal{O}_X(D)) = 0$.
\end{lemma}

\begin{proof} Recall that $D$ is an ample $\bQ$-Cartier $\bZ$-divisor.  While this is a standard consequence of Kawamata-Viehweg vanishing when $X$ is log terminal, for the general case we apply \cite[Theorem 1.7]{fujino2} to $(X,0)$, noting that $D$ and $-K_X$ are ample, so $D-K_X$ is ample.  Note that the hypotheses in \cite[Theorem 1.7]{fujino2} are satisfied because $(X,\frac{q}{p}D)$ is slc, and by definition $D$ does not contain any component of the conductor of $X$ (see \cite[Definition-Lemma 5.10]{newkollar}).
\end{proof}

We continue Example \ref{ex:Dnotcartier} and compute the $4$-canonical covering. 

\begin{example}\label{ex:4cancover}
    Let $X$ be the cone over the anticanonical embedding of $\PP^1 \times \PP^1$, a section of $\calO_W(2)$ in $W = \PP(1,1,1,1,2)$, so $\calO_X(K_X) = \calO_W(-4)|_X$.  Therefore, $\calO_X(\frac{1}{2}K_X) = \calO_X(-2)$ is Cartier, so we need to construct a $2:1$ cover to make $\frac{1}{4}K_X$ a Cartier divisor.  Alternatively, we are taking a cover to make $\calO_X(1)$ a Cartier divisor. Let $X' \subset \PP^4$ be the cone over the quadric surface in $\PP^3$, so $X'$ is the cone over the $(1,1)$-embedding of $\PP^1 \times \PP^1$.  Because $X$ is the cone over the $(2,2)$-embedding of the same surface, there is a finite morphism $X' \to X$ which is the desired cover near the singular point of $X$.  
\end{example}

In \cite[Section 3]{hacking}, the author computes the deformation and obstruction spaces for these pairs.  We can use this and Examples \ref{ex:Dnotcartier}, \ref{ex:4cancover} to show that the moduli space $\calM_{\dfours}$ is generically smooth along the locus parameterizing pairs $(X,D)$ where $X$ is the cone over the anticanonical embedding of $\PP^1 \times \PP^1$.  Later, in Proposition \ref{prop:Xdivisor}, we will show that this locus is actually a divisor in $\calM_{\dfours}$. 

\begin{proposition}\label{prop:smoothdivisor}
 Let $X$ be the cone over the anticanonical embedding of $\PP^1 \times \PP^1$.  Then, $X$ has unobstructed $4$-$\mathbb{Q}$-Gorenstein deformations.  In particular, the moduli space $\calM_{\dfours}$ is smooth at the generic point of the locus parameterizing surfaces on $X$.  
\end{proposition}

\begin{proof}
    By \cite[Remark after 3.9]{hacking}, the obstructions to extending $4$-$\mathbb{Q}$-Gorenstein deformations of $X$ are contained in sheaves $T^2_{QG,X}$ (see \cite[Notation 3.6]{hacking}).  Furthermore, there is a spectral sequence $H^p(\calT^q_{QG,X}) \Rightarrow T^{p+q}_{QG,X}$.  If, locally, $\pi: V \to X$ is the $4$-canonical cover where $X$ is the quotient by a group $G$, we can compute $\calT^q_{QG,X} = (\pi_*\calT^q_V)^G$.  Furthermore, $\calT^0_{QG,X} = \calT^0_X$, $\calT^1_V$ is supported on the singular locus of $V$, and $\calT^2_V$ is supported where $V$ is not a local complete intersection.  By Example \ref{ex:4cancover}, if $X$ is the cone over the anticanonical embedding of $\PP^1 \times \PP^1$, then $V$ is the cone over the quadric surface.  Because the singular locus of $V$ is a single point and $V$ is a hypersurface, we have $H^1(\calT^1_V) = 0$ and $H^0(\calT^2_V) = 0$.  A computation shows that $H^2(\calT^0_X) = 0$, hence $T^2_{QG,X} = 0$. 
\end{proof}

Finally, we aim to describe the moduli functor.  In the definition of $\Zpqs$ \he stable pairs, we require that $(X,D)$ has a smoothing to $(W, D_W)$, where $W$ is a smooth Fano variety deformation equivalent to $Z$.  Therefore, we are interested only in certain `smoothable' deformations of $(X,D)$, made precise below.  

\begin{definition}
Let $(X,D)/\mathbb{C}$ be a $\Zpqs$ \he stable pair.  Let $(\mathcal{X}^u, \mathcal{D}^u)/S_0$ be a versal $\mathbb{Q}$-Gorenstein deformation of $(X,D)$, where $S_0$ is finite type over $\mathbb{C}$.  Let $S_1 \subset S_0$ be the open subscheme where the fibers of $\mathcal{X}^u$ over $S_0$ are smooth and deformation equivalent to $Z$ and $S_2$ the (scheme-theoretic) closure of $S_1$ in $S_0$.  A $q$-$\mathbb{Q}$-Gorenstein deformation of $(X,D)$ is said to be smoothable if it can be obtained by pullback from the deformation $(\mathcal{X}^u, \mathcal{D}^u) \times_{S_0} S_2 \to 0 \in S_2$.
\end{definition}

\begin{remark}
In \cite{hacking}, it is shown that this condition is vacuous for $(d,3)$ $\PP^2$-smoothable \he stable pairs when $d$ is not a multiple of $3$, but non-trivial when $3$ divides $d$. 
\end{remark}

For the definition of \textit{family} of $\bQ$-Gorenstein smoothable pairs, and more technicalities on the necessary conditions for the divisor $\calD$ to be well-behaved, we refer the reader to the recent paper \cite{kflatness} of Koll\'ar.  We include the necessary definition here (see \cite[Definition 10]{kflatness}).  Note that the idea of K-flatness over a reduced base is equivalent to the usual definition of family of stable pairs \cite{modulibook}.

\begin{definition}
A family of stable pairs is a morphism $f : (X, cD) \to  S$, where
    \begin{enumerate}
        \item $f: X \to S$ is flat and projective, 
        \item $D$ is a K-flat family of divisors on $X$ (\cite[Definition 2]{kflatness})
        \item $K_X + cD$ is $\bQ$ Cartier and relatively ample, and 
        \item the fibers $(X_s, cD_s)$ are slc. 
    \end{enumerate}
\end{definition}

\begin{definition}\label{def:functor}
Let $Sch$ be the category of noetherian schemes over $\mathbb{C}$.  For a smooth Fano variety $Z$ and $(p,q)\in \mathbb{N}^2$, we define a moduli pseudofunctor over a reduced base $S$ as $\mathcal{M}_{\Zpqs} \to Sch$ by
\[ \mathcal{M}_{\Zpqs}(S) = \left\{ \begin{array}{c|c} (\mathcal{X},\mathcal{D})/S  & (\mathcal{X},\mathcal{D})/S \text{ is a q-} \mathbb{Q} \text{-Gorenstein smoothable family} \\ & \text{of $\Zpqs$ \he stable pairs} \end{array} \right\} \]
\end{definition}

As stated above, in \cite{hacking}, it is shown that the $\mathbb{Q}$-Gorenstein deformation condition is equivalent to requiring the Koll\'ar condition on families.  Using the deformation theory in \cite[Section 3]{hacking}, Theorem \ref{thm:proper}, and the finiteness of automorphism groups \cite[Corollary 10.69]{newkollar} we deduce the following theorem. 

\begin{theorem}
The moduli space of  $\Zpqs$ \he stable pairs is a proper Deligne-Mumford stack.
\end{theorem}

\begin{remark}
As mentioned in Section \ref{sec:pairs}, one could remove the condition that \he stable pairs admit a smoothing to $Z$ and define an analogous moduli functor $\calM_{\Zpq}$ of pairs, although the class of arbitrary $\Zpq$ \he stable pairs is not obviously bounded.  However, Theorem \ref{thm:boundednessP3} says that $\dfour$ \he stable pairs are bounded for odd degree $d$, and this moduli space is still proper.  Although the proof of properness was given for pairs that admit a smoothing, the same proof applies more generally.  By Hacking's work in \cite[Section 3]{hacking} and the boundedness theorem, we obtain the following. 
\end{remark}

\begin{theorem}
For odd degree $d$, the moduli space $\calM_{\dfour}$ is a proper Deligne-Mumford stack.
\end{theorem}

We will explore $\calM_{\dfours}$ and $\calM_{\dfour}$ in Sections \ref{sec:classification} and \ref{sec:quintics}.  As a final remark, note that one could define an alternative moduli functor via the work of Abramovich and Hasset \cite{stablevars}. We can consider the substack of the algebraic stack $\mathcal{K}^{\omega}_{\text{slc}}$ (cf. \cite[Section 5]{stablevars}) satisfying the locally closed condition $dK_X+4D \sim 0$ \cite[Lemma 5.8]{ypgmod}.  This condition is algebraic, so we could define a variant $\calP_{\dfours}$ of $\calM_{\dfours}$ as this substack.  It is not clear if the presence of the divisor $D$ has an effect on the structure of this stack.

\section{Classification}\label{sec:classification}

Now, we turn our attention to $\dfour$ \he stable pairs.  We begin an explicit classifcation of the threefolds appearing in this moduli space. We dedicate our attention only to the threefolds $X$, not the pair $(X,D)$, since the ample divisor $D$ must be in a linear system determined by a multiple of $K_X$.  As a starting point, we have the following result of de Fernex and Fusi that implies any log terminal degenerations are rational. 

\begin{theorem}\cite[Theorem 1.3]{defernexfusi}\label{thm:totaro}
Rationality specializes in families of complex klt varieties of dimension at most 3.
\end{theorem}

A partial classification of rational, log terminal varieties that admit a smoothing to $\mathbb{P}^3$ is discussed in Section \ref{sec:lt}.  One necessary criterion is that $(-K_X)^3 = 64$ (see below).  We point out that such a classification is known in dimension 2 (log terminal surfaces that smooth to $\PP^2$) by \cite{manetti} and will be recalled in Section \ref{sec:lt}.  

\begin{proposition}\label{prop:constantvolume}
Let $f: \mathcal{X} \to C$ be a flat family of $n$-dimensional projective varieties over a pointed curve $0\in C$.  Assume that $K_{\mathcal{X}/C}$ is $\mathbb{Q}$-Cartier, the general fiber $X_t$ is smooth, and the special fiber $X_0$ is slc.  Then, $(K_{X_0})^n = (K_{X_t})^n$.  In particular, if $X_t \cong \mathbb{P}^3$, $(K_{X_0})^3 = -64$.
\end{proposition}

\begin{proof}
Let $l$ be an integer such that $lK_{\mathcal{X}}$ is Cartier.  Then, for any $t \in C$, $\mathcal{O}_{X_t}(lK_{X_t}) \cong \omega_{X_t}^{[l]} \cong (\omega_{\mathcal{X}}^{[l]})_t$.  By definition, $(lK_{X_t})^n$ is the coefficient of $m_1m_2\dots m_n$ in $\chi(X_t, \mathcal{O}_{X_t}((m_1+m_2+\dots+m_n)lK_{X_t}) )$.  Because $f$ is flat, this polynomial is constant, so $(lK_{X_t})^n = l^nK_{X_t}^n$ is constant.  Therefore, $K_{X_t}^n$ is constant, as desired.
\end{proof}

\begin{remark}
The assumption that $K_{\mathcal{X}}$ is $\mathbb{Q}$-Cartier is essential; see \cite[Example 7.61]{km98}.
\end{remark}

It is also easy to construct a non-rational (normal) degeneration of $\mathbb{P}^3$, as shown by the following example.  Any such example is at least log canonical, in light of Theorem \ref{thm:totaro}.  

\begin{example}\label{ex:K3cone}
Given a projectively normal variety $V \subset \mathbb{P}^N$, there is a standard degeneration of $V$ to a cone over its hyperplane section \cite[7.61]{km98}.  Thus, taking the $4$-uple embedding $\mathbb{P}^3 \hookrightarrow \mathbb{P}^{34}$, the general hyperplane section of the image corresponds to a K3 surface in $\mathbb{P}^3$, which has trivial canonical divisor.  The cone $X$ over such a surface $S$ is log canonical: let $Y$ be the blow up of $X$ at the vertex, $f: Y \to X$.  Then, $f$ is birational with exceptional divisor isomorphic to $S$, so $K_Y \sim f^* K_X + a S$.  By adjunction, $K_S \sim (K_Y + S)\big|_S$, so $0 \sim K_S \sim (f^*K_X + (a+1)S)\big|_S$.  Given any curve $C \subset S$, $0 = K_S \cdot C = (a+1)S\big|_S \cdot C$, hence $a = -1$ and $X$ is log canonical.  A calculation shows that $-K_Y-S$ is nef and $0$ exactly on curves contained in the exceptional locus $S$, so $-K_X$ is ample.  However, for $X$ to occur as a threefold in a pair $(X,D)$ on the boundary of the moduli space above, we must have $-\frac{d}{4}K_X \equiv D$.  By the discussion above, $K_X\cdot C \in \mathbb{Z}$ for any curve $C \subset X$, and a calculation shows that $K_X \cdot \Gamma = -1$ for a ruling of the cone.  Since the singularity of $X$ is not klt, in order for $(X,(\frac{4}{d}+\epsilon) D)$ to also be log canonical, $D$ must miss the singularity of $X$.  Hence, $D$ is contained in the smooth locus of $X$ and is therefore Cartier, so $D\cdot C \in \mathbb{Z}$, which implies $\frac{d}{4}\in \mathbb{Z}$.  Therefore, for $d$ not divisible by 4, $X$ cannot occur as a boundary threefold.
\end{example}

From this observation and the comment on boundedness above, we first focus on the log canonical but non-klt threefolds appearing in the moduli problem.  The main result is that, for odd degree $d$, there are none.

\subsection{Non-klt Fano threefolds}\label{sec:lcthreefolds}

The inspiration for classification of the non-klt threefolds in this moduli problem is the following.

\begin{theorem} \cite{ishii}\label{thm:ishii}
If $X$ is a normal, Gorenstein variety of dimension $n$ with $K_X$ anti-ample and with finite (non-empty) irrational locus, then $X$ is a cone over a variety $S$ with canonical singularities and $K_S \sim 0$.
\end{theorem}

 Following Ishii, we will refer to the log canonical but non-klt locus as the \textit{strictly} log canonical locus.  If $X$ is a normal, Gorenstein, lc variety with $K_X$ anti-ample, the strictly log canonical locus coincides with the irrational locus \cite[Corollary 5.24]{km98}.  Therefore, this theorem implies that if a normal, Gorenstein threefold $X$ has a finite (non-empty) non-klt locus, it is either a cone over a K3 surface or two dimensional Abelian variety.  Note that the Gorenstein hypothesis implies that any klt singularity must be canonical, so the klt but non-canonical locus is empty in this case.  We provide the following generalization. 

\begin{theorem}\label{thm:conelogcanonical}
Let $X$ be a log canonical projective variety with a finite number of strictly log canonical singularities $\{p_1, \dots, p_n  \}$ and $-K_X$ ample.  If $a(E,X) \in \{-1, \mathbb{R}^{\ge 0} \}$ for every exceptional divisor $E$ over $X$ with $\text{center}_X(E)\subset \{p_1, \dots , p_n  \}$, then $X$ is a cone over a variety $Z$ with $K_Z \equiv 0$.
\end{theorem}

The extra hypotheses in this result arise from removing the Gorenstein hypotheses in Theorem \ref{thm:ishii}.  In order to ensure $X$ is a cone, there needs to be a certain extremal ray in the cone of curves.  Before getting to the proof, we provide a few definitions and technical lemmas.  In all cases, we consider dlt pairs $(X,D)$ and study properties of various $K_X$-negative and $K_X + D$-negative contractions.  The motivating idea is to study contractions that happen `over' $D$: divisorial contractions that are $K_X + D$-negative and $D$-positive must have a certain structure.

\begin{definition}
	Given a proper variety $X$, the effective cone $NE(X)$ is the collection of effective 1 cycles on $X$ modulo numerical equivalence, and its closure is denoted $\overline{NE}(X)$.  If $R$ is an extremal ray in $\overline{NE}(X)$, we say that the contraction of $R$ is an elementary extremal contraction.  In what follows, we will refer to the contraction of $R$ as simply an extremal contraction and always mean the contraction of an extremal ray.  
\end{definition}

 We begin with the negativity of $K_X$ in certain $K_X$-negative contractions.  Namely, the next lemma shows that $K_X$ cannot be `too' negative on fibers. 

\begin{lemma}\label{lem:boundonK}
	Let $X$ be a normal, log terminal projective variety such that $K_X$ is $\mathbb{Q}$-Cartier.  If $\phi: X \to Y$ is a birational contraction of a $K_X$-negative extremal ray with fibers of dimension at most 1, then each fiber $F$ is a chain of $\mathbb{P}^1$s whose configuration is a tree such that $-1 \le K_X \cdot C < 0$ for each irreducible component $C$ of $F$.
\end{lemma}

\begin{proof}
	By assumption, $R^2 \phi_* \mathcal{F} = 0$ for any coherent sheaf $\mathcal{F}$ on $X$.  By Grauert-Riemenschneider vanishing \cite[Corollary 10.38]{newkollar}, $R^1 \phi_* \omega_X = 0$, and by
	\cite[Theorem 1-2-5]{KMM}, because $-K_X$ is $\phi$-ample, $R^1 \phi_* \mathcal{O}_X = 0$.
	Then, consider any sheaf of ideals $J$ such that $\mathcal{O}_X / J$  is supported on a fiber $F$ of $\phi$:
	\[ 0 \to J \to \mathcal{O}_X \to \mathcal{O}_X / J \to 0 \]
	Pushing forward to $Y$, we see that $R^1\phi_* (\mathcal{O}_X / J) = H^1(F, \mathcal{O}_X / J |_F) = 0$.  Similarly, we see that $H^1(F, \omega_X / J\omega_X |_F) = 0$ so $H^1(F, (\omega_X / J\omega_X)|_F /T) = 0$, where $T \subset (\omega_X / J\omega_X)|_F$ denotes torsion.  Taking $J$ to be the ideal of $F$, we see that $F$ is a chain of $\mathbb{P}^1$s whose configuration is a tree.
	Then, consider an irreducible component $C \subset F$ and the sheaf $(\omega_X \otimes \mathcal{O}_C) / T$, where $T$ is the torsion in $\omega_X \otimes \mathcal{O}_C $.  This is a torsion-free sheaf on $\mathbb{P}^1$, so must be a vector bundle of the form $(\omega_X \otimes \mathcal{O}_C) / T \cong \oplus \mathcal{O}_{\mathbb{P}^1}(a_i)$.  The vanishing of $H^1$ given above implies that $a_i \ge -1$ for each $i$.  If $m$ is an integer such that $\omega_X^{[m]}$ is Cartier, we must have that $\omega_X^{[m]}\otimes \mathcal{O}_C = \mathcal{O}_{\mathbb{P}^1}(b)$ is a negative degree line bundle.  But, there is a nonzero morphism from taking the double dual of $\omega_X$:  $$(\omega_X \otimes \mathcal{O}_C / T)^{\otimes m} \to \omega_X^{[m]} \otimes \mathcal{O}_C$$ and $a_i \ge -1$ implies that $b \ge -1$.  Therefore, $-1 \le K_X \cdot C < 0$.
\end{proof}

The previous lemma bounds the negativity of $K_X$.  If curves $C$ are contained in the smooth locus, because $K_X \cdot C \ge -1$ for contracted curves, if $C \cap D \ne \emptyset$, that should force $(K_X + D) \cdot C \ge 0$.  Certainly this could be false if $X$ was highly singular and $D \cdot C \notin \Z$, but with a few restrictions on the singularities, we can apply the lemma to our advantage. 

\begin{lemma}\label{lem:divisorialcont}
	If $(X,D)$ is dlt, $X$ is $\bQ$-factorial, and $D$ is an effective prime $\bZ$-divisor that is Cartier in codimension $2$, then any $K_X+D$-negative extremal divisorial contraction is an isomorphism on $D$ if and only if the exceptional divisor does not intersect $D$.
\end{lemma}

\begin{proof}
Let $\phi: X \to Y$ be the given contraction.  Because $\phi$ is $K_X + D$ negative and divisorial, the negativity lemma implies that
\[ \phi^* (K_Y+D_Y) = K_X + D - aE\]
where $D_Y = \phi_*D$, $E$ is the (irreducible) exceptional divisor, and $a > 0$.  Because $D$ is Cartier in codimension 2 and normal by \cite[Theorem 4.16]{newkollar}, $K_X + D |_D = K_D$.  By \cite[Corollary 3.44]{km98} and \cite[Theorem 4.16]{newkollar}, $D_Y$ is also normal, so restricting to $D$ (where, by abuse of notation, we still denote the by $\phi$ induced map $D \to D_Y$):
\[ \phi^*(K_{D_Y} + \Diff_{D_Y}(0)) = K_D - aE|_D  \]
where $\Diff_{D_Y}(0)$ is the correction term to the adjunction formula.  This correction term is effective by \cite[Proposition 4.5]{newkollar}, and $aE|_D$ is effective, so $\phi|_D: D \to D_Y$ is an isomorphism if and only if $E|_D = 0$.  This also shows that $\Diff_{D_Y}(0) = 0$.
\end{proof}

From this observation and Lemma \ref{lem:boundonK}, we see that given a `nice' contraction that is an isomorphism on $D$, the map is forced to be a fibration.  However, one should be cautious; this lemma (and the corollaries) are false without the hypothesis that $D$ is Cartier in codimension 2.  

\begin{example}
    Let $X = \PP^2$ and let $\pi: Y \to X$ be the $(n,1)$ weighted blow up of the point $(0,0)$ in linear coordinates $(x/z, y/z)$ for any $n > 1$.  Let $L = (y = 0)$ be a line in $\PP^2$ and let $L_Y$ be the strict transform.  Denote the exceptional divisor of $\pi$ by $E$ and note that $E^2 = -\frac{1}{n}$.  By construction, $L_Y$ and $E$ intersect at the unique $\frac{1}{n}(1, n-1)$ singularity of $Y$ and are not Cartier at that point.  We can compute $$\pi^*K_X = K_Y - nE$$ and $$ \pi^*L = L_Y + E$$ so that $K_Y \cdot E = -1$ and $L_Y \cdot E = \frac{1}{n}$.  Then,  $$\pi^*(K_X + L) = K_Y + L_Y - (n - 1) E$$ so the contraction $\pi: Y \to X$ of $E$ is $K_Y + L_Y$-negative and is an isomorphism on $L_Y$, but $L_Y \cap E \ne \emptyset$.  
\end{example}

If $D$ is Cartier in codimension 2, however, we avoid the behavior in the previous example.  

\begin{corollary}\label{cor:primeD}
	If $(X,D)$ is dlt and $D$ is an effective, prime divisor that is Cartier in codimension $2$, then any $K_X+D$-negative, $D$-positive extremal contraction that contracts a divisor but contracts no curves in $D$ is a Fano fiber contraction $X \to D$.
\end{corollary}

\begin{proof}
Let $\pi: X \to Y$ be the contraction.  If a divisor is contracted, then the morphism is either a divisorial contraction or Fano fiber contraction onto a variety with strictly lower dimension.  If no curves in $D$ are contracted, the induced map $D \to \pi(D)$ is finite, but $(Y, \pi_*D)$ is dlt by \cite[Corollary 3.44]{km98} and $D$ is a prime divisor, hence $\pi_*D$ is normal by \cite[Theorem 4.16]{newkollar}.  Furthermore, because no curves in $D$ are contracted, the fibers have dimension at most $1$.  Indeed, let $F$ be a fiber of $\pi$.  Because $\pi$ contracts no curves in $D$, $F \cap D$ must be a collection of points.  However, $\pi$ is a $D$-positive contraction, so $D|_F$ is ample, hence must be a divisor on $F$.  Therefore, $F$ has dimension at most 1.  

But, if $\pi$ is divisorial, Lemma \ref{lem:boundonK} implies $K_X \cdot C \ge -1$ for any irreducible curve $C$ contracted by $\pi$.  Suppose that $C$ is an irreducible component of a general fiber of dimension one.  By assumption, $D\cdot C > 0$, and because $D$ is Cartier in codimension 2, $D$ is Cartier when restricted to $C$, so $D \cdot C \in \mathbb{Z}$.  Therefore, $(K_X + D) \cdot C \ge 0$, a contradiction.  Thus, the contraction cannot be birational and must be a fibration with general fiber $\mathbb{P}^1$.  In this case, for general fiber $C$, $K_X \cdot C = -2$, so we must have $D \cdot C = 1$, so $\pi|_D: D \to \pi(D)$ is generically of degree $1$.  Therefore, by Zariski's Main Theorem, and because $D$ is prime, $\pi_*D$ must be isomorphic to $D$ and $\pi: X \to Y$ is a Fano fiber contraction and $Y \cong D$. 
\end{proof}

\begin{corollary}
	If $X$ is a variety with terminal singularities and $(X,D)$ is dlt for some effective prime divisor $D$ where $-D|_D$ is nef, then any $K_X+D$-negative $D$-positive contraction gives a Fano fibration $X \to D$.
\end{corollary}

\begin{proof}
If $X$ is terminal, the singular set has codimension at least 3 in $X$ (\cite[Corollary 2.30]{newkollar}), hence $D$ is Cartier in codimension 2.  If $-D|_D$ is nef, then any $D$-positive contraction contracts no curves in $D$, so by Corollary \ref{cor:primeD}, the contraction of such a ray gives a Fano fibration $X \to D$.
\end{proof}

We should point out that Lemma \ref{lem:boundonK} does not require the contraction be divisorial; it could be a small contraction and the result would still hold.  In particular, the next lemma shows that $K_X+D$-negative and $D$-positive small contractions cannot exist with certain assumptions on the singularities of $(X,D)$.  

\begin{lemma}\label{lem:smallcont}
If $X$ has terminal, $\bQ$-factorial singularities and $(X,D)$ is a pair with log terminal singularities such that $D$ is an effective prime divisor, then the contraction of a $K_X + D$-negative, $D$-positive extremal ray $R$ that contracts no curves in $D$ cannot be a small contraction.
\end{lemma}

\begin{proof}

Assume such a small contraction exists.  Because this is a $K_X$-negative contraction, we consider the flip of $\phi$ as in the following diagram, where $Z$ is a log resolution of the rational map $X \dashrightarrow X^+$ (so both $Z$ and the strict transform of $D$ are smooth).  The flip exists by \cite[Corollary 1.4.1]{BCHM}.

\begin{center}

\begin{tikzcd}
    & Z \arrow{ld}[swap]{\pi} \arrow{rd}{\pi^+} \\
	X \arrow{rd}[below]{\phi \quad}
	\arrow[dashed]{rr}
	&& X^+ \arrow{ld}{\phi^+} \\
	& Y
\end{tikzcd}

\end{center}

Note that the fiber of the contraction $\phi: X \to Y$ is not contained in $D$, by assumption.  Because every $\pi$ exceptional divisor $E$ has nonnegative discrepancy $a(E,X,D)$, if $D_Z = \pi^{-1}_* D$, we have
\[K_Z + D_Z = \pi^*(K_X + D) + \sum a_i E_i \]
where $a_i > -1$ for each $i$, and $a_i \ge 0$ for any exceptional divisor $E_i$ such that $\text{center}_X(E_i) \not \subset D$.  Restricting to $D$, because $D$ is Cartier in codimension 2, we get
\[ K_{D_Z} = \pi|_D^*(K_D) + \sum a_i E_i|_{D_Z} .\]
  But, by \cite[Lemma 3.38]{km98}, flips can only improve singularities, so
\[{\pi^+}^*(K_{X^+}+D^+) = \pi^*(K_X + D) - \sum c_iE_i  \]
where $c_i \ge 0$.  Because the flip was $K_X$-negative, by the same lemma, $X^+$ is also terminal, so $D^+$ is Cartier in codimension 2.  Then, restricting to $D$ and $D^+$ we see that
\[ {\pi^+}|_{D^+}^*(K_{D^+}) = \pi|_D^*(K_D) - \sum c_iE_i|_{D_Z} .\]
Substituting, we see that
\begin{equation}\label{eq1}
    {\pi^+}|_{D^+}^*(K_{D^+}) = K_{D_Z} - \sum a_i E_i|_{D_Z} - \sum c_iE_i|_{D_Z} .
\end{equation}

Because no curves in $D$ are contracted, the map $D \to \phi(D)$ is finite, and because $\phi$ is a small contraction, the map has degree 1, so either $D \cong \phi(D)$ or $D$ is its normalization.  In either case, because $X^+$ is dlt, $D^+$ is normal by \cite[Theorem 4.16]{newkollar}, so there is a morphism $f:D^+ \to D$ making the diagram

\begin{center}
\begin{tikzcd}
    & D_Z \arrow{ld}[swap]{\pi|_{D_Z}} \arrow{rd}{\pi^+|_{D_Z}} \\
	D \arrow{rd}[below]{\phi|_D \quad}
	&& D^+ \arrow{ld}{\phi^+|_D} \arrow{ll}[above]{f}  \\
	& D_Y := \pi(D)
\end{tikzcd}
\end{center}

\noindent commute.  Observe that $f$ cannot be an isomorphism: if $D^+ \cong D$, then $f^*(K_D) = K_{D^+}$, and hence $\pi|_D^*(K_D) = \pi^+|_{D^+}^*(K_{D^+})$.  But, this implies that $c_i = 0$ for all $i$.  By \cite[Lemma 3.38]{km98}, this implies that, for any exceptional divisor $E$ over $X$, $\phi \circ \pi = \phi^+ \circ \pi^+$ is an isomorphism over the generic point of $\mathrm{center}_Y E$.  This implies that neither $\pi$ nor $\pi^+$ extract any divisors, so by normality of $X, Z,$ and $X^+$,  and $\bQ$-factoriality of $X$ and $X^+$, $X \cong X^+$, a contradiction.  

Therefore, there is some exceptional divisor $E_0$ such $E_0|_{D_Z}$ is not contracted by $\pi^+|_{D_Z}$ but is contracted by $\pi|_{D_Z}$.  In this case, we must have $c_0 > 0$.  If we can show that $\text{center}_X(E_0) \not \subset D$, this is a contradiction because it would imply the coefficient of $E_0|_{D_Z}$ in Equation \ref{eq1} is nonzero, but $E_0|_{D_Z}$ is not an exceptional divisor of $\pi^+|_{D_Z}$.  

To conclude the proof, assume for contradiction that $\text{center}_X(E_0) \subset D$ for any such divisor $E_0$.  Then, for any exceptional divisor $E_0$ contracted to a curve in $D^+$, it must be contracted to a point in $D$, as the small contraction did not contract any curves in $D$.  By definition of $Z$, there must exist a divisor $E_1$ whose image is a one-dimensional component of $C$, i.e. $\text{center}_X{E_1} \not \subset D$.  By assumption, for any such divisor $E_1$, $\pi^+(E_1) \cap D^+$ must be finite.  However, let us compute discrepancies.  Because $\pi(E_1)$ is a curve not contained in $D$ and $X$ has only isolated singularities, $K_Z = \pi^*(K_X) + b_1E_1 + \text{ other terms }$ where $b_1 \in \mathbb{Z}$, $b_1 > 0$.  And, $\pi^*(D) = D_Z + 0 \cdot E_1 + \text{ other terms }$.  So, the discrepancy $a_1$ of $E_1$ in the map $\pi$ is $a_1 = b_1 \in \mathbb{Z} > 0$.  However, on $D^+$, we then obtain that $K_{D_Z} = \pi^+|_{D^+}^*(K_{D^+}) + (1+c_1)E_1 + \text{ other terms } ,$  where the other terms by assumption correspond to divisors with positive discrepancy in the same fashion as $E_1$.  All discrepancies in the extraction $D_Z$ to $D^+$ are thus at least $1$.  By assumption, $D_Z$ was smooth, so $D^+$ is therefore smooth.  As $D_Z \to D^+$ can be factored as blow ups of smooth points \cite[Chapter 5.5]{hartshorne}, the discrepancy of some $E_i$ (without loss of generality, assume it is $E_1$) must be $1$, and $c_1 = 0$.  But this contradicts the fact that $\phi: X \to Y$ is not an isomorphism above the generic point of $\text{center}_Y(E_1)$, as $E_1$ is contracted to a curve via $\pi$ \cite[Lemma 3.38]{km98}.  Therefore, we have reached a contradiction, and thus there exists an $E_0$ such that $\text{center}_X(E_0) \not \subset D$.
\end{proof}

We can tie the previous lemmas together in the following result, seemingly technical but the key ingredient in the proof of Theorem \ref{thm:conelogcanonical}.

\begin{lemma}
	Let $X$ be a variety with terminal singularities and $(X,D)$ a pair with canonical singularities with $D$ an effective prime divisor such that $K_X|_D$ is nef.  If the class of a $K_X+D$-negative extremal ray $R$ contains a curve $C$ such that $C\cap D$ is finite and non-empty (hence, $R$ is $D$-positive), and the contraction of $R$ has fiber dimension at most 1, then it must be a Fano fibration $X\to Y$ such that the general fiber is isomorphic to $\mathbb{P}^1$ and $Y \cong D$.
\end{lemma}

\begin{proof}
Because varieties with terminal singularities are singular only in codimension $\ge 3$, $D$ is Cartier in codimension 2.  By hypothesis, any curve $C \subset D$ has $K_X \cdot C \ge 0$, hence the contraction $\phi: X \to Y$ of a $K_X + D$-negative $D$-positive extremal ray cannot contract any curves in $D$.  By Lemma \ref{lem:smallcont}, the contraction of $R$ cannot be a small contraction, so must contract a divisor.  Applying Corollary \ref{cor:primeD}, we obtain that $X\to Y$ is a Fano fibration and $Y \cong D$. 
\end{proof}

We are now ready to prove Theorem \ref{thm:conelogcanonical}.

\begin{proof}
By a result of Hacon \cite[Theorem 3.1]{dubois}, there is a minimal $\mathbb{Q}$-factorial dlt model of $Y \to X$.  More precisely, there is a $\mathbb{Q}$-factorial variety $Y$ and a morphism $\pi: Y \to X$ extracting all divisors $E_i$ with discrepancy $a(E_i, X) = -1$ such that $K_Y$ is relatively nef.  Let $E = \sum E_i$ and observe that $K_Y + E = \pi^*K_X$.  Because $-K_X$ is ample and $K_Y + E$ is numerically trivial on $E$ and negative on all curves not contained in $E$, there must exist a $K_Y + E$ negative, $E$ positive extremal ray $R$ in $\overline{NE}(Y)$.  In fact, because $E \cdot R > 0$, there must be some component $E_0 \subset E$ such that $E_0 \cdot R >0$ and $(K_Y + E_0) \cdot R < 0$.  Indeed, if $(K_Y + E_0) \cdot R \ge 0$, for any curve $[C] \in R$, we would have $C \subset E$, so $(K_Y+E) \cdot C = 0$, a contradiction. 

Let $\phi: Y \to S$ be the contraction of $R$.  By assumption, the pair $(Y,E)$ is canonical along the components of $E$ (since $a(F,Y,E) = a(F,X)$ for any exceptional divisor $F$ over $X$).  By Corollary \ref{cor:primeD} applied to $(Y, E_0)$, $\phi: Y \to S$ is a fiber contraction of relative dimension 1 and $S \cong E_0$.  Note that, for a general fiber $l$ of $\phi$, $K_Y \cdot l = -2$ and $[l] \in R$, so $(K_Y + E) \cdot l < 0$.  Choosing an appropriate fiber $l$ that misses the singular points of $Y$, one sees that $E_i \cdot l \in \mathbb{Z}$ for each $i$ because $l$ is contained in the smooth locus of $Y$, and $E_i$ cannot be contracted by $\phi$ as curves in $E_i$ are $K_Y+E$-trivial.  Therefore, because $K_Y\cdot l = -2$ and $(K_Y + E) \cdot l < 0$, there is only one exceptional divisor $E_0 = E$. Because $(Y,E)$ is dlt, $E$ is normal and $\phi$ contracts no curves in $E$, hence $S \cong E$, giving $\phi: Y \to S$ the structure of a $\mathbb{P}^1$ bundle.  However, as $E$ is contractible by $\pi: Y \to X$, we see that $X$ is a cone over $E$ (where `cone' is only defined as the contraction of a section of a $\mathbb{P}^1$-bundle over $E$ to a point).  We can further characterize $E$ by observing that $(K_Y + E)|_E = K_E$, hence $K_E$ is numerically trivial.  
\end{proof}

Since one cannot guarantee that the exceptional divisors over a variety are in the set given in Theorem \ref{thm:conelogcanonical}, we first make an easy observation, that follows from the proof of Theorem \ref{thm:conelogcanonical}.

\begin{proposition}\label{prop:Knegativeray}
Let $X$ be a log canonical projective variety with a finite number of strictly log canonical singularities $\{p_1, \dots, p_n  \}$ and $-K_X$ ample.  Consider a minimal dlt modification $\pi: Y \to X$ extracting the $-1$ divisors of $X$, so $K_Y + E = \pi^*(K_X)$.  If there exists an extremal ray $R \in \overline{NE}(Y)$ such that a curve $C \not\subset E$, $[C]\in R$, intersects $E$ at a smooth point of $Y$, then $X$ is a cone over a numerically Calabi-Yau variety.
\end{proposition}

To remove the restrictions on the discrepancies in Theorem \ref{thm:conelogcanonical}, we would like to say there always exists a ray as in Proposition \ref{prop:Knegativeray}.  However, it is not obvious why this is true or even clear that it should be true.  Instead, we proceed to use the ideas in Theorem \ref{thm:conelogcanonical} to study the given moduli problem.  Note first that many standard examples of log canonical singularities have resolutions where an exceptional divisor is not rational or ruled.  If that is the case, the following result characterizes these singularities. 

\begin{theorem}\label{thm:rationalruled}
If $X$ is a projective log canonical threefold with a finite number of strictly log canonical singularities and $-K_X$ ample such that at least one exceptional divisor $E$ over $X$ with discrepancy $a(E,X) = -1$ is not rational or birationally ruled, then there is only one such $E$ and $X$ is birational to a $\mathbb{P}^1$ bundle over $E$.
\end{theorem}

\begin{proof}
We proceed in a similar fashion as in the previous proof.  By \cite[Theorem 1.33, Corollary 1.37]{newkollar}, there is a \textit{terminal model of} $X$; a $\mathbb{Q}$-factorial variety $Y$ and a morphism $\pi: Y\to X$ extracting divisors $\Delta_i$ with discrepancy $a(\Delta_i,X) \le 0$ such that $Y$ is terminal and $K_Y$ is relatively nef.  Let $E = \sum \Delta_j$ be the sum over divisors $\Delta_j$ with discrepancy $-1$ and $F = \sum -a(\Delta_k,X)\Delta_k$ be the sum over divisors with discrepancy larger than $-1$.  By construction of $Y$ (which is terminal, hence has finitely many singular points), these effective divisors are Cartier in codimension 2, $\pi^*(K_X) = K_Y + E + F$, and for any curves $C\subset \Supp (E+F)$ contracted by $\pi$, $K_Y\cdot C \ge 0$.  By assumption on $X$, the general curve through $E$ has negative $K_X$-degree. 

We would like to find an $E$ positive and $K_Y + E$ negative extremal ray in the cone of curves.  If so, we proceed exactly as in the proof of Theorem \ref{thm:conelogcanonical} to conclude that the contraction of such a ray gives a Fano fibration $\phi: Y \to E$ (and $E$ consists of only one component, necessarily not rational nor ruled by assumption).  Then, we conclude as in Theorem \ref{thm:conelogcanonical} that the log canonical locus in $X$ consists of a single point $x\in X$, $E$ is a single component, and $X$ is birational to a $\bP^1$ bundle over $E$. 

In more generality, run a $K_Y$-MMP (contracting $K_Y$ negative extremal rays) on $Y$.  If at any point in the MMP we reach an intermediate variety $Y'$ with a $K_{Y'} + E'$-negative $E'$-positive extremal ray $R$ in $\overline{NE}(Y')$, the arguments in the previous paragraph imply that the MMP terminates in a fibration of dimension 1 over $E'$.  As we will see in the arguments below, the only components that could have been contracted in running the MMP up to this point are rational or ruled, so we obtain the desired result.  

Finally, assume that we do not find such a ray in the course of the MMP.  Because $X$ was a Fano threefold, the MMP must terminate with a Fano fibration $f: Y' \to S$ such that $\dim S < 3$.  Because $Y$ has terminal singularities and at each stage of the MMP as we are contracting $K_Y$-negative extremal rays, by \cite[Corollary 3.43(2)]{km98}, $Y'$ also has terminal singularities.  We claim that the only components of $E$ that could be contracted by an MMP are rational or ruled.  

In running the $K_Y$-MMP, first consider the case that in one of the steps one obtains a divisorial contraction $\phi: Y' \to Y''$ of a component $\Delta$ of $E'$.  Because $(Y,E)$ is at worst log canonical (and hence $(Y', E')$ is at worst log canonical), $\Delta$ is a log canonical surface, contracted by $\phi$ to a point or curve.  If the divisorial contraction $\phi: Y' \to Y''$ contracted $\Delta$ to a curve, the general fiber is $\bP^1$, so it is birationally ruled.  Thus, in this case, the result follows.  If instead $\Delta$ is contracted to a point, $\Delta$ is Fano, and a normal log canonical surface, so only singular at points.  By \cite[Theorem 17.4]{kollarplus}, the locus of strictly log canonical non-klt singularities is connected and hence a single point.  Therefore, \cite[Theorem 1.2]{RC} implies that $\Delta$ is rationally chain connected.  Finally, by \cite[Proposition 3.3.4]{kollarrationalcurves}, this implies that $\Delta$ must be uniruled, and any uniruled surface is birationally ruled by \cite[Exercise 1.1.6.2]{kollarrationalcurves}, so the result follows.

Now, suppose no component of $E'$ is contracted until the termination of the MMP $f': Y' \to S$.  If $Y'$ is a terminal Fano variety of Picard rank 1, because $K_{Y} + E = \pi^*(K_X)$ was negative on the generic curve in $Y$, we must have $K_{Y'} + E'$ negative.  Hence, for any component $\Delta$ of $E'$, $K_{Y'} + \Delta$ is negative and $\Delta$ is a log canonical Fano surface, hence rationally connected, hence rational or birationally ruled by the same argument above.  

If $Y'$ has Picard rank 2 and $S$ is a curve, if $\Delta$ is a fiber of $f'$, because $Y'$ is terminal, $\Delta$ must be log terminal and Fano, and hence rational.  If instead $f'|_{\Delta}: \Delta \to S$ is surjective, $\Delta$ is birationally ruled.  Finally, if $\dim S = 2$, $\dim f'(\Delta) = 0$ implies $\Delta$ is Fano and therefore rational or birationally ruled by the arguments above.  If $\dim f'(\Delta) = 1$, $\Delta$ is again birationally ruled.

Therefore, if there exists a non-rational or birationally ruled component $\Delta$ of $E'$, it cannot be contracted in any way described above, and thus we must be in the final remaining case: $f': Y' \to S$ is a Fano fibration and $\dim f'(\Delta) = 2$.  Because $f'$ is the contraction of a $K_{Y'}$-negative ray (and necessarily $K_{Y'} + E'$-negative, by consideration of the general fiber $l$, whose transform in $Y$ must be $K_Y+E$-negative), we have $K_{Y'}\cdot l = -2$ and hence must have $E' \cdot l = \Delta\cdot l = 1$, so $\Delta$ is birational to $S$.  In particular, we are in the case of the Fano fibration over a surface birational to $\Delta$.  By the computation $E' \cdot l = \Delta\cdot l = 1$, we see that $\Delta$ is the only such component that can map birationally to $S$, which implies the result. 
\end{proof}

To relate this to the moduli problem, exactly as in Example \ref{ex:K3cone}, we can show that varieties satisfying the hypotheses of Theorem \ref{thm:conelogcanonical}, Proposition \ref{prop:Knegativeray} or Theorem \ref{thm:rationalruled} do not appear as $\dfour$ stable pairs.  

\begin{corollary}\label{cor:easyeven}
If $X$ is a threefold satisfying the hypotheses of Theorem \ref{thm:conelogcanonical}, Proposition \ref{prop:Knegativeray}, or Theorem \ref{thm:rationalruled}, and $D$ is a $\bQ$-Cartier $\bZ$ divisor on $X$ that does not contain the locus of strictly log canonical singularities such that $dK_X + 4D \sim 0$, then $d$ is even.  
\end{corollary}

\begin{proof}
    By Theorem \ref{thm:conelogcanonical}, Proposition \ref{prop:Knegativeray}, or Theorem \ref{thm:rationalruled}, there is a birational model $Y \to X$ such that $\pi^*K_X = K_Y + E + F$ and the minimal model program on $Y$ terminates in a Fano fibration $Y' \to S$, where $S$ is a surface birational to $E_0$, a component of $E$.   
    
    If $X$ is as in Theorem \ref{thm:conelogcanonical} or Proposition \ref{prop:Knegativeray}, from their proofs, we have $Y' = Y$.  Consider a general fiber $l \subset Y$: we have $K_{Y} \cdot l = -2$, and $E \cdot l = 1$, so  $pi^*(K_X \cdot) l = (K_Y+E)\cdot l = -1$.  Hence, if $\overline{l}$ is the image of $l$ on $X$, $K_X \cdot \overline{l} = -1$.  But, $D$ does not contain the strictly log canonical locus of $X$, so choosing $l$ sufficiently generally implies that $D \cdot \overline{l} = n \in \bZ$, hence the relationship $dK_X + 4D \sim 0$ implies $d(-1) + 4n = 0$, so $d$ is even. 
    
    In the next case, suppose $X$ is as in Theorem \ref{thm:rationalruled}.  Following the notation and proof, we obtain that the MMP results in a Fano fibration $f': Y' \to S$, where $S$ is birational to the unique non-rational or ruled component of $E'$ (and this component is generically a section of $f'$).  Now, consider $F'$, the image of the divisors $F$ in $Y$ with discrepancy $> -1$.  By \cite[Theorem 1.2]{RC}, these divisors are all rational or ruled.  Hence, they cannot be birational to $S$, so for each component $\Delta_F \subset \Supp F'$, we have $\dim f'(\Delta_F) < 2$.  Therefore, choosing a sufficiently general fiber $l$ of $f'$, $l$ does not intersect any component of $\Supp F'$ and $(K_{Y'} + E') \cdot l = -1$.  Therefore, if $l'$ is the pre-image of $l$ in $Y$, $\pi^* K_X \cdot l' = (K_Y+E+F)\cdot l' = -1$, and if $\overline{l}$ is the image of $l'$ on $X$, $K_X \cdot \overline{l} = -1$.  But, $D$ does not contain the strictly log canonical locus of $X$, so choosing $l$ sufficiently generally implies that $D \cdot \overline{l} = n \in \bZ$, hence the relationship $dK_X + 4D \sim 0$ implies $d(-1) + 4n = 0$, so $d$ is even. 
\end{proof}

With great care and analysis of the MMP, we can obtain the same result for strictly log canonical (non-klt) Fano threefolds in more generality. 

\begin{theorem}\label{thm:even}
If $X$ is a strictly log canonical Cohen-Macaulay threefold such that $-K_X$ is ample, $dK_X + 4D \sim 0$ for some $\mathbb{Q}$-Cartier $\bZ$-divisor $D$, and $D$ does not contain the locus of strictly log canonical singularities, then $d$ is even.
\end{theorem}

\begin{proof}

By \cite[Theorem 1.33]{newkollar}, there is a $\bQ$-factorial variety $Y$ and a morphism $\pi: Y \to X$ extracting divisors $\Delta_j$ with discrepancy $a(\Delta_i,X) \le 0$ such that $Y$ is terminal and $K_Y$ is relatively nef.  Let $E = \sum \Delta_j$ be the sum over divisors $\Delta_j$ with discrepancy $-1$ and $F = \sum -a(\Delta_k,X)\Delta_k$ be the sum over divisors with discrepancy larger than $-1$.  By construction of $Y$ (which is terminal, hence has finitely many singular points), these effective divisors are Cartier in codimension 2, $\pi^*(K_X) = K_Y + E + F$, and for any curves $C\subset \Supp (E+F)$ contracted by $\pi$, $K_Y\cdot C \ge 0$.  

We will run a $K_Y$-MMP to obtain the result.  The proof is somewhat technical, so we first provide a sketch, pointing out the necessary technicalities.  If there is a component of $E_0$ that is not contracted by the MMP, the MMP terminates in a fibration $Y' \to S$ and one can show that $S$ is birational to $E_0$.  Now, to conclude that $d$ is even, we wish to perform a similar intersection theoretic computation to that above, finding a sufficiently generic fiber $l$ of $Y' \to S$ such that its preimage $l'$ in $Y$ satisfies $(K_Y+E_0)\cdot l' = -1$.  However, to conclude that $\pi^*(K_X)\cdot l' = -1$, we must show that either there do not exist other components of $\Supp E$ or $\Supp F$ that map finitely to $S$ or we can use the other components to conclude that $d$ is even.

If instead all components of $E$ are contracted by the MMP, we will prove the necessary existence of certain components of $\Supp F $ intersecting $E$ so that the curves in $E$ (which, initially may be $K_Y$-positive) become $K_Y$-negative and thus contractible.  Then, we make the following observation: because $dK_X + 4D \sim 0$ and the strictly log canonical (non-klt) locus of $X$ is not contained in $D$, for any component $\Delta$ of $\Supp F$ that maps into the non-klt locus of $X$, the discrepancy of $\Delta$ must be of the form $\frac{a}{d}$ for some $a \in \mathbb{Z}$.  Furthermore, on curves contained in components of $\Delta$ of $E$ contracted by $\pi$, we have $(K_Y + E + F)|_{\Delta} = 0$, so if $C_E = (E-\Delta)|_\Delta$ and $C_F = F|_\Delta$, we know $K_{\Delta} + C_E + C_F = 0$.  If $\Delta$ is sufficiently nice, we will use this relationship to show that either $d$ must be even or the discrepancies of the components of $F$ must be in a finite collection of exceptional possibilities, and we will rule out these possibilities by hand.  

We begin in the simpler case, when there exists a component $E_0 \subset E$ that is not contracted by the MMP. 

\begin{lemma}\label{lem:enotcontracted}
    If there is a component $E_0 \subset E$ that is not contracted by the MMP (including the final step as a Fano fibration), then $d$ is even.  
\end{lemma}

    \begin{proof}
        In this case, the MMP must terminate with a morphism $Y \dashrightarrow Y' \to S$, and $\dim S = 2$.  Exactly as in the proof of Theorem \ref{thm:conelogcanonical}, we obtain that $S$ is birational to $E_0'$, where $E_0'$ is the image of $E_0$ in $Y'$, and by negativity of $K_Y + E = \pi^*K_X$, the general fiber $l$ of $Y' \to S$ cannot intersect other components of $E$.  We are done by the same argument as in Corollary \ref{cor:easyeven} if $l$ intersects no components of $\Supp F$.  However, a priori, $l$ could intersect components of $\Supp F$: suppose $\Delta_0' \to S$ is generically finite, where $\Delta_0'$ is the image of some $\Delta_0 \subset \Supp F$.  Because $\Delta_0$ is contracted by $\pi$, let $C \subset S$ be the image of a general fiber of $\pi|_{\Delta_0}$, noting that $C$ must be rational as $\Delta_0$ contracts to a log terminal singularity (\cite[Theorem 1.2]{RC}).  Because $Y'$ is terminal, it has isolated singularities, so choosing $C$ sufficiently generally, we may assume that the pre-image of $C$ is a smooth ruled surface $P \to C$ with multi-section $\Delta_0'|_P$.  As $E_0'$ is a section of $P$, and all fibers of $P$ are numerically equivalent with respect to $E_0'$, we see that there cannot be any singular fibers (using that $P$ is smooth).  However, because $\Delta_0'|_P$ is a contractible multisection in the smooth ruled surface $P$, it must in fact be a section.  As there is only one contractible section in $P$, we see that $E_0'|_P$ cannot be contractible.  Similar analysis implies, for any other component $\Delta_1$ of $\Supp F$ mapping finitely to $S$, $\Delta_1'|_P$ cannot be contractible.  
        
        However, for any such $\Delta_1'$, reversing roles of $\Delta_1'$ and $\Delta_0'$, we find a surface $P'$ with contractible section $\Delta_1'|_{P'}$ and non-contractible sections $\Delta_0'|_{P'}$ and $E_0'|_{P'}$.  Computation shows that this would imply that $E_0$ is in fact not contractible in $Y$, a contradiction.  Therefore, there can be at most one component $\Delta_0$ of $\Supp(F)$ that maps birationally onto $S$.
        
        Thus, in the smooth ruled surface $P$, we may assume that $E_0'|_P$ is a positive section, $(E_0'|_P)^2 = n$.  Choosing another section $\gamma$ such that $\gamma \cap \Delta_0'|_P = \emptyset$, $\gamma^2 = E_0' \cdot \gamma  = n$, and $K_P \cdot \gamma  = K_{Y'} \cdot \gamma = -2 - n$, so that $K_X \cdot \gamma' = (K_{Y'} + E_0')\cdot \gamma = -2$, where $\gamma'$ is the image of $\gamma$ on $X$.  Provided that $C$ and $\gamma$ are chosen generically, $D \cdot \gamma' = n \in \mathbb{Z}$, hence the relationship $dK_X + 4D \sim 0$ implies that $d(-2) + 4n = 0$ and $d$ is even. 
    \end{proof}

Now, assume that all components of $E$ are contracted by the MMP.  Let $E_0$ be the first component of $E$ contracted, and assume for contradiction that $d$ is odd. 

Consider the dlt model of the pair $(K_Y, E + F)$ over $X$, $(X^{\mathrm{dlt}}, E^{\mathrm{dlt}}) \to X$.  By \cite[Corollary 1.36]{newkollar}, this contracts all components of $\Supp F$.  

\begin{lemma}\label{lem:chainofF} 
    Suppose that $F_1$ is a component of $F$ whose image in $X^{\mathrm{dlt}}$ is a curve $C$ contained in $E$.  Then, the generic point of $C$ in $E$ is a cyclic quotient singularity of index $n$, and the discrepancy $-a(F_1, X) = \frac{a_1}{n}$.  If $F_i \cap E \ne \emptyset$, then the discrepancy is $-a(F_1, X) = 1 - \frac{1}{n}$.  Furthermore, if $E_0$ is the component of $E$ containing $C$, the collection of exceptional divisors $E_0$ and those $\{ F_i \}$ mapping to $C$ form a chain (i.e. above the generic point of $C$, the divisors can be written $E_0 \cup F_1 \cup F_2 \dots \cup F_k$ where each divisor only meets the adjacent divisors in the list and each intersection is only one point). 
\end{lemma} 

\begin{proof}
    Because $F$ is contracted to a curve and the discrepancy of the divisor $F$ is greater that $ -1$, at the generic point of $C$, we reduce to studying log terminal surface singularities.  Because this curve is contained in $E_0$, these are cyclic quotient singularities \cite[Theorem 4.15]{km98}.  Let $n$ be the index (and note that, by the discussion above, $n \mid d$, so $n$ is odd).  Then, the Cartier index of the divisor $K_{X^\mathrm{dlt}} + E^{\mathrm{dlt}}$ is a divisor of $n$, so all discrepancies take the form $\frac{a_i}{n}$.  
    
    Consider the map $g: E \dashrightarrow E^{\mathrm{dlt}}$.  At the generic point of $C$, this is an isomorphism, so a local computation of the discrepancy of $F_i$ can be done using adjunction.  Near the generic point of $C$, we have $K_Y + E + \sum a_i F_i = g^*(K_{X^{\mathrm{dlt}}} + E^{\mathrm{dlt}})$, and restricting to $E$, we obtain $K_E + \sum a_i F_i|_E = K_E + \Diff$, where $\Diff$ is supported on the codimension one part of singular locus of $X^{\mathrm{dlt}}$ contained in $E$, which is $C$.  By \cite[4.4]{newkollar}, the coefficients of the different are precisely $1 - \frac{1}{n}$.  So, for $F_1$ such that $F_1 \cap E \ne \emptyset$, the discrepancy is as claimed. 
    
    Finally, the description of the divisors as a chain is \cite[Theorem 4.15(3)]{km98}.
\end{proof}

We will divide the divisors $F_i$ in $\Supp F$ into two groups: those whose image in $X^{\mathrm{dlt}}$ is a curve in $E$, which we will continue to denote by $F_i$, and those whose image is not, which we will denote by $G_j$, and write $F = \sum -a(X,F_i)F_i$, respectively $G = \sum -a(X,G_j)G_j$, so we have the relationship $K_Y + E + F + G = \pi^*(K_X)$.  The exceptional divisors $F_i$ will play an important role in what follows.  Label the steps of the MMP $Y \dashrightarrow Y_1 \dashrightarrow \dots \dashrightarrow Y_{m+1} \dashrightarrow $, where $Y_m \to Y_{m+1}$ is the contraction of the image of $E_0$, and $\dim Y_{m+1} \le 3$.  We will denote the image of a divisor $H \subset Y$ in $Y_k$ by $H^k$.

\begin{lemma}\label{lem:firstcurvecontracted}
    Let $Y_l \dashrightarrow Y_{l+1}$ be the first contraction of a curve $C$ in $E_0^l$ (so $E_0 \cong E_0^l$).  Then, the only steps of the minimal model program up to this point were contractions that did not intersect $E_0$ or contractions of components of the exceptional locus of $\pi$.  If $E_0^l \to E_0^{l+1}$ is birational, then $C$ is not contained in any component of the exceptional locus of $\pi$ other than $E_0$ and and $C$ is a fiber of $\pi|_{E_0}: E_0 \to \pi(E_0)$.  
\end{lemma}

\begin{proof}
    Because $E_0 \cong E_0^l$, we will write $E_0$ in what follows.  Assume that $C$ is the first curve in $E_0$ contracted by the MMP.  Near $E_0$, prior to the contraction of $C$, we could only have contracted divisors $\Delta$ onto $E_0$ (i.e. the image of any contraction that intersects $E_0$ is a curve in $E_0$) or flipped curves into $E_0$ because $C$ was the first curve contracted.  But, if this was a divisor that intersected $E$ or a flip not contained in the exceptional locus, the curves contracted would indeed be $K+E$-negative, so the induced contraction could not be birational because it contracts no curves in $E_0$ (Lemma \ref{lem:smallcont}, Corollary \ref{cor:primeD}).  Therefore, the only contractions near $E_0$ before this curve $C$ were be divisorial and contracted divisors in the exceptional locus of $\pi$.  Furthermore, they must all be contracted to curves in $E_0$ by assumption that $C$ is the first curve contracted.  If $E_0$ is contracted to a curve $\pi(E_0)$ by $\pi$, and the first curve $C$ contracted is a (multi-)section of $\pi|_{E_0}: E_0 \to \pi(E_0)$, we divide into cases based on if the contraction of $C$ is birational on $E_0$.  If it is not birational, then we find that $E_0$ must be $\bP^1 \times C'$ by considering the projections from the contraction of $C$ and contraction $\pi|_{E_0}$.  (There cannot be any singular fibers: they would have been flipped into $E_0$, but we then find a contradiction to $K$-negativity of the contraction of $C$.)   If the contraction of $C$ is birational, we can obtain a contradiction to contractibility of $E_0^m$ by considering a general curve $\Gamma \subset E_0^m$ that is contracted when $E_0^m$ is contracted.  By a intersection-theoretic computation, $\Gamma$ cannot be a fiber of $\pi|_{E_0}$, so its (reduced) image in $X$ is the same as that of $C$.  Considering the intersection of these two curves with $\pi^*(K_X)$ (which should be multiples of each other) gives a contradiction.
    
    So, we may assume that the first curve $C$ in $E_0$ contracted is a fiber of $\pi|_{E_0}$.  In the MMP, we are only contracting $K_Y$-negative, $K_Y + E$-non positive curves.  In order to contract such a curve $C$ in $E_0$ in a step of the MMP, we must have first performed contractions over $E_0$ to change the positivity of $C$.  By the previous paragraph, these contractions must be those of divisors contained in the exceptional locus of $\pi$.  Suppose $\Delta$ is such a divisor contracted to $E_0$.  If it is contracted to a curve in $E_0$, the contraction would not change the sign of $(K_Y+E)\cdot C$.  If it is contracted to a point in $E_0$, there must have been a prior contraction changing the positivity of the curve $\Delta\cap E_0$, contradicting that this is the first curve contracted.  
    
    If $C$ was an intersection curve $G_j \cap E_0$, where $G_j$ is a component of the exceptional locus that is contracted to a point in $E$ in the dlt model, either it is contracted to a canonical point of $X^{\mathrm{dlt}}$, in which case the curves in $G_j$ are $K_Y$-non-positive, but $K_Y+E$-positive, because $E^{\mathrm{dlt}}$ is not canonical at the point.  Or, if it is contracted to a non-canonical point (or non-terminal curve not contained in $E$), the curves in $G_j$ are $K_Y$-positive (non-negative) and $K_Y+E $ positive.  In order for $C$ to be contracted, the positivity with respect to $K+E$ must change, but for any divisor contracted to a curve in $E_0$, the contraction would not change the sign of $(K_Y+E)\cdot C$.  Therefore, $C$ cannot be the intersection of a component $G_j \cap E_0$.   
\end{proof}

Next, we emphasize the importance of the relative nef-ness of $K_Y$ and the coefficients of $F_i$. 

\begin{lemma}\label{firstmustdofcontr}
     In the map $\pi: Y \to X$, let $C$ be a fiber of $\pi|_{E_0}$.  If the contraction of $E_0^m$ contracts a generic fiber $C$, before the contraction of $E_0$, there must have been a contraction of some $F_i$ such that $F_i \cap E_0 \ne \emptyset$.  In particular, the coefficient of such an $F_i$ in $\pi^*(K_X)$ is $1 - \frac{1}{n}$.  
\end{lemma}

    \begin{proof}
        By construction of $Y$, $K_Y$ is relatively nef, so for any fiber $C \subset E_0$, $K_Y \cdot C \ge 0$.  Therefore, in order for $C$ to be contracted in the course of the MMP $Y \dashrightarrow Y_m \to Y_{m+1}$ as a $K_{Y_m}$-negative extremal ray, we must have performed a contraction $Y_l \to Y_{l+1}$ of a divisor that intersected $E_0$ to change the positivity of $K_Y$ on $C$.  If the contraction was not a component of $\Supp F$, a generic contracted curve would necessarily be $K_{Y_l} + E_0^l$ negative (because its transform in $Y$ satisfies the same negativity), so Corollary \ref{cor:primeD} and Lemma \ref{lem:smallcont} imply that it would have been a contraction onto a lower dimensional variety, giving $Y_l$ the structure of a Fano fibration, which is a contradiction, or the contraction of a component $G_i^j$ onto a curve in $E_0^j$ for some $j < l$ followed by more contractions or flips.  However, because each $G_i$ is contracted to a point in the dlt model, by $\bQ$-factoriality, we must have performed a flip into $G_i$ before the contraction (otherwise, $G_i$ would already admit a contraction to a curve followed by a small contraction to $X^{\mathrm{dlt}}$).  However, using the relative nef-ness of the canonical divisor on $G_i$, a computation using the formulas for discrepancies as in Lemma \ref{lem:smallcont} shows that, after a flip, $G_i^j$ could not admit a fibration over a curve because we could not have $K_{Y_j} \cdot l = -1$ and $G_i^{j}\cdot l = -1$ for a generic fiber $l$ of the contraction.  Therefore, we must have performed a contraction of such an $F_i$ that intersects $E_0$ in a curve. 
    \end{proof}

Now, we separate into two cases based on the non-klt locus of $X$.  

\noindent \textbf{Case 1: $E_0$ is contracted to a point in $X$.}

By the previous lemma, before contracting $E_0^m$, we must have contracted a component of $\Supp F$ onto $E_0$.  Let $F_1$ be this component, with discrepancy $a(X,F_1) = 1 - \frac{1}{n}$, and suppose it is contracted at the step of the MMP $Y_l \to Y_{l+1}$.  Because $n \mid d$, we may assume $n \ge 3$ (if $n = 2$, $d$ would be even, as desired).  Note that $F_1$ is contractible in $Y$ to a curve (via the map to $X^{\mathrm{dlt}}$), so must have the structure of a (possibly singular) ruled surface.  Consider a generic fiber $f$ of this contraction, which is also a generic fiber of $F_1^l$.  Because $K_Y$ was relatively nef to begin with, we must have contracted some collection of divisors $\Delta$ onto $F_1$ before contracting $F_1^l$ to change the positivity of $K_Y$ on $F_1$.  In particular, $K_{Y_l} \cdot f^l = -1$ and $F_1^l \cdot f^l = -1$.  Furthermore, by Lemma \ref{lem:chainofF}, $E_0 \cdot f = E_0^l \cdot f^l = 1$, and $f$ can meet only one other component $F_2$ of the exceptional locus of $\pi$.  Furthermore, if $F_2$ exists, the intersection curves $E_0 \cap F_1$ and $F_1 \cap F_2$ must be disjoint (if they were not disjoint, the intersection curve $E_0 \cap F_2$ would necessarily be $K_Y$-negative, but we are assuming $K_Y$ is relatively nef). 

Note that a component $\Delta$ contracted onto $F_1$ cannot intersect $E_0$ because its contraction would necessarily contract a $K_Y\ge 0$ curve in $E_0$, and similarly, if $F_2$ exists and was not contracted before $\Delta$, $\Delta$ cannot intersect $F_2$.  If it did, because $\Delta$ was contracted by a $K$-negative extremal MMP contraction, all the fibers are numerically equivalent, however $F_2$ is contractible in $\Delta$ (to $X$, so the curve $\Delta \cap F_2$ must be a component of some fiber.  However, the other component of the fiber would then be $F_2$-positive, but the generic fiber of $\Delta$ is $F_2$-trivial; a contradiction.  Therefore, if $F_2$ exists, $\Delta \cap F_1$ is disjoint from both $E_0 \cap F_1$ and $F_1 \cap F_2$.  Also, in this setting where $\Delta$ is contracted before $F_2$, in the contraction of $F_1$, there cannot be any reducible fibers: the fibers must be numerically equivalent, but $F_1 \cap F_2$ and $E_0 \cap F_1$ are disjoint sections of the ruled surface $F_1$, so any component of a singular fiber must pass through both sections.  Because they are disjoint, a reducible fiber would have genus larger that 0, a contradiction.  Therefore, if $\Delta$ is contracted before $F_2$, $\Delta \cap F_1$ is a multisection of the ruled surface $F_1$ (with no singular fibers, so Picard rank 2) that is disjoint from the other two disjoint sections of $F_1$: $E_0 \cap F_1$ and $F_1 \cap F_2$.  This implies that in fact $F_1$ is $\mathbb{P}^1 \times C$, where $C$ is a smooth genus $g$ curve and $\Delta \cap F_1$ is a section.  Taking a generic horizontal section ${pt} \times C$, near $C$, we have the relationship $0 = \pi^*(K_X) \cdot C = (K_Y + F_1 - \frac{1}{n}F_1) \cdot C$, so $n(2g - 2) = F_1\cdot C$ and hence $K_Y \cdot C = - (n-1)(2g-2)$.  By relative nefness of the canonical divisor, this implies $g = 0$ or $g = 1$.  Finally, note that this analysis applies for any $\Delta$ contracted onto any $F_i$ in the chain of divisors above the curve $E_0 \cap F_1$.  The following lemma will show that this behavior is impossible when $d$ is odd.

\begin{lemma}\label{deltaonf1}
    If $F_i$ is $\bP^1 \times C$ and $\Delta$ a non-exceptional divisor of $\pi$ contracted onto $F_i$, then $i = 1$. 
\end{lemma}

\begin{proof}
     Because $\Delta$ is contracted onto $F_i$ in an extremal contraction of the minimal model program, it is a ruled surface, and $F_i$ is smooth, so there cannot be any singularities along $F_i \cap \Delta$.  Therefore, there cannot be any singular fibers of the ruled surface $\Delta$, and cannot be any singularities on $\Delta$ because curves through the singular points would have $-(K+\Delta)\cdot C $ less than that of the generic fiber, but $F_i \cdot C$ equal to that of a generic fiber, impossible by relative numerical equivalence.  So, $\Delta$ is a smooth ruled surface.  Because $F_i$ is contractible to a point, the section $F_i \cap \Delta$ is the negative section $\sigma_0$, and let $\sigma_{\infty}$ be generic positive section that does not intersect $\sigma_0$, and let $f$ be a generic fiber.  
     
     If $g = 1$, by the computation above, we see that $K_Y \cdot \sigma_0 = F_i \cdot \sigma_0 = \Delta \cdot \sigma_0 = 0, $ and as $F_i \cdot \sigma_0 = \sigma_0^2$, this implies that $\Delta$ is in fact $\PP^1 \times C_1$, where $C_1 \cong \sigma_0$ is genus 1, so choosing a generic section $C_1$, we see that $(K_Y+ \Delta)\cdot C_1 = 0$.  Because $\Delta $ is contractible, when $\Delta$ is contracted $Y_m \to Y_{m+1}$ in the minimal model program, near the generic section $C_1$, the pullback of $K_{Y_m}$ to $Y$ is $K_Y - \Delta -\sum \Gamma_i$, where $\Gamma_i$ are the other divisors contracted in the course of the MMP.  However, the image of each $\Gamma_i$ must be a point in $Y_m$: if it was a curve, either it is an entire fiber (and hence would have intersected $F_i$; contracting a curve in $F_i$, which is impossible), or a (multi)-section of $\Delta$ (which would have forced the fibers of $\Delta$ to be $K_Y$-non-negative, impossible as they are not contained in the exceptional locus of $\pi$).   None of these $\Gamma_i$ intersect $\sigma_0$ and none intersect the generic $C_1$, and $\sigma_0$ and $C_1$ have the same image in $Y_{m+1}$, so \[ 0 = (K_Y- \Delta)\cdot \sigma_0 = (K_Y - \Delta)\cdot C_1 .\]  Thus, $(2K_Y) \cdot C_1 = 0$, so $K_Y \cdot C_1 \ge 0$.  This is impossible as $C_1$ is not contained in the exceptional locus of $\pi$.  Therefore, $g\ne 1$. 
     
     If $g = 0$, we perform a similar computation: we find that $\Delta$ is a smooth ruled surface over a rational curve $\sigma_0$ and no $\Gamma_i$ as described above intersect $\sigma_0$ or the generic section, and the computations before the proof of this lemma say $K_Y\cdot \sigma_0 = 2(n-1)$, $F\cdot \sigma_0 = -2n$, and $\Delta\cdot \sigma_0 = 0$.  Therefore, for a generic section $\sigma_\infty$ of $\Delta$, $(K_Y+\Delta) \cdot \sigma_0 = -2(n+1)$, and as above, because $\sigma_0$ and $\sigma_{\infty}$ have the same image when $\Delta$ is contracted, \[ 2(n-1) = (K_Y-\Delta)\cdot \sigma_0 = (K_Y-\Delta)\cdot \sigma_{\infty} \] so $2K_Y\cdot \sigma_{\infty} = -4$ and $K_Y\cdot \sigma_{\infty} = -2$.  Therefore, if $\sigma_{\infty}$ intersects no other exceptional divisors of $\pi$, as it is contained in the smooth locus of $Y$, and $dK_X + 4D \sim 0$ implies that $d$ is even.  
     
     Finally, suppose that $\sigma_{\infty}$ intersects some other exceptional divisor $G_j \in \Supp G$ whose discrepancy $-a(G_j,X,\frac{4}{d}D) = \frac{a}{d}$.  Because $G_j$ and $D$ do not intersect $\sigma_0$, their restrictions to $\Delta$ are $G_{\Delta} \equiv p\sigma_{\infty}$ and $D_{\Delta} = q\sigma_{\infty}$.  For the generic fiber $f$ of $\Delta$, $f$ intersects a single $F_i$ and the discrepancy of $F_i$ is $\frac{n-r}{n}$ for some integer $r$.  Thus, the relationship $dK_X + 4D\sim 0$ pulled back to $Y$ implies that \[K_Y\vert_{\Delta}+\frac{n-r}{n}\sigma_0 + \frac{a}{d} G_{\Delta} + \frac{4}{d} D_{\Delta} \equiv 0 .\]
     
     Intersecting with a generic fiber $f$ and the generic section $\sigma_0$, we find that $r=1$ and such a $\Delta$ can only occur on $F_1$. 
\end{proof}

Now, in the above set-up, the proof is complete by the following lemma. 

\begin{lemma}\label{deltaontof1}
    If $\Delta$ is contracted by an extremal contraction onto a curve contained in $F_1$, followed by an extremal contraction of $F_1$, and $\Delta$ is not an exceptional divisor of $\pi$, then $d$ is even. 
\end{lemma}

\begin{proof}
    Suppose the generic fiber of $\Delta$ does not intersect any other exceptional divisor of $\pi$.  Because the contraction of $\Delta$ is a $K$-negative extremal contraction, for general fiber $f_g$, $K_{Y_l} \cdot f_g = -1$, and the same is true for the strict transform in $Y$, so $K_Y \cdot f_g = -1$.  Then, $\pi^*(K_X)\cdot f_g = K_Y\cdot f_g + (1 - \frac{1}{n}) F_1 \cdot f_l = -1 + 1 - \frac{1}{n} = -\frac{1}{n}$ because $f_g$ intersects no other exceptional components.  Let $f$ be the image of $f_g$ in $X$.  This curve meets $D$ in the smooth locus of $X$, so $0 = (dK_X + 4D)\cdot f =  - \frac{d}{n} + 4 k$ for some $k \in \bZ$.  Therefore, $d = 4nk$ so $d$ is even. 
    
    Now suppose $f_g \subset \Delta$ intersects some other exceptional divisor of $\pi$.  Call this divisor $G_j$ as above.  Because $G_j$ must be contracted to its image $G_{\Delta}$ in $X$, $G_j$ is a fibration over the curve $G_{\Delta}$.  Consider a generic fiber $f_g$ of $G_j$, and suppose $G_j\cdot f_g = -k$ (so $K_Y\cdot f_g = k-2$).  Intersecting with $\pi^*(K_X)$, we see that $k-2 + \frac{a}{d}(-k) + s = 0$, where $s \ge 0$ is the intersection of $f_g$ with other divisors in the formula $\pi^*(K_X)$.  Because $\frac{a}{d} < \frac{1}{n}$ and $n \ge 3$, this implies that $k = 2$. 
    
    Now, because our MMP contracts $\Delta$ and then $F_1$, we can study what happens to $G_j$ in these contractions.  After these contractions in the threefold $Y_{l+1}$, the images of the curves $f$ in $G_j^{l+1}$ are $K_{Y_{l+1}}$-negative, $E^{l+1}$-positive, and $K_{Y_{l+1}} + E^{l+1}$-negative.  In addition, they are $G_j^l$-trivial.  From the description of the exceptional divisors of $Y \to X$, this implies that there must be an extremal ray $R$ in the cone of curves of $Y_{l+1}$ that is $K_{Y_{l+1}}$-negative, $E^{l+1}$-positive, and $K_{Y_{l+1}} + E^{l+1}$-negative, and $G_j^{l+1}$-trivial, and its contraction cannot contract any curves in $E^{l+1}$ (the only curves in $E^{l+1}$ that are $G_j^{l+1}$ trivial are those that do not intersect $G_j^{l+1}$, but then their positivity is the same as that in $Y$, so they are $K_{Y_{l+1}}$-positive).  Therefore, choosing an alternate path in the MMP and contracting $R$, the MMP terminates at this contraction with a fibration over $E^{l+1}$, which implies that $d$ is even by Lemma \ref{lem:enotcontracted}.
\end{proof}

Now, we consider the case that $\Delta$ contracted onto the chain of divisors is such that no $F_i$ is $\bP^1 \times C$, but $\Delta$ is disjoint from $F_{i-1}$ and $F_{i+1}$, so the arguments above imply that $\Delta$ is only connected to $F_k$, the last divisor in the chain (c.f. Lemma \ref{lem:chainofF}).  Therefore, to get to the extremal contraction of the (image of) $F_1$ in the MMP, we contract $\Delta$ onto $F_k$, then $F_k$ onto $F_{k-1}$, $\dots$, $F_i$ onto $F_{i-1}$, $\dots$, and ultimately, $F_1$ onto $E$.  If any other surface $G_j$ is contracted onto $F_i$ for any $i$, it must be some component of $\Supp G$ (because some curve in it must begin $K_Y$-non-negative).  A brief computation with the discrepancies as in Lemma \ref{lem:notsmall} and using that $\Delta$, $F_i$ are ruled surfaces (albeit possibly singular) and that the curves $C$ in $F_i$ must satisfy the relationship $\pi^*(K_X) \cdot C = 0$, one can show that there are no flips into and out of $F_i$.  At this point, we work backwards from the contraction of $F_1$.  We have reduced to the case that we performed several divisorial contractions of divisors in $\Supp F$, possibly with some divisors in $\Supp G$, after beginning with a divisorial contraction of a component $\Delta$ not contained in the exceptional locus onto a curve $C$ in $F_k$ (the last divisor in the chain).  It is not hard to show that, if the generic fiber of $\Delta$ intersects some exceptional divisor $G_j$ other than $F_k$, then $G_j\cap \Delta$ is not contractible in $\Delta$; i.e $G_j$ contracts onto the curve $G_j \cap \Delta$ in the morphism $\pi: Y \to X$.  We begin with a specific case, $n = 3$.  

\begin{lemma}\label{n=3}
    Suppose $n = 3$.  Then, $d$ is even. 
\end{lemma}

\begin{proof}
    If $n = 3$, the chain $F_i$ is either one component $F_1$ with discrepancy $-\frac{2}{3}$ or two components $F_1, F_2$ with discrepancies $-\frac{2}{3}$ and $-\frac{1}{3}$. Suppose we are in the first case; then the generic fiber $f_g$ of $\Delta$ contracted by the first step of the MMP intersects $F_1$ so $\pi^*(K_X)\cdot f_g = (K_Y + \frac{2}{3} F_1 + \sum \frac{a_i}{d} G_i) \cdot f_g = -\frac{1}{3} + \sum \frac{a_i}{d} m_i$ for some $m_i$.  If there are no components $G_i$ intersecting $f_g$, $K_X\cdot \pi(f_g) = - \frac{1}{3}$, and the relationship $dK_X + 4D$ then implies $d$ is even.  If there do exist $G_i$, the equation implies each coefficient $\frac{a_i}{d} < \frac{1}{3}$.  For a component $G_j$ intersecting $f_g$, this implies the general fiber $l_g$ of the ruled surface $G_j$ satisfies $G_j \cdot l_g = -2$, so by the same argument in Lemma \ref{deltaontof1}, after contracting $\Delta$ and $F_1$, there exists a $K+E$-negative $E$-positive extremal ray in the cone of curves, and we can choose the MMP to terminate with a fibration over $E$.  
    
    In the second case, we contract $\Delta$ onto $F_2$ and then $F_2$ onto $F_1$, followed by $F_1$ onto $E_0$.  Again considering a generic fiber $f_g$ of $\Delta$, if $f_g$ intersects no components $G_i$, we find that $K_X\cdot \pi(f_g) = - \frac{2}{3}$ and the relationship $dK_X + 4D$ implies $d$ is even.  Suppose there are components $G_i$ intersecting $f_g$.  Let $G_j$ be such a component and let $l_g$ be a general fiber of the ruled surface $G_j$.  If $G_j \cdot l_g = -3$ or $-2$, after contracting some of the components $\Delta, F_2$, and $F_1$, we find a $\pi^*(K_X)$-negative but $E_0 + F$-positive extremal ray.  The contraction of such a ray terminates in a fibration over $E_0$ or one component $F_i$.  Note that $E_0$ cannot be contracted (as the curves in $E_0$ are not $K$-negative) so by Lemma \ref{lem:enotcontracted}, $d$ is even. 
    
    Suppose then all components $G$ meeting $f_g$ have $G\cdot l_g \ge -4$.  Using that $\pi^*(K_X)\cdot f_g = (K_Y + \frac{2}{3} F_1 + \sum \frac{a_i}{d} G_i) \cdot f_g < 0$, we know that the coefficient of any such $G$ is $< \frac{2}{3}$.  Given this bound on the discrepancy and that $G\cdot l_g \le -4$, we can compute all possible configurations of divisors $G_j$.  By taking a generic hyperplane section $H$ meeting the image of $G$ in $X$ and noting that this is log terminal at the generic point of $\pi(G)$, we reduce to studying configurations $C$ of curves on $H_Y = \pi^{-1}_*H$ that contract to a log terminal singularity with at least one $-n$ curve, $n \ge 4$, with coefficient of that curve $< \frac{2}{3}$.  We find only the following possibilities: $C$ is one $-5$ curve, with coefficient $\frac{3}{5}$, $C$ is a $-4$ curve meeting a $-3$ curve, with coefficient $\frac{7}{11}$ on the $-4$ curve, or $C$ is a $-4$ curve meeting a chain of $k-1$ $-2$ curves, with coefficient $\frac{2k}{3k+1}$ on the $-4$ curve.  From this description, we see that $\pi^*(K_X)\cdot f_g = (K_Y + \frac{2}{3} F_1 + \sum \frac{a_i}{d} G_i) \cdot f_g < 0$ implies that there is only one $G$ meeting $f_g$, and plugging in the given discrepancies in each of the three cases, we see that $K_X \cdot \pi(f_g) = -1 + \frac{1}{3} + \frac{a}{d}$, which is either $-\frac{1}{15}$, $-\frac{1}{33}$, or $-\frac{2}{3(3k+1)}$.  With the relationship $dK_X+4D \sim 0$, each case implies that $d$ is even. 
\end{proof}

Now, suppose $n > 3$, so $n \ge 5$ because $n$ is odd.  We perform contractions of $\Delta$ and $F_i$ ultimately contracting $F_1$ onto $E$.  In our final contraction $F_1$ contracting onto $E$, say $Y_l \to Y_{l+1}$, along the image of $F_1$ in $E^{l+1}$, the threefold $Y_{l+1}$ has only index one singularities.  Indeed, as these are all divisorial contractions, all of the fibers are numerically equivalent, i.e. all multiple of the general fiber $f_g$ which satisfies $K_{Y_l} \cdot f_g = -1$, $F_1^l \cdot f_g = -1$, and $E_0^l \cdot f_g = 1$, and the pullbacks of both the general fiber $f_0$ and $f_g$ must satisfy $0 = \pi^*(K_X) \cdot f$.  Because  $\pi^*(K_X) = K_Y + E_0 + (1-\frac{1}{n})F_1+\dots$, for a special fiber $f_0$, we can use the wealth of information on divisorial extremal neighborhoods \cite{mori, classofflips, moriprokhorov, tzi} and formulas for intersection with $-K_{Y_l}$ (and its numerical equivalence with $F_i^l$ and $E_0^l$) along with adjunction on the ruled surface $F_1$ to compute that $-K_{Y_l} \cdot f_0$ must be $-\frac{1}{r}$, where $r$ is the index of the threefold singularity on the fiber $f_0$.  Then, \cite[Theorem 5.5]{tzi} implies that $Y_{l+1}$ has only index one singularities and $Y_{l+1}$ is factorial.  Therefore, $K_{Y_{l+1}}$ and $E^{l+1}$ are both Cartier, so by adjunction, $E_0^{l+1}$ has only Du Val singularities along the image of $F_1$.  Let $C_F$ be the curve that is the image of $F_1$.  By \cite[Lemma 5.1]{tzi}, $C_F$ is a smooth curve and cannot intersect any exceptional divisors of $\pi$ other than $E_0^{l+1}$ (if it did, the very first contraction of $\Delta$ would necessarily have contracted a $K$-non-negative curve in one of these divisors, a contradiction).  By comparing $C_F$ to $\pi^*(K_X)|_{E_0}$, we find that $C_F$ satisfies $(K_{E_0} + \frac{n-1}{n}C_F)\cdot C_F = 0$, i.e. $(K_{E_0} + C_F)\cdot C_F = \frac{1}{n}C_F^2$.  In particular, $C_F$ is contractible in $E_0$ if and only if it is rational with at most three singularities.  However, if $C_F^2 < 0$, then it is not hard to show it is extremal in the cone of curves of $Y_{l+1}$, so we may contract (and possibly flip) $C_F$ before continuing the MMP.  This is certainly a $K_{Y_{l+1}} + E_0^{l+1}$-positive contraction, but the flip exists because the divisor $D^{l+1}$ (the image of $\pi^{-1}_*D$) satisfies $(Y_{l+1}, E^{l+1} + (\frac{4}{d} - \epsilon)D^{l+1})$ is log terminal but the log canonical divisor is negative on $C_F$ for $\epsilon$ sufficiently small.  Because this ray is $K_{Y_{l+1}}+E_0^{l+1}$-positive, the contraction/flip may create worse singularities on the threefold $Y_{l+2}$, but $E_0^{l+2}$ is still Cartier in codimension two, so our earlier results (c.f. Lemma \ref{lem:divisorialcont}, \ref{lem:notsmall}) still apply.  In this case, after the contraction of $C_F$, the curves in $E_0^{l+2}$ are still $K_{Y_{l+2}}$-non-negative, so if $E_0^{l+2}$ is contracted by the MMP, following the logic in Lemmas \ref{lem:firstcurvecontracted}, \ref{firstmustdofcontr}, there must exist another chain of contractions of components of $\Supp F$ changing the positivity of these curves.  Therefore, by repeating this process, we may assume that the curve $C_F$ has non-negative self intersection.  If $C_F^2 = 0$, then there exists a morphism $E_0 \to C$ fibering $E_0$ over a curve with $C_F$ as one of its fibers.  However, as $C_F$ does not intersect any other exceptional divisors, the same is true for the generic fiber, so the generic fiber $f_g$ of this map would satisfy $K_{E_0} \cdot f_g = -2$ so $0 = \pi^*(K_X)\cdot f_g = (K_X + E_0 + \dots)\cdot f_g = -2$ (where $\dots$ indicate the other exceptional divisors which by assumption do not intersect $f_g$), a contradiction.

Therefore, we may assume that $C_F^2 > 0$ which implies that $C_F$ is a big divisor on the surface $E_0$.  Because $-K_{E_0} = \frac{n-1}{n}C_F$, $-K_{E_0}$ is also big.  Because $C_F$ is a smooth curve through at worst Du Val singularities that does not intersect any other exceptional divisors of $\pi$, the pair $(E_0, C_F + \sum -a(\Delta_i, X) \Delta_i|_{E_0})$ is dlt, where the $\Delta_i$ are the other exceptional divisors of $\pi$, so we may consider the $K_{E_0} + C_F + \sum -a(\Delta_i, X) \Delta_i|_{E_0}$ MMP on $E_0$.  Here, the divisor $K_{E_0} + \frac{n-1}{n} C_F + \sum -a(\Delta_i, X) \Delta_i|_{E_0} = 0$ and the components $\Delta_i$ do not intersect $C_F$, so this contracts all components $\Delta_i$ along with any additional curves that do not intersect $C_F$.  Denote the image of this surface by $S$.  This is an isomorphism on $C_F$, so we continue to denote $C_F$ in the same way.  By construction, $C_F$ is ample on $S$ and $-K_S = \frac{n-1}{n} C_F$, so $S$ is Fano.  If $S$ has any strictly log canonical (non klt) singularities, then $S$ has index of $-K_S$ equal to 1,2,3,4 or 6 (see Lemma \ref{logcanonicalsurfacesings}), and the index divides $d \ge 5$, so we must have index 1 or 3.  It the index is 1, $K_S$ is Gorenstein near the non klt singularity.  Because $-K_S$ is ample, by Theorem \ref{thm:conelogcanonical}, $S$ is an elliptic cone, in which case it is impossible for $-K_S = \frac{n-1}{n}C_F$ unless $n =2$, which implies that $d$ is even.  If the index is 3, then $S$ is rational \cite[Theorem 8.5]{hacking}, so we assume that $S$ is a rational Fano surface.  

Now, we proceed by comparing the minimal resolution of $S$ and its minimal model.  Let $\phi: M \to S$ be the minimal resolution, and let $M \to Z$ be the minimal model of $M$ (with respect to $K_M$).  By \cite[Corollary 3.4]{nakayama} any curve $C \subset M$ with $C^2 < 0 $ is either contained in the exceptional locus of $M \to S$ or is a $-1$ curve.  By the ampleness of $C_F$, any $-1$ curve in $M$ must intersect the pullback of $C_F$.

Following the ideas in \cite{nakayama, Fujita}, we can classify such surfaces.  Although their work is in the case when $C_F$ is Cartier and $S$ is klt, the Du Val singularities do not provide a significant complication.  For example, suppose that $S$ is singular at a point $p\in C_F$.  By assumption, this is a Du Val singularity.  Let $P$ be the chain of $-$2 curves above $p$ in the minimal resolution $M$.  If $C \subset M$ is a $-1$ curve that intersects the chain $P$, we can contract $C$ and then contract the curves in $P$ until we find either a degree 0 curve (in which case, the MMP can terminate with a fibration) or produce a $-1$ curve connected to only smooth points of $C_F$ and then follow the ideas of \cite{Fujita}.  For brevity, we only sketch the idea here.  We first reduce to $S$ only having $A_1$ singularities along $C_F$: if there were larger $A_n$ singularities, there must be a $-1$ curve connected somewhere to the chain, and contracting the $-1$ curve and connected $-2$ curves yields a fibration where $\phi^*(K_S + \frac{n-1}{n}C_F)$ is not trivial.  Then, we note that any $-2$ curve must meet a $-1$ curve (it cannot be a section in $Z$ as a general fiber of the fibration would also not satisfy $\phi^*(K_S + \frac{n-1}{n}C_F)$ trivial), and such a $-1$ curve cannot meet $C_F$ in any other point (because if it did, and because $n \ge 5$, that $-1$ curve would not satisfy $\phi^*(K_S + \frac{n-1}{n}C_F)$ triviality).  Therefore, we find $-1$ curves meeting a $-2$ curve and any other negative curves attached to the $-1$ curves have discrepancy at most $\frac{n-1}{2n}$.  Ultimately, this is impossible to satisfy while simultaneously maintaining $\phi^*(K_S + \frac{n-1}{n}C_F) = 0$.  Therefore, we conclude that there are no singularities along $C_F$.  Then, we show that $S$ must be klt: if $S$ has a log canonical singularity of Cartier index 3, we show that the unique fork of the dual graph in the minimal resolution must be the negative section of the minimal model, but then $C_F$ must also be a section and hence rational.  However, from the formula $K_S + F = \frac{1}{n}C_F$, this is only possible if $C_F^2 = -2n <0$, which is a contradiction.   Therefore, we find that $S$ is a klt Fano surface with $C_F$ Cartier, and we can apply Fujita's results directly.  

So, we assume $S$ is a log terminal Fano surface, and $C_F$ is a Cartier divisor on $S$ satisfying $K_{S} + \frac{n-1}{n}C_F =0$ for some $n \ge 5$, and we use the classification in \cite{Fujita}) to bound its fractional index and prove that there is only one possibility for $S$.

\begin{definition}
    For a $\mathbb{Q}$-Fano log terminal variety $Z$, define the \textit{fractional index} as \[ r(Z) = \sup \{r\in \mathbb{Q}_{>0} \mid -K_Z \equiv rL \text{ for an ample Cartier divisor } L \} .\]
\end{definition}

\begin{lemma}
    In the setting above, the index must in fact be $\frac{n-1}{n}$.
\end{lemma}

    \begin{proof}
        Indeed, if it were less than this, computations show that we would find a smaller ample divisor $L$ and $kL = C_F$, $k \ge 2$, but this would imply that the fractional index is at least $1$, but then by \cite{Fujita} $S$ is either $\bP(1,1,q)$ or has at worst Du Val singularities.  If $S$ is $\bP(1,1,q)$, using the relationship $K_S + \frac{n-1}{n}C_F = 0$, the fact that $C_F$ misses the singular point of $S$, a short computation shows that $n \le 4$, contradicting our assumption that $n \ge 5$.  If instead $S$ has Du Val singularities, then we can find a curve $l$ such that $K_S \cdot l = -1, -2, $ or $-3$, and then use the relationship $K_S + \frac{n-1}{n}C_F = 0$ to get a contradiction if $n \ge 5$.
    \end{proof}

So, given that the fractional index is $\frac{n-1}{n}$, $n \ge 5$, we can apply the results in \cite{Fujita}.  From the classification in that paper, we see that the only possible fractional index of this form is $\frac{4}{5}$.  By Fujita's classification, $S$ must have a singularity of type $\frac{1}{5}(1,2)$ away from $C_F$.  $S$ is constructed as a birational model of $\mathbb{F}_3$: blow up two points on a fiber $l$ away from the negative section of $\mathbb{F}_3$, and let $M \to \mathbb{F}_3$ be the resulting surface.  Then, contract the strict transform of $l$ and the strict transform of the negative section to obtain $M \to S$, with an isolated $\frac{1}{5}(1,2)$ singularity on $S$.  The curve $C_F$ is obtained as the image of the strict transform of the degree two multisection of $\mathbb{F}_3$ that passes through the two points blown up.  Let $l_1$ and $l_2$ be the images of the two $-1$ curves in $S$.  Then, by construction, either $E_0 = S$, $E_0 = M$, or $E_0$ is obtained from $M$ by blowing down either the $-2$ or the $-3$ curve (but not both).  Note that this implies that $E = E_0$ (there are no divisors with discrepancy $1$) and there are no other components of $\Supp F$ that intersect $E$ (no divisors have discrepancy of the form $\frac{n-1}{n}$).

At this point, the problem becomes purely computational, studying the positivity and negativity of $K_{Y_{l+1}}$ and $E^{l+1}$.  Suppose we are in the first case, $E = S$ (the other cases are similar, with extra divisors $G_i$ attached along the exceptional curves in $E$).  Because $C_F$ is contained in the smooth locus of $E$, the contraction of each component of the exceptional chain $F_1 \cup \dots \cup F_k$ is the contraction of a smooth ruled surface.  Because $n = 5$, there are only three possibilities for the chain itself (as the generic point contracts to a cyclic quotient singularity of index 5): either there is one surface $F_1$ with coefficient $\frac{4}{5}$ in $\pi^*(K_X)$; two surfaces $F_1 \cup F_2$ with coefficient either $\frac{4}{5}$ and $\frac{2}{5}$ or $\frac{4}{5}$ and $\frac{3}{5}$; or four surfaces $F_1 \cup F_2 \cup F_3 \cup F_4$ with coefficient (of $F_i$) $1 - \frac{1}{n}$.  Therefore, thus far the MMP has been the contraction of a smooth ruled surface $\Delta$ onto $F_k$ (the last divisor in the chain) followed by the contractions of each $F_i$.  By computation, $(K_{Y_{l+1}} + E^{l+1})\cdot C_F = -8$ and $C_F^2 = 10 $ (as a curve in $E^{l+1}$).  Denote by $E^{l+1} \cdot C_F = -d$ and $K_{Y_{l+1}}\cdot C_F = d-8$.  Let $\phi: Y \to Y_l$ be the composition of the contractions thus far.  We can compute $\phi^*K_{Y_{l+1}}$ and $\phi^*E^{l+1}$ and, together with the hypotheses that curves in $\Supp F$ are $K$-non-negative but the generic curves in $\Delta$ are $K$-negative and that all components of $\Supp F$ and $\Delta$ are smooth ruled surfaces, computation implies that $d > 8$ and there can be no extra exceptional components $G_i \in \Supp G$ of $\pi$ attached to $\Delta$.  (In fact, this already rules out the first two possibilities for the chains $F_i$: the generic fiber $f_g$ of $\Delta$ would then have image in $X$ satisfying $K_X \cdot \pi(f_g) = \frac{i}{5}$, $i = 1, 2, $ or $3$, so we use $dK_X + 4D \sim 0$ to conclude $d$ is even.) 

Now, from the computation that $d > 8$, we see that $E^{l+1} \cdot C_F < 0$ and $K_{Y_{l+1}}\cdot C_F > 0$.  There must exist a $K$-negative extremal ray somewhere in $Y_{l+1}$, and the only contractible curves in $E^{l+1}$ are $l_1$ and $l_2$.  If at this point of the MMP we find another $K$-negative, $E^{l+1}$-positive extremal ray that does not contain $l_1$ or $l_2$, because there were no extra exceptional components $G$ of $\pi$ attached to $\Delta$, this ray is necessarily $K+E^{l+1}$-negative, and the MMP terminates in a fibration, so we are done by previous arguments.  If there are $K$-negative rays away from $E^{l+1}$ (i.e. are $E$-trivial and do not contain $l_1$ or $l_2$), we contract and continue the MMP. 

The only other possibility is that we find only $K$-negative rays that contain $l_1$ or $l_2$, and we will show that this is not possible.  These cannot both be $K$-negative: it is straightforward to show $C_F$ in in the face spanned by these curves, and $K_{Y_{l+1}}\cdot C_F > 0$.  Suppose the only $K$-negative ray contains $l_1$.  Let $\phi_{l+1}: Y_{l+1} \dashrightarrow Y_{l+2}$ be the associated step of the MMP.  At this stage, $E^{l+2} \cong \mathbb{P}(1,1,2)$, so at the next stage of the MMP, $l_2$ and $C_F$ must be numerically equivalent.  However, on $E^{l+1}$, $\phi_{l+1}$ is a morphism, so we can compute the pullback of $K_{Y_{l+2}}$ to $E^{l+1}$.  Using the numerical equivalence of $l_2$ and $C_F$ after the contraction, we find that $K_{Y_{l+1}} \cdot l_2 < 0$, contradicting that both curves were not $K$-negative.  This concludes the proof in the case that $E_0$ is contracted to a point in $X$. 

\noindent \textbf{Case 2: $E_0$ is contracted to a curve in $X$.}

This case uses that the generic point of the image of $E_0$ is a log canonical surface singularity, so we have a classification.  The following result is well-known by classification of log canonical surface singularities (for example, \cite[Theorem 4.7]{km98} or \cite[Chapter 3]{kollarplus}).

\begin{lemma}\label{logcanonicalsurfacesings}
    If $X$ is a threefold with a one-dimensional locus of log canonical singularities $C$, over the generic point $p$ of a component of $C$, the pre-image of $p$ in the terminalization is: 
        \begin{enumerate}
            \item A smooth elliptic or nodal curve, or a cycle of smooth rational curves; 
            \item A tree with exactly one fork, which has three branches with lengths $n_1 - 1, n_2 - 1, $ and $n_3 - 1$.  The possible values for $(n_1, n_2, n_3)$ are $(2,3,6)$, $(2,4,4)$, or $(3,3,3)$.  
            \item A tree with exactly two forks as pictured in \cite[Theorem 4.7(3)]{km98}.  
        \end{enumerate}
    In each case, the index of the canonical divisor near $p$ is $1, 2, 3, 4, $ or $6$. 
\end{lemma}

Because the index of $K_X$ divides $d$, we immediately obtain the following:

\begin{corollary}
    If $X$ is log canonical along a curve and $d$ is odd, the index at the generic point must be $1$ or $3$.  
\end{corollary}

The only cases above that have index $1$ or $3$ are case (i) or the $(3,3,3)$ case of (ii).  Suppose $C$ in $E_0$ is the first curve contracted.  By Lemma \ref{lem:firstcurvecontracted} and its proof, if $C$ is not a fiber of $\pi|_{E_0}$, then $E_0 = \bP^1 \times C'$.  If $C$ intersects no other exceptional components of $\pi$, then $K_X\cdot\pi(C) = -2$, so the relationship $dK_X + 4D$ implies that $D$ is even.  If $C$ does intersect other exceptional components, they must be components of $\Supp F$ with discrepancy $\frac{n-1}{n}$.  In this case, if any component of $\Supp F$ is not supported on fibers of $\pi|_{E_0}$, then we must be in case (ii), $(3,3,3)$, which implies $C' = \bP^1$ so $E_0 = \bP^1 \times \bP^1$ and $n = 3$.  In this case, the configuration of divisors above the log canonical singularity consists of one fork and three branches, each with length 2, and the branches contract to cyclic quotient singularities of type $\frac{1}{3}(1,2)$ along the generic point.  The discrepancies of the $F_i$ meeting $E$ are therefore $\frac{2}{3}$, and because $C \cong \bP^1$, we find that $K_X \cdot \pi(C) = -2 + \frac{2}{3}k$, where $k$ is the number of components $F_i$ that $C$ intersects.  This implies $d$ is even provided $k \ne 1$.  If $k = 1$, then $C$ intersects $F_1$ but not $F_2$ or $F_3$, so $F_1|_{E_0}$ is a section of $\calO(1,1)$ and $F_2$ and $F_3$ restrict to fibers proportional to $C$.  However, considering the fiber of $\pi$ over the intersection point of $F_1$ and $F_2$, these surfaces intersect along some curve $f$ that is contracted to $X$.  However, computing $0 = \pi^*(K_X)\cdot f$, we find a contradiction that $K_Y \cdot f \ge 0$.  Therefore, $C$ only intersects components of $\Supp F$ supported on the fibers of $\pi|_{E_0}$, and it is easy to see that there must exist another $K$-negative $E$-positive extremal ray (other than $C$) and it cannot be contained in the exceptional locus of $\pi$, so must be $K+E$-negative, so its contraction results in a fibration and we see that $d$ is even. 

Otherwise, the first curve in $E_0$ contracted must be a fiber of $\pi|_{E_0}$.  The only contractions thus far did not contract curves in $E_0$, and this is the first curve contracted, so there must have been a contraction of some component of $\Supp F$ onto $E_0$ (Lemma \ref{firstmustdofcontr}).  The existence of such a component $F_i$ implies that we are in the (3,3,3) case in (ii), and the contraction of divisors $F_i$ means we can apply the arguments in the contracted to a point case about contractions of chains of divisors in $\Supp F$.  However, in order for the singularity to be log canonical with the appropriate configuration, not only is the index 3, but this discrepancy of the divisor $F$ must be $\frac{2}{3}$, so we must have $n = 3$.  This is impossible by Lemma \ref{n=3}.  This concludes the proof in the case that $E_0$ is contracted to a curve. 
\end{proof}

Lastly, we can immediately generalize this result to the case of non-normal slc varieties with anti-ample canonical sheaf and whose normalizations have non-lt singularities.  By normalizing the variety and considering each component $(X^\nu,\Delta)$ (where $X^\nu$ is a component of the normalization and $\Delta$ is the double locus), this is equivalent to studying the case of a pair $(X^\nu,\Delta)$ where $- (K_{X^\nu} + \Delta)$ is ample and the $1$-dimensional locus of log canonical singularities intersects $\Delta$.  Because the locus of log canonical singularities must intersect $\Delta$ \cite{kollarplus} but cannot be contained in $\Delta$, $X^\nu$ must in fact have a log canonical singularity along a curve.  Then, a terminal modification $X'$ of $X^\nu$ and minimal model program on $X'$ gives the same conclusion, where $\Delta$ is not considered as a component of $E$ (but we instead use the relationship $d(K_{X^\nu} + \Delta) + 4D \sim 0$). 

\subsection{Applications to $\calM_{\dfours}$ and  $\calM_{\dfour}$}
We can use Theorem \ref{thm:even} to obtain boundedness of $\dfour$ \he stable pairs without the smoothability assumption and other interesting corollaries.  First, we slightly rephrase Theorem \ref{thm:even} and combine this with the discussion in the previous paragraph.  

\begin{theorem}\label{thm:slt} If $(X,D)$ is a $\dfour$ \he stable pair and $d$ is odd, then any component of the normalization of $(X,\frac{4}{d}D)$ is dlt.
\end{theorem}

\noindent For quintic surfaces, we can be even more precise (see Theorem \ref{quinticsplt}).

In what follows, we make a careful distinction in each statement between smoothable and non-smoothable pairs.  First, a corollary of Theorem \ref{thm:totaro}: 

\begin{corollary}
For $d$ odd, the normal varieties $X$ occurring in a $\dfours$ \he stable pair are rational.
\end{corollary}

Finally, using Theorem \ref{thm:slt}, we are also able to obtain boundedness for $\dfour$ \he stable pairs when $d$ is odd, not just smoothable ones.  The following theorem is a special case of \cite[Corollary 1.7]{hmx2} for threefolds.  

\begin{theorem}\label{thm:boundednessP3} For fixed odd degree $d$, the set of $\dfour$ \he stable pairs form a bounded family.
\end{theorem}

\begin{proof}
Restricting to the normal case, by Theorem 3.1, $(X,\frac{4}{d}D)$ is klt, and because of the assumption that $dK_X + 4D \sim 0$, the pairs $(X,\frac{4}{d}D)$ are $\epsilon$-log terminal because $dK_X + 4D \sim 0$ is linear equivalence (not just numerical).  Then, $-K_X$ is ample and $K_X + \frac{4}{d} D$ is numerically trivial by assumption, hence by \cite[Corollary 1.7]{hmx2}, form a bounded family.  We can restrict to the normal case because the non-normal pairs are in bijection with normal pairs and a certain involution as in \cite[Theorem 5.13]{newkollar} and, by Theorem \ref{thm:slt}, the normalization is $\epsilon$-log terminal.
\end{proof}

Next, we study log terminal Fano degenerations of $\PP^3$ to determine the boundary of the moduli space of $\dfours$ \he stable pairs, generalizing from $\PP^3$ to $\PP^n$ when possible.  We will first focus on threefolds with canonical singularities appearing in the moduli problem.  

\subsection{Canonical Fano threefolds}\label{sec:canonicalclassification}

A standard reference for canonical threefolds is \cite{ypg}.  In the Fano case, particularly when $X$ is Gorenstein, such threefolds can be classified by invariants like $K_X^3$ and the Fano index.  If $X$ has at worst canonical singularities, the Fletcher-Reid plurigenus formula \cite[Theorem 10.2]{ypg} gives the plurigenera of $X$ in terms of $K_X^3$, $\chi(\calO_X)$, and coefficients $c_P$ determined by a basket of singularities $\{ Q_i \}$ for $X$.  In \cite[Theorem 1.1]{fletcher}, Fletcher shows the plurigenus formula is exact, meaning that any two canonical threefolds with the same plurigenera have the same $K_X^3$, $\chi(\calO_X)$, and basket of singularities are the same.  The contribution from the singularities is nonzero precisely when there are points $Q_i$ such that $K_{X'}$ is not Cartier at $Q_i$.  In the case at hand, $X$ is a flat degeneration of $\PP^3$, the plurigenera of $X$ and $\PP^3$ are the same, so this inversion of the plurigenus formula implies that $X'$ must be a terminal Gorenstein variety.  Because $X' \to X$ is any crepant partial resolution such that $X'$ has only terminal singularities, $K_{X'}^3 = -64$ and we can take $X'$ to be $\Q$-factorial.  Thus, we obtain the following.

\begin{theorem}\label{thm:terminal}
If $X$ is a terminal variety that admits a smoothing to $\PP^3$, then $X \cong \PP^3$.
\end{theorem}

\begin{proof}
The Fletcher-Reid plurigenus formula shows that if $X$ is not Gorenstein, it does not admit a smoothing to $\PP^3$, so it suffices to consider Gorenstein threefolds $X$.  In this case, \cite[Theorem 2.1]{almostfano} implies that the Fano index of $X$, the maximal integer $r$ such that $K_X \sim -rH$ for $\calO(H) \in \Pic(X)$, is equal to that of $\PP^3$.  Therefore, the Fano index of $X$ is 4.  Then, \cite[Theorem 3.1]{almostfano} says that, because the Fano index is maximal, $X \cong \PP^3$.
\end{proof}

There do exist non-trivial canonical degenerations of $\PP^3$.  

\begin{example}\label{ex:O(2,2)}
Observe that the standard embedding of the quadric surface $\PP^1 \times \PP^1 \subset \PP^3$ is an element of the linear system $\calO_{\PP^3}(2)$.  Let $Z$ be the image of the degree two embedding of $\PP^3 \hookrightarrow \PP^9$.  There is a standard degeneration from $Z$ to the cone over a hyperplane section of $Z$ by taking the cone over $Z$ (see, for example, \cite[Example 7.61]{km98}).  In this case, the hyperplane section of $Z$ corresponds to an element of $\calO_{\PP^3}(2)$, and is the $\calO(2,2)$ embedding of the quadric surface in $\PP^8$.  A computation shows that the cone over this is indeed Gorenstein as it is the cone over the anti-canonical embedding of $\PP^1 \times \PP^1$.  A check shows that this has canonical singularities; for details see \cite[Lemma 3.1]{newkollar}.  Therefore, this gives an example of a flat degeneration of $\PP^3$ to a Gorenstein canonical variety.  
\end{example}

A priori there may be many canonical degenerations of $\PP^3$, but the following theorem shows that if $d$ is odd, they must be closely related to the previous example.  In fact, they must be $\PP^3$ or cones over the anticanonical embeddings of elements of the linear system $|\calO_{\mathbb{P}^3}(2)|$, which have a simple description.  

\begin{theorem}\label{thm:nocanonical}
For odd degree $d$, if $X$ is a canonical threefold appearing in a $\dfours$ \he stable pair $(X,D)$, then $X$ is either $\PP^3$, the cone over the anticanonical embedding of $\PP^1 \times \PP^1$, or the cone over the anticanonical embedding of the quadric cone, $\PP(1,1,2,4)$.  
\end{theorem}

If $X$ has only terminal singularities, Theorem \ref{thm:terminal} implies the result.  If $X$ has canonical singularities, consider a crepant partial resolution $X' \to X$ such that $X'$ is terminal and $\Q$-factorial.  Before giving the proof, we give a sketch of the argument.  

By \cite{fletcher}, if $K_{X'}$ is not Cartier, there is a nonzero contribution to a basket of singularities on $X$, so $X$ is not isomorphic to $\mathbb{P}^3$.  It then suffices to consider the case where $X'$ is a terminal, $\Q$-factorial Gorenstein variety with $-K_{X'}$ nef.  Running a minimal model program on $X'$, if it terminates in a morphism $X' \dashrightarrow Y \to \Spec k$, then $Y$ must be a terminal Fano threefold with $\rho(Y) = 1$.  Studying the pseudo-index of $Y$ as in \cite{almostfano} and combining this with the fact that $K_{X'}^3 = -64$ would imply that $X'$ itself must have been $\PP^3$, so $X\cong \PP^3$.  If a run of the minimal model program on $X'$ terminates in a morphism $X' \dashrightarrow Y \to C$, where $C$ is a curve, the generic fiber of $Y \to C$ must be a smooth del Pezzo surface, so there are sufficiently general curves $L \subset X'$ such that $K_{X'} \cdot L = -3$ or $-2$, and if the termination is in a surface $W$, there are sufficiently general curves $L \subset X'$ such that $K_{X'} \cdot L = -2$.  If any of these curves miss the exceptional divisors of the partial resolution $\pi: X'\to X$, then $D\cdot \pi(L) \in \Z$, and we can argue as in the example above to show that $d$ must be even.  Similarly, we can reach the same conclusion if $D$ does not pass through the strictly canonical singularities of $X$.  The remaining case is when $D$ contains the strictly canonical singularities of $X$ and the general fiber $L$ intersects the exceptional divisors of $\pi: X' \to X$, because it is not obvious that $D \cdot \pi(L) \in \Z$.  However, we can explicitly understand the fibration when this occurs.  

First, let us recall results of Cutkosky on contractions of extremal rays on terminal, $\Q$-factorial Gorenstein threefolds.

\begin{lemma}\cite[Lemma 2]{cutkosky}
Suppose that $X$ is a terminal, $\Q$-factorial Gorenstein threefold.  Then, $X$ is factorial.
\end{lemma}

\begin{lemma}\cite[Lemma 3]{cutkosky}\label{lem:notsmall}
Suppose that $X$ is a terminal, $\Q$-factorial Gorenstein threefold and $\phi: X \to Y$ is the contraction of a $K_X$-negative extremal ray with at most one dimensional fibers.  Then, $Y$ is factorial.  In particular, $Y$ is a terminal, $\Q$-factorial Gorenstein threefold, and $\phi$ cannot be a small contraction.
\end{lemma}

\begin{theorem}\cite[Theorem 4]{cutkosky}\label{thm:curve}
Suppose that $X$ is a terminal, $\Q$-factorial Gorenstein threefold and $\phi: X \to Y$ is a birational contraction of a surface $W \subset X$ to a curve $C \subset Y$.  Then, $Y$ is smooth near $C$.
\end{theorem}

\begin{theorem}\cite[Theorem 5]{cutkosky}\label{thm:cutkosky}
Suppose that $X$ is a terminal, $\Q$-factorial Gorenstein threefold and $\phi: X \to Y$ is a birational contraction of a surface $W \subset X$ to a point $p \subset Y$.  Then, one of the four cases below occur:

    \begin{enumerate}[(i)]
    \item $Y$ is nonsingular near $p$, $W \cong \mathbb{P}^2$, and $\calO_W(W) \cong \calO_{\PP^2}(-1)$.
    \item $W \cong \PP^1 \times \PP^1$ and $\calO_W(W) \cong \calO_{\PP^1 \times \PP^1}(-1,-1)$.
    \item $W$ is isomorphic to a reduced, irreducible singular quadric surface $D$ in $\PP^3$ and $\calO_W(W) \cong \calO_{\PP^3}(-1)\otimes \calO_D$.
    \item $Y$ is singular at $p$, $W \cong \mathbb{P}^2$, and $\calO_W(W) \cong \calO_{\PP^2}(-2)$.
    \end{enumerate}
\end{theorem}

Now we can prove Theorem \ref{thm:nocanonical}.

\begin{proof}
Let us begin with the simplest case: no component of the locus of canonical singularities is contained in $D$.  Then, the contraction of a $K_{X'}$ negative extremal ray must be birational $X' \to Y$ or a Fano fibration $X' \to S$ or $X' \to C$, where $\dim S = 2$ or  $\dim C = 1$.  Because $-K_{X'}$ is nef and non-trivial, the contraction cannot be $X' \to \Spec k$.  Also, by Lemma \ref{lem:notsmall}, $X' \to Y$ is necessarily a divisorial contraction.  Therefore, in every case, the generic curve contracted has $K_{X'} \cdot C  = -1, -2$, or $-3$, so the image of $C$ on $X$ has $K_X \cdot \pi(C) = -1, -2$, or $-3$.  Because $D$ does not contain the locus of canonical singularities, for a sufficiently generic curve $C$, $D \cdot \pi(C) \in \mathbb{Z}$.  Therefore, the relationship $dK_X + 4D \sim 0$ implies $d$ is even. If a component $\Delta$ of the locus of canonical singularities is contained in $D$, we can separate into two cases: either $\Delta$ is one- or zero-dimensional.

\noindent \textbf{Case 1. $\mathbf{\dim \Delta = 1.}$}

If $\Delta$ is one-dimensional, consider the partial resolution $\pi: X' \to X$.  Because $X$ has only canonical singularities, the fibers of $\pi$ must be chains of rational curves.  We can study the pullback $\pi^*D$: in particular, $\pi^*D = \tilde{D} + \sum a_i F_i$, where $\tilde{D}$ is the strict transform of $D$ and $F = \bigcup F_i$ is the fiber over $\Delta$.  Let $F_0$ be a component of $F$ such that $\dim \pi(F_0) = 1$.  For a generic curve $C \subset F_0$ contracted by $\pi$, $K_{X'} \cdot C = 0$ and $F \cdot C < 0$.  However, $-2 = K_{F_0} \cdot C = (K_{X'} + F_0) \cdot C$, so there can be at most one component $F_i$ meeting $C$ with $F_i \cdot C = 1$.  Therefore, either there is no such $F_i$ and

\[ 0 = \pi^*D \cdot C = \tilde{D} \cdot C + \sum a_i F_i \cdot C = n + a_0 (-2), \]
for some $n \in \bZ$ so $a_0 \in \frac{1}{2}\Z$ or there is some $F_i$ that meets $C$ and a contracted curve $C' \subset F_i$ meeting $F_0$ so
\[ 0 = \pi^*D \cdot C = \tilde{D} \cdot C + \sum a_i F_i \cdot C = n + a_0 (-2) + a_i(1), \]
\[ 0 = \pi^*D \cdot C' = \tilde{D} \cdot C' + \sum a_i F_i \cdot C' = m + a_0(1) + a_i(-2),  \]
so $a_0, a_i \in \frac{1}{3}\Z$.  This shows that, for generic curves in $X$ meeting $D$, the intersection with $D$ is in $\frac{1}{6}\Z$.  With this in mind, now contract a $K_{X'}$ negative extremal ray on $X'$.  As above, we have the following options: (A) $X' \to Y$ is divisorial; (B) $X' \to S$ is a Fano fibration over a surface, or (C) $X' \to C$ is a Fano fibration over a curve. 

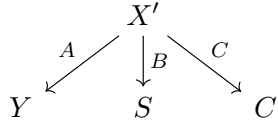
\textbf{Case 1A.}  Assume first that $X' \to Y$ is divisorial.  If $X' \to Y$ is divisorial and with at most one dimensional fibers, the generic fiber $C$ has $K_{X'} \cdot C = -1$, so the image in $X$ has $K_{X} \cdot \pi(C) = -1$.  For sufficiently generic $C$, $D \cdot \pi(C) \in \frac{1}{6}\Z$.  Therefore, the relationship $dK_X + 4D \sim 0$ implies $d$ must be even.  If $X' \to Y$ is divisorial but contracts a surface to a point, if any case other than $(i)$ occurs as in Theorem \ref{thm:cutkosky}, we still find a generic curve $C$ in the fiber with $K_{X'} \cdot C = -1$.   

If case $(i)$ occurs, the threefold $Y$ is still terminal, $\mathbb{Q}$-factorial, and Gorenstein, so we can contract a new $K_Y$ negative extremal ray and repeat.  If at any point our contraction one of the cases $(ii), (iii)$, or $(iv)$, by the same argument above, we are done.  If we perform a divisorial contraction with at most one-dimensional fibers, again the output is terminal, $\mathbb{Q}$-factorial and Gorenstein, so we can continue.  Therefore, it suffices to analyze the possible fibrations that arise as minimal models of a terminal, $\mathbb{Q}$-factorial, Gorenstein variety $X'$ where, at each step of the minimal model program, the resulting variety is also terminal, $\mathbb{Q}$-factorial, and Gorenstein.

However, after some number of divisorial contractions, we reach the point of a fibration, then the divisorial contractions were blow ups of some point(s) on the fibration.  Therefore, either the general fiber of the fibration doesn't intersect $F$, or after blowing up, a fiber of the divisorial contraction doesn't intersect $F$.  Therefore, its image on $X$ has $D \cdot C \in \mathbb{Z}$.  Arguing as above implies $d$ is even.  Therefore, the only two cases that remain to be studied are if the only possible $K_{X'}$ negative contraction yields a fibration.

\textbf{Case 1B.} If $\phi: X' \to S$ is a fibration with general fiber $\cong \mathbb{P}^1$, either there are $F$-trivial fibers $C$ or $F$ is relatively ample.  In the first case, $K_X \cdot \pi(C) = -2$ and $D \cdot \pi(C) \in \mathbb{Z}$, so $d$ be even.  Assume then that $F$ is relatively ample.  By \cite[Theorem 7]{cutkosky}, $S$ must be smooth and $X'$ must be a conic bundle over $S$.  If $X' \to S$ has any singular fibers, then there exist curves $C$ such that $K_{X'} \cdot C = -1$, and we argue as before to show $d$ must be even.  Therefore, we may assume every fiber is smooth and $X' \to S$ is a smooth $\mathbb{P}^1$-bundle over a smooth surface $S$.  Furthermore, by \cite[Lemma 2.5]{almostfano}, $-K_S$ is big and nef.  Because $F$ is relatively ample, for some component $F_0$ of $F$, the induced morphism $F_0 \to S$ must be finite.  However, $F_0$ is contractible on $X'$, so we have a diagram
\begin{center}
\begin{tikzcd}
	X' \arrow[r, "\pi"] \arrow[d, "\phi"] & X \\
	S & \\
\end{tikzcd}
\end{center}
Consider a smooth curve $C \subset S$ such that $ Z = \phi^{-1}(C)$ contains a contracted curve in $F_0$.  Because every fiber of $\phi$ is $\mathbb{P}^1$, $Z$ is a ruled surface over $C$ and because $F_0 \to S$ is finite, $F_0|_Z$ is a multisection of $\phi|_Z: Z \to C$.  However, this multisection is contractible in $Z$ to a surface $\overline{Z} \subset X$.  For generic $Z$, $Z$ is not contracted by $\pi$, so intersection theory on ruled surfaces implies that $F_0|_Z$ is actually a section.  This is true for any such $Z$, so the degree of $\phi|_{F_0}: F_0 \to S$ must be 1, hence $S\cong F_0$ and $F_0$ is a section of $\phi$.  

Assume first that $F_0$ is contracted to a curve via $\pi: X' \to X$.  Then, $S \cong F_0$ must be a ruled surface over $C$ with $-K_S$ big and nef, so $S$ must be $\PP^1 \times \PP^1$, $\mathbb{F}_1$, or $\mathbb{F}_2$.  Because each of these surfaces have $\rho(S) = 2$, it follows that $\rho(X') = 3$ and there are at most two components of $F$.  If there is only one component of $F$, there is exactly one contraction $X' \to X$ that is $K_{X'}$ trivial, but $\rho(X') \ge 3$ implies that there are at least two $K_{X'}$-negative contractions.  One corresponds to the map $\phi: X' \to S \cong F$ and the other must correspond to another case, so we use the argument in the other cases to find a contradiction or show $X \cong \PP(1,1,2,4)$.  

If there are two components of $F$, so $F = F_0 \cup F_1$, then $F_1$ must also be relatively ample as it is covered by $K_{X'}$-trivial curves so cannot be contracted by $\phi$.  Therefore, $S \cong F_0 \cong F_1$.  However, this is only possible if both $F_0$ and $F_1$ are contracted to a curve via $\pi$; otherwise, say $F_1$ is contracted to a point, then there exist $F_1$ trivial curves intersecting $F_0$, so $F_1$ is not relatively ample.  Now, as above, consider a smooth curve $C \subset S$ such that $ Z = \phi^{-1}(C)$ contains a contracted curve in $F_0$.  Because every fiber of $\phi$ is $\mathbb{P}^1$, by the argument above, $Z$ is a ruled surface over $C$ with a contractible section.  However, $Z$ must also contain a contracted curve in $F_1$, hence $Z$ has two contractible sections.  However, this is a contraction, as it would imply $\pi|_Z: Z \to \pi(Z)$ contracts $Z$ to a curve.   

Therefore, we can assume that $F_0$ is contracted to a point via $\pi: X' \to X$.  If each $\phi$-ample divisor $F_i$ does not intersect $\tilde{D}$, we find curves $C$ such that $D\cdot C \in \Z$ and $K_X \cdot C = -2$ so $d$ is even.  Therefore, it suffices to consider only $F_i$ that intersect $\tilde{D}$.  Because $D$ contains $\Delta$, there is some $F_1$ that is contacted to a curve via $\pi$ such that $F_1 \cap D \ne \emptyset$ and $F_1 \cap F_0 \ne \emptyset$.  The intersection $F_1 \cap F_0$ must be a fiber of the ruled surface $F_1$, hence $F_0$ contains a curve $C$ such that $K_{X'} \cdot C = 0$, $F_1 \cdot C = -2$, and $F_0 \cdot C = 0$.  Therefore, $K_{F_0} \cdot C = 0$ so $F_0 \cong \mathbb{F}_2$.  

This implies that $\rho(X') = 3$ because $\phi: X' \to S \cong \mathbb{F}_2$ is a $\PP^1$ bundle, hence there is only one exceptional divisor $F_1$ with $\dim \pi(F_1) = 1$.  If $F_1$ were also $\phi$-ample, we must have $F_1 \cong \mathbb{F}_2$ also be a section.  However, the intersection curve $C = F_0 \cap F_1$ is a section of $F_0$ but a fiber of $F_1$, a contradiction.  Therefore, $F_1$ is not $\phi$-ample, so we must have $F_1 \cong \PP^1 \times \PP^1$.  Contracting $F_0$ and $F_1$ to $X$ shows $\rho(X) = 1$ and $X$ has a $\frac{1}{4}(1,1,2)$ singularity, hence we must have $X \cong \PP(1,1,2,4)$.

\textbf{Case 1C.} If $\phi: X' \to C$ is a fibration with general fiber a smooth del Pezzo surface and $C \cong \mathbb{P}^1$, we can first note that if the general fiber is a surface other than $\mathbb{P}^2$ or $\mathbb{P}^1 \times \mathbb{P}^1$, there exist curves $C$ with $K_{X'} \cdot C = -1$, so we argue as before to conclude $d$ is even.  Similarly, if the fiber is $\mathbb{P}^2$, there exist curves $C$ with $K_{X'} \cdot C = -3$, and again we can conclude $d$ is even.  Therefore, it suffices to analyze the case when the general fiber is $\mathbb{P}^1 \times \mathbb{P}^1$.  Because $\rho(C) = 1$ and $\phi$ was an extremal contraction, $\rho(X') = 2$.  There is then only one component of $F$.  We would like to show that there are divisors $D_1$ and $D_2$ whose restriction to each fiber are the different rulings.  Those are not linearly equivalent nor are they linearly equivalent to the general fiber $F$, hence it would imply $\rho(X') \ge 3$, a contradiction.

Suppose for contradiction $X'$ does exist.  If $F$ was contained in a fiber of $\phi$, then there exist many $F$-trivial curves with $K_{X'} \cdot C = -2$, and on $X$, $D \cdot \pi(C) \in \Z$.  As usual, we consider the relation $dK_X + 4D \sim 0$, so find that $d$ must be even.  Now consider the case that $F$ is $\phi$-ample, so $\phi|_F: F \to C$ gives $F$ the structure of a ruled surface over $C$ and contracts only $K_{X'}$-negative curves.  Because $\pi|_F$ also contracts $F$ to a curve but contracts only $K_{X'}$-trivial curves, $F$ must have the structure of a product, so $F \cong \PP^1 \times \PP^1$.

Let $\Gamma \cong \PP^1 \times \PP^1$ be a fiber of $\phi$.   We claim that $F$, $\Gamma$, and $K_{X'}$ in $N^1(X)$ are independent so we must have $\rho(X') \ge 3$, a contradiction.  To see the claim, note that $\Gamma|_F$ must be a ruling of $F$, so $\Gamma|_F \in | \calO_F(1,0) |$.  Next, observe that $K_{X'}|_F$ is negative on the fibers contracted by $\phi$ and trivial on the fibers contracted by $\pi$.  However, these are the two rulings of $F$, so $K_{X'}|_F \in | \calO_F(-2,0) |$.  Furthermore, $K_{X'}$ and $\Gamma$ are certainly not linearly equivalent.  Finally, consider $F|_F$.  On the fibers of $F$ contracted by $\pi$, by the negativity lemma, this must be negative, so $F|_F \in |\calO_F(a, -b)|$ for $b >0$.  Therefore, $F$ cannot be linearly equivalent to any linear combination of $\Gamma$ and $K_{X'}$, so $\rho(X') \ge 3$, so $X'$ cannot exist.

\noindent \textbf{Case 2. $\mathbf{\dim \Delta = 0.}$}

Finally, suppose the locus of log canonical singularities is a point contained in $D$.  We can study the same contractions: (A) $X' \to Y$ is divisorial; (B) $X' \to S$ is a Fano fibration over a surface, or (C) $X' \to C$ is a Fano fibration over a curve.  In this case, if $F$ is the exceptional locus of the map $\pi: X' \to X$, the curves in $F$ are all $K_{X'}$ -trivial, so none can be contracted by a $K_{X'}$-negative contraction.  Therefore, the third arrow (Case C) $X' \to C$ is not possible.

\textbf{Case 2A.} Because the curves in $F$ are all $K_{X'}$ -trivial, the only possible $K$-negative divisorial contraction over $F$ is $X' \to Y$ that has at most one dimensional fibers.  However, then $Y$ would be terminal, Gorenstein, and $\mathbb{Q}$-factorial, so we can continue the minimal model program on $Y$.  Much of this argument is the same as Case A above.  If divisorial contractions happen first, there will exist curves with $K_{X'} \cdot C$ equal to $-1, -2, $ or $-3$ that don't intersect $F$, and (invoking factoriality of $X'$), $D \cdot \pi(C) \in \mathbb{Z}$, and we can conclude $d$ is even.

\textbf{Case 2B.} The remaining case is if the only $K_{X'}$ negative contraction is a fibration $X' \to S$, and as in Case B above, we can assume every fiber is smooth and isomorphic to $\mathbb{P}^1$.  Therefore, we find ourselves in the situation where $\phi: X' \to S$ is a smooth $\mathbb{P}^1$-bundle over a smooth surface $S$ and by \cite[Lemma 2.5]{almostfano}, $-K_S$ is big and nef.  Exactly as above, we can conclude $S \cong F_0$ for some component $F_0$ of $F$ and $F_0$ is a section of $\phi$.  Furthermore, for any component $F_i$ of $F$, because $\pi(F_i)$ is a point, $K_{X'}\cdot C_i = 0$ for any $C_i \subset F_i$, so $C_i$ cannot be contracted by $\phi$.  Therefore, every component $F_i$ of $F$ is $\phi$-ample and $S \cong F_i$ for all $i$. 

Briefly turning our attention to the map $\pi: X' \to X$, because components of $F$ are contracted to points by $\pi$, every curve in $F$ is $K_{X'}$-trivial and $F$-negative.  By adjunction, for $C\in F_i$, $(K_{X'} + F_i) \cdot C = K_{F_i} \cdot C$, so $F_i \cdot C = K_{F_i} \cdot C$.  Therefore, not only is $-K_S = -K_{F_i}$ big and nef, but it is ample, so $S \cong F_i$ is a Fano surface.  If there are any $-1$ curves on $F$, taking the intersection product with $\pi^*D = \tilde{D} + aF$ implies $a \in \mathbb{Z}$, so for any curve $C$ on $X$, $D \cdot C \in \Z$.  Therefore, for a fiber of $\phi$ with $K_{X'} \cdot C = -2$, we find that $d$ must be even.  Similarly, if $F_i \cong \PP^2$, we find lines with $F \cdot C = -3$, so $D \cdot C \in \frac{1}{3}\Z$ and the same conclusion holds.

Therefore, the only remaining case is if $F_i \cong \PP^1 \times \PP^1$ for all $i$:
\begin{center}
	\begin{tikzcd}
		X' \arrow[r, "\pi"] \arrow[d, "\phi"] & X \\
		\PP^1 \times \PP^1 & \\
	\end{tikzcd}
\end{center}
and $\pi: X' \to X$ contracts $F$.  Consider $Z = \phi^{-1}(C)$ for a generic ruling $C$ on $\PP^1 \times \PP^1$.  By construction, each $F_i|_Z$ is a contractible section of the smooth ruled surface $Z$.  Because $Z$ is not contracted by $\pi$, this implies that there is only one $F_i$ and $F = F_0$.  Then, $X$ is locally isomorphic to the cone over the anticanonically embdedded $\mathbb{P}^1 \times \mathbb{P}^1$, Example \ref{ex:O(2,2)}.  However, $\rho(X) = 1$, so $X$ must actually be isomorphic to that cone.
\end{proof}

Ultimately, the odd degree pairs are behaving in a very special way: oddness of the degree is forcing constraints on the threefolds $X$ that can appear.  Summarizing the previous two sections, no log canonical threefolds $X$ can appear and there are only three possibilities if the threefold $X$ has canonical singularities.

\subsection{A divisor in the moduli space}
By a simple dimension count, we obtain the following.

\begin{proposition}
    Let $X$ be the cone over the anticanonical embedding of $\PP^1 \times \PP^1$, considered as a hypersurface in $W = \PP(1,1,1,1,2)$ (c.f. \ref{ex:1124}), so $\calO(K_X) = \calO_W(-4)|_X$.  If the pair $\left( X, \frac{4}{d} D \right)$ is log terminal for the general member $D$ in the linear system $|\calO_X(d)|$, then there is a divisor $\calD_Q$ in the moduli space $\calM_{\dfours}$ parameterizing surfaces on $X$.  
\end{proposition}

To show that surfaces on $X$ appear as a divisor in the moduli space for any degree $d$, we show that the general member $D \in |\calO_X(d)|$ is such that $(X, \frac{4}{d} D)$ has log terminal singularities.  In fact, because the log canonical threshold is upper semicontinuous, it suffices to show this for a particular member $D \in |\calO_X(d)|$.  For even degree $d$, we can find a smooth member $D \in |\calO_X(d)|$ missing the unique singular point of $X$, and because $X$ has canonical singularities, certainly $(X, \frac{4}{d} D)$ is log terminal.  For odd degree $d$, consider $D_0 = D_1 \cup D_{(d-1)/2}$, where $D_1 \in |\calO_X(1)|$ and $D_{(d-1)/2} \in |\calO_X((d-1)/2)|$ are general members meeting transversally.  $D_1$ has a unique $\frac{1}{4}(1,1)$ singularity at the vertex of $X$ and $D_2$ is smooth, missing the vertex of $X$, and a computation shows $(X, \frac{4}{d} D_0)$ is log terminal. Therefore, we have proven the following. 

\begin{proposition}\label{prop:Xdivisor}
    Let $X$ be the cone over the anticanonical embedding of $\PP^1 \times \PP^1$.  For any degree $d > 4$, there is a divisor $\calD_Q$ in the moduli space $\calM_{\dfours}$ parameterizing surfaces on $X$.  Furthermore, by Proposition \ref{prop:smoothdivisor}, $\calM_{\dfours}$ is smooth generically along $\calD_Q$. 
\end{proposition}

In higher dimensions, the same proof gives the existence of a divisor in the moduli space of $(d,n+1)$ $\PP^n$-smoothable \he stable pairs. 

\begin{theorem}
Let $X_n \subset \PP(1^{n+1}, 2)$ be the cone over the degree $2$ embedding of a smooth quadric $Q \in | \calO_{\PP^n}(2) |$.  For any degree $d > n+1$, there is a divisor $\calD_Q$ in the moduli space $\calM_{\PP^n\s,(d,n+1)}$ parameterizing hypersurfaces on $X_n$.  
\end{theorem}

In fact, we can describe the singularities of the generic degree $d$ hypersurface on $X_n$.  

\begin{proposition}\label{prop:DQ}
 For a generic point $(X_n,D)$ on the divisor $\calD_Q$, where $X_n \subset \PP(1^{n+1}, 2)$ is the cone over the degree $2$ embedding of a smooth quadric $Q \subset \PP^n$, we can describe $D$ as follows.
    \begin{enumerate}
        \item If $d$ is even, $D$ is smooth.
        \item If $d$ is odd, $D$ passes through the vertex of the cone $X_n$ and is locally isomorphic to the vertex of the cone $X_{n-1}$. 
    \end{enumerate}
\end{proposition}

\begin{remark}
For $(d,3)$ $\PP^2$-smoothable \he stable pairs, there is a divisor parameterizing curves on $X_2 = \PP(1,1,4)$, the cone over the degree $2$ embedding of a conic in $\PP^2$.  If $d$ is even, the curves are smooth and miss the singular point, and if $d$ is odd, the curves are nodal at the singular point.  The node can be interpreted as the singularity of a cone over two points, $X_1$.  For three dimensional pairs, there is a divisor parameterizing surfaces on $X_3$, the cone over the anticanonical embedding of $\PP^1 \times \PP^1$.  If $d$ is even, again the generic surface misses the singular point of $X_3$, but if $d$ is odd, the generic surface has a singularity of type $\frac{1}{4}(1,1)$ at the vertex of $X_3$, matching the singularity type of $X_2$.  Proposition \ref{prop:DQ} shows that this behavior persists in all dimensions. 
\end{remark}

\subsection{Log terminal degenerations of $\mathbb{P}^n$}\label{sec:lt}

To completely classify normal varieties appears as $\dfours$ \he stable pairs when $d$ is odd, it remains to understand log terminal threefolds $X$ that are degenerations of $\mathbb{P}^3$.  We will approach this in general (not only in the $d$ odd case) and state results for $\PP^n$ whenever possible.  There are natural log terminal varieties to consider: weighted projective spaces.  We summarize the case in dimension 2, due to Hacking and Prokhorov.

\begin{theorem}\cite{smoothdp}
If $X$ is a normal, log terminal degeneration of $\mathbb{P}^2$ such that the total space is $\mathbb{Q}$-Gorenstein, then $X \cong \PP(p^2, q^2, r^2)$ or a smoothing of such a space, where \[ 3pqr = p^2 + q^2 + r^2. \]  Futhermore, all such varieties admit a $\mathbb{Q}$-Gorenstein smoothing to $\mathbb{P}^2$.
\end{theorem}

In addition to the theorem, we can describe all solutions with an infinite graph:

\begin{theorem}\label{thm:mutation}
All solutions to \[ 3pqr = p^2 + q^2 + r^2 \] can be obtained by starting with the obvious solution $(1,1,1)$ and performing a sequence of mutations: if $(p,q,r)$ is a solution, then $(p,q,3pq-r)$ is a solution.
\end{theorem}

One could hope for an analogue of the description of degenerations in the three dimensional case, although that seems far out of reach.  However, there are partial results, using properties of weighted projective spaces. Recall that a weighted projective space $\PP(a_0, \dots, a_n)$ is called well-formed if every subset of $n$ of the $a_i$ has no common factors.  We will call the set of integers $(a_0, \dots, a_n)$ well-formed if the associated weighted projective space is. 

\begin{proposition}
If $\PP(a,b,c,d)$ is well-formed and admits a $\mathbb{Q}$-Gorenstein smoothing to $\PP^3$, then
\[ 64abcd = (a + b + c + d)^3. \]
\end{proposition}

\begin{proof}
In order for $X = \PP(a,b,c,d)$ to have a $\Q$-Gorenstein smoothing, necessarily $K_X^3 = K_{\PP^3}^3 = -64$ by Proposition \ref{prop:constantvolume}.  But, $\calO(K_X) = \calO(-a-b-c-d)$, and
\[ K_X^3 = \frac{(-a-b-c-d)^3}{abcd}. \]
Therefore, we must have $64abcd = (a + b + c + d)^3$.
\end{proof}

One could make the immediate generalization to $n$-dimensional weighted projective spaces:

\begin{proposition}
If $\PP(a_0, a_1, \dots , a_n)$ is well-formed and admits a $\mathbb{Q}$-Gorenstein smoothing to $\PP^n$, then
\[ (n+1)^n\Pi a_i = (\sum a_i)^n. \]
\end{proposition}

\begin{remark}
While this formula appears different from the formula in Theorem \ref{thm:mutation}, the equation $9abc = (a+b+c)^2$ simplifies to $3pqr = p^2 + q^2 + r^2$ (where $a = p^2, b = q^2, c = r^2$) when $(a,b,c)$ is well formed. 
\end{remark}

Given the equation, one must ask how to find solutions.  There is certainly an infinite family of solutions, which makes sense geometrically.  If we have a degeneration of $\mathbb{P}^2$ to such a weighted projective space, it should induce a degeneration of $\mathbb{P}^3$ to an appropriate cone over that weighted projective space.  This is the content of the following proposition, stated first in the three-dimensional case and then in general, whose proof is simple arithmetic.

\begin{proposition}\label{prop:abc}
If $\PP(a,b,c)$ admits a smoothing to $\PP^2$, then $d =\sqrt{abc} = \frac{a + b + c}{3} \in \mathbb{Z}$ and $P(a,b,c,d)$ satisfies the condition \[ 64abcd = (a + b + c + d)^3. \]
\end{proposition}

\begin{example}
    The first non-trivial example of a solution is $\bP(1,1,2,4)$.  Because this can be embedded as a degree two section of $\PP(1,1,1,1,2)$, it clearly admits a smoothing to $\bP^3$.
\end{example}

We have the immediate generalization to $\bP^n$:

\begin{proposition}
If $\PP(a_0, a_1, \dots , a_n)$ satisfies
\[ (n+1)^n\Pi a_i = (\sum a_i)^n,  \]
then $b = (\Pi a_i)^{1/n} = \frac{\sum a_i}{n+1} \in \mathbb{Z}$ and $\PP(a_0, a_1, \dots , a_n, b)$ satisfies \[ (n+2)^{n+1}b\Pi a_i = (b+\sum a_i)^{n+1}.  \]
\end{proposition}

Let us further investigate the three-dimensional case.  Using a computer, one can list the integer solutions to the equation $64abcd = (a + b + c + d)^3$ such that the associated weighted projective space is well-formed, and finds the following weighted projective spaces as the first few solutions.

\begin{center}
\begin{tabular}{lll}
	$\PP(1,1,1,1)$ &	$\PP(1,1,2,4)$ & $\PP(1,2,9,12)$ \\
	 $\PP(1,4,10,25)$ & $\PP(1,4,16,27)$ & $\PP(1,6,9,32)$ \\
	$\PP(1,7,27,49)$ & $\PP(1,9,50,60)$ & $\PP(1,22,32,121)$\\
\end{tabular}
\end{center}

We can immediately determine that some of these threefolds do not admit smoothings to $\PP^3$, using the following theorem of Schlessinger \cite[Theorem 3]{schlessinger}. 

\begin{theorem}\cite{schlessinger}\label{thm:schlessinger}
Assume $Y$ is smooth of dimension $\ge 3$, $G$ is a finite group, and $ X = Y / G$.  Let $p: Y \to X$ be the quotient map.  If $y \in Y$ is the only fixed point of $G$, then $X$ is rigid. 
\end{theorem}

In the list above, it implies that neither $\PP(1,4,16,27)$ nor $\PP(1,7,27,49)$ are smoothable, so we see that, in dimension $\ge 3$, satisfying the equation is not a sufficient condition to admit a smoothing (contrasting with the case for dimension 2). For the remaining candidates, in light of Proposition \ref{prop:abc}, there are some solutions $\PP(a,b,c,d)$ arising from the degenerations $\PP(a,b,c)$ of $\PP^2$ where $ d = \frac{a+b+c}{3}$ is the average of $a$, $b$, and $c$.  The only three that appear in this truncated list are $\PP(1,1,1,1)$, $\PP(1,1,2,4)$, and $\PP(1,4,10,25)$.  These solutions are well understood and, following work of Hacking, Prokhorov, and Manetti, we have the following result.  This is simply a restatement of Theorem \ref{thm:mutation}, adding in the fourth variable $d$. 

\begin{proposition}\label{prop:tree1}
	There is an infinite family of well-formed solutions to the equation $64abcd = (a + b + c + d)^3$ given by $(a,b,c,d) = (\alpha^2, \beta^2, \gamma^2, \alpha\beta\gamma) = (\alpha^2, \beta^2, \gamma^2, \frac{\alpha^2 + \beta^2 + \gamma^2}{3})$.  All such $\alpha$, $\beta$, and $\gamma$ lie on an infinite tree and are obtained by a mutation of the form $(\alpha, \beta, \gamma) \to (\alpha, \beta, 3\alpha\beta - \gamma)$ starting from $(1,1,1)$. 
\end{proposition}

\begin{definition}
	We will call a solution of this form $\PP^2$-type because it arises from a degeneration of $\PP^2$.
\end{definition}

The deformation theory of these weighted projective spaces is in general quite complicated, but we can show that all solutions of $\bP^2$-type admit a smoothing to $\bP^3$.

\begin{proposition}\cite{hackinggentype}\label{prop:p2}
	The weighted projective spaces appearing as solutions of $\PP^2$-type can be connected as a family of threefolds over a two-parameter base, and are each $\mathbb{Q}$-Gorenstein deformations of a common smoothing. 
\end{proposition}

\begin{proof}
	This is proved in \cite[Example 7.7]{hackinggentype}.  We relate the weighted projective spaces one step apart on the infinite tree over a two-parameter base. Let $\PP(a,b,c,d)$ and $\PP(a,b,c',d')$ be two solutions to $64abcd = (a + b + c + d)^3$ of $\PP^2$-type related by one mutation so that \[ \PP(a,b,c,d) = \PP(\alpha^2, \beta^2, \gamma^2, \alpha\beta\gamma) \] and \[ \PP(a,b,c',d') = \PP(\alpha^2, \beta^2, \gamma'^2, \alpha\beta\gamma' ) = \PP(\alpha^2, \beta^2, (3\alpha\beta - \gamma)^2, \alpha\beta(3\alpha\beta - \gamma) ). \]
	
	Using the fact that $3\alpha\beta\gamma = \alpha^2 + \beta^2 + \gamma^2$ (and similarly for $\gamma'$), we can form the two-parameter family \[ \calX :  x_0x_1 = sx_2^{\gamma'} + tx_3^{\gamma} \subset \PP(\alpha^2, \beta^2, \gamma, \gamma', \alpha\beta) \times \Aff^2_{s,t} \] of weighted degree $\alpha^2 + \beta^2 = \gamma\gamma'$ threefolds in $\PP(\alpha^2, \beta^2, \gamma, \gamma', \alpha\beta).$
	
	When $s = t = 0$, we get a non-normal threefold $\PP(\alpha^2, \gamma, \gamma', \alpha\beta) \cup \PP(\beta^2, \gamma, \gamma', \alpha\beta)$.  
	
	When $s = 0$ but $t \ne 0$, we get $\PP(\alpha^2, \beta^2, \gamma^2, \alpha\beta\gamma) $ via the degree $\gamma$ embedding \[ \PP(\alpha^2, \beta^2, \gamma^2, \alpha\beta\gamma) \to (x_0x_1 = tx_3^{\gamma}) \subset \PP(\alpha^2, \beta^2, \gamma, \gamma', \alpha\beta) \] given by \[ (u,v,w,t) \mapsto (x_0,x_1,x_2,x_3,x_4) = (u^{\gamma}, v^{\gamma}, w, uv, t ). \]
	
	When $s \ne 0$ but $t = 0$, we get $\PP(\alpha^2, \beta^2, \gamma'^2, \alpha\beta\gamma') $ via the degree $\gamma'$ embedding \[ \PP(\alpha^2, \beta^2, \gamma'^2, \alpha\beta\gamma') \to (x_0x_1 = sx_2^{\gamma'}) \subset \PP(\alpha^2, \beta^2, \gamma, \gamma', \alpha\beta) \] given by \[ (u,v,w,t) \mapsto (x_0,x_1,x_2,x_3,x_4) = (u^{\gamma'}, v^{\gamma'}, uv, w, t ). \]
	
	Finally, for $s \ne 0$ and $t \ne 0$, we get a smoothing of the singularities of index $c$ and $c'$, respectively.  Because this is taking place as a complete intersection in weighted projective space, which is $\Q$-factorial, the total space of these smoothings is $\Q$-Gorenstein. 
\end{proof}

\begin{remark}
    Because $\PP^3$ is the `linear cone' over the anticanonically embedded $\PP^2$, it makes sense that `cones' (the weighted projective spaces $\PP(\alpha^2, \beta^2, \gamma^2, d)$) over degenerations of $\PP^2$ are appearing as degenerations of $\PP^3$.  Analogously, the equation $4\alpha\beta\gamma = \alpha^2 + \beta^2 + 2\gamma^2$ parameterizes weighted projective spaces $\PP(\alpha^2, \beta^2, 2\gamma^2)$ that appear as degenerations of $\PP^1 \times \PP^1$ \cite[Theorem 1.2]{smoothdp}.  Because $\PP^3$ is a smoothing of the cone over the anticanonical embedding of $\PP^1 \times \PP^1$, it makes sense that `cones' (weighted projective spaces $\PP(\alpha^2, \beta^2, 2\gamma^2, d)$) over degenerations of $\PP^1 \times \PP^1$ should be appearing as degenerations of $\PP^3$.  This allows us to find a second infinite tree of solutions to the equation.
\end{remark}

\begin{proposition}
	There is an infinite family of well-formed solutions to the equation $64abcd = (a + b + c + d)^3$ given by $(a,b,c,d) = (a,b,c, a+b+c)$.  All such $(a,b,c,d)$ lie on an infinite tree and are obtained by a mutation of the form $(a,b,c,d) \to (a,b,8ab-a-b-d,8ab-d)$ starting from $(1,1,2,4)$.
\end{proposition}

\begin{proof}
	If $a + b + c = d $, the equation $64abcd = (a+ b+ c+d)^3$ simplifies to $8abc = (a+b+c)^2$.  If desired, one can simplify this further by showing $a = \alpha^2$, $b = \beta^2$, and $c = 2\gamma^2$ so the equation becomes $4\alpha\beta\gamma = \alpha^2 + \beta^2 + 2\gamma^2$.  These parametrize weighted projective degenerations of $\bP^1 \times \bP^1$, and we use the mutation process in \cite[Theorem 1.2]{smoothdp} to obtain the result.
\end{proof}

In the list above, one sees that $\PP(1,1,2,4)$, $\PP(1,2,9,12)$, and $\PP(1,9,50,60)$ are all of this form.

\begin{definition}
	We will call a solution of this form $Q$-type because it arises from a degeneration of the quadric. 
\end{definition}

One can prove a simple lemma showing that $(1,1,2,4)$ is the only overlap between the types. 

\begin{lemma}
	The only solution to the equation $64abcd = (a + b + c + d)^3$ that is both of $\PP^2$-type and $Q$-type is $(1,1,2,4)$.  
\end{lemma}

As in the case of solutions of $\PP^2$-type, we can relate two weighted projective spaces of $Q$-type that are one mutation apart.  

\begin{proposition}\label{prop:sum}
	Given two weighted projective spaces that are solutions of $Q$-type one mutation apart, there is a two-parameter $\Q$-Gorenstein family connecting them and each are $\mathbb{Q}$-Gorenstein deformations of a common smoothing.
\end{proposition}

\begin{proof}
	Let $(a,b,c,d)$ be the first solution and $(a,b,c',d') = (a,b,8ab-a-b-d,8ab-d)$ be the second.  Without loss of generality, assume $d < d'$.  First, because $d = a + b + c$, we have $8abc = (a + b + c)^2 = d^2$.  Then, observe that $d(a+b) = d(d - c) = d^2 - dc = 8abc - dc = c(8ab-d) = cd'$, hence $a+b = \frac{cd'}{d}$, and similarly, $a+b = \frac{c'd}{d'}$.  Using this relationship repeatedly, we can form the desired family.
	
	Then, we can consider the family \[ \calX : x_0x_1 = tx_2^{a+b} + sx_3^{c} \subset \PP(ac,bc,c,a+b,d) \times \Aff^2_{s,t}. \]  When $s = 0 $ and $t = 0$, this is a non-normal threefold $\PP(ac,c,a+b,d) \cup \PP(bc,c,a+b,d)$.
	
	For $t =0$ but $s \ne 0$, this is the image of the degree $c$ embedding of \[ \PP(a,b,c,d) \to \PP(a,b,c,a+b,d) \] given by \[ (x,y,z,w) \mapsto (x_0,x_1,x_2,x_3,x_4) = (x^c, y^c, z, xy, w). \]  
	When $s = 0$ but $t \ne 0$, this is the image of the degree $a+b$ embedding of \[ \PP(a,b,c',d') \to \PP(a,b,c,a+b,d) \] given by \[ (x,y,z,w) \mapsto (x_0,x_1,x_2,x_3,x_4) = (x^{a+b}, y^{a+b}, xy, z, w). \]
	
	Finally, for $s \ne 0$ and $t \ne 0$, this gives a partial smoothing of the singularities of index $c$ and $d$ and $c'$ and $d'$.   Because the total space is a complete intersection in weighted projective space, it is $\Q$-Gorenstein
\end{proof}

\begin{remark}
 Although one could write the smoothings in Propositions \ref{prop:p2} and \ref{prop:sum} over a one-parameter base, the family over the two-dimensional base shows how to degenerate each pair of normal threefolds to a non-normal threefold, which can also appear in the moduli problem. 
\end{remark}

\begin{remark}
	The fact that there are two essentially distinct families of solutions to the equation $64abcd = (a + b + c + d)^3$ indicates the increase in complexity when studying degenerations of $\PP^3$ versus those of $\PP^2$.  Although we do not know smoothability for all solutions, Propositions \ref{prop:p2} and \ref{prop:sum} show that weighted projective spaces of $\PP^2$ or $Q$ type are smoothable to $\PP^3$.  
\end{remark}

Furthermore, based on preliminary computation we make the following conjecture.  

\begin{conjecture}
The only weighted projective spaces $\PP(a,b,c,d)$ that admit $\mathbb{Q}$-Gorenstein smoothings to $\PP^3$ are solutions of $\PP^2$ type or $Q$ type. 
\end{conjecture}

\subsection{Examples of log terminal degenerations and their smoothings}

\begin{example}\label{ex:1124}
    Let $X$ be the cone over the anticanonically embedded $\PP^1 \times \PP^1$.  In other words, $X$ is the cone over the anticanonical embedding of the quadric surface in $\PP^3$.  By construction, $X$ is a hyperplane section of $\PP(1,1,1,1,2)$, the cone over the anticanonical embedding of $\PP^3$.  However, we could apply the same construction to the cone over the anticanonical embedding of the singular quadric $(xy- z^2 = 0) \subset \PP^3$ to realize the cone $\PP(1,1,2,4)$ as another hyperplane section of $\PP(1,1,1,1,2)$.  Taking an appropriate pencil of these hyperplanes, we realize $X$ as a $\Q$-Gorenstein smoothing of $\PP(1,1,2,4)$.  
\end{example}

\begin{example}\label{ex:X26}
    The weighted projective space $\PP(1,4,10,25)$ admits at least seven different partial smoothings that all admit smoothings to $\PP^3$.  For the construction of the smoothings and a discussion of the applications to the study of moduli of quintic surfaces, see Section \ref{sec:quintics}.  Here, we draw a rough sketch of some smoothings.  First, we observe that $\PP(1,4,10,25)$ has singular locus $\PP^1 \cup \PP^1$.  At the general point of the first component $\PP^1$, it is isomorphic to $\frac{1}{2}(1,1) \times \Aff^1$.  At the general point of the second component $\PP^1$, it is isomorphic to $\frac{1}{5}(1,4) \times \Aff^1$. This has canonical singularities at all points except the unique $\frac{1}{25}(1,4,10)$ singularity.  There is a partial smoothing $Y_{26}$ that smooths the $\frac{1}{4}(1,1,2)$ singularity to a singularity of type $v$, isomorphic to the vertex of the cone over the anticanonical embedding of $\PP^1 \times \PP^1$ (as in Example \ref{ex:1124}). There is a different partial smoothing $W_{26}$ of $\PP(1,4,10,25)$ that partially smooths the $\frac{1}{25}(1,4,10)$ singularity to a non-isolated singularity $q$ isomorphic to the quotient of $ab - c^3d = 0 \subset \mathbb{A}^4$ by the $\mu_5$ action $(a,b,c,d) \sim (\zeta_5 a, \zeta_5^4 b, \zeta_5 c, \zeta_5^2 d) $. We can further smooth $W_{26}$ to a threefold $Z_{26}$ that has an isolated singularity $p$ in place of $q$.  This singularity $p$ is the quotient of the isolated singularity $ab - c(d^2+dc^2+c^4) = 0 \subset \mathbb{A}^4$, a perturbation of a cone over a $D_6$ singularity, by the same $\mu_5$ action.  
    
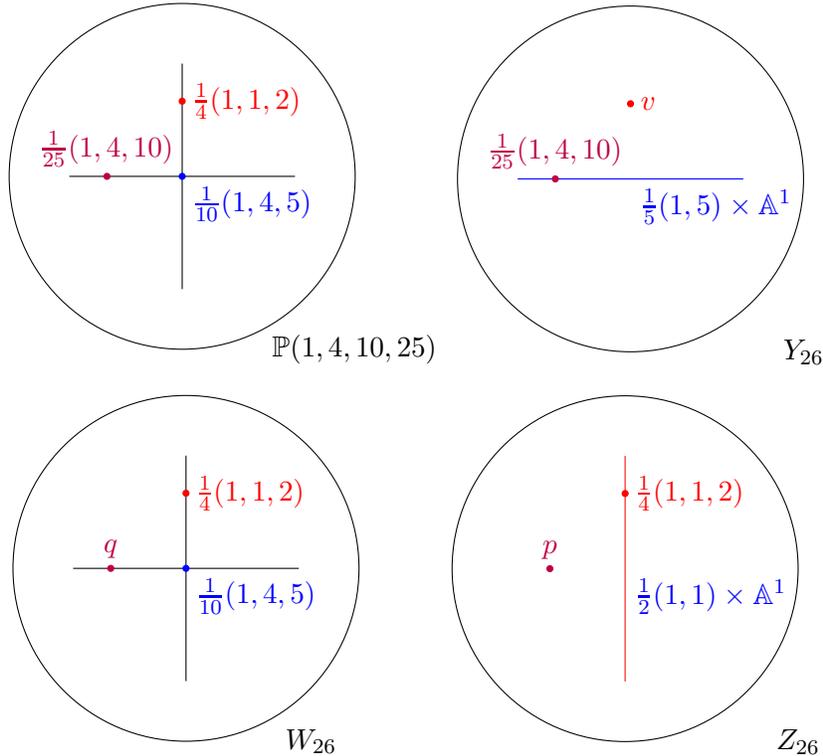
\begin{center}
\begin{tikzpicture}[scale=.8,font=\small]
\draw (0,0) circle [radius=2.25];
\draw (0,-1.5) -- (0,1.5);
\draw (-1.5,0) -- (1.5,0); 
\draw [fill=black] (0,0) circle [radius=.06]; 
\node [below right] at (0,0) {$\frac{1}{10}(1,4,5)$};
\draw [fill=black] (0,1) circle [radius=.06]; 
\node [right] at (0,1) {$\frac{1}{4}(1,1,2)$};
\draw [fill=black] (-1,0) circle [radius=.06]; 
\node [above] at (-1.2,0) {$\frac{1}{25}(1,4,10)$};
\node at (0,-2.6) {$\PP(1,4,10,25)$};
\end{tikzpicture}
\begin{tikzpicture}[scale=.8,font=\small]
\draw (0,0) circle [radius=2.25];
\draw (-1.5,0) -- (1.5,0); 
\node [below right] at (-.2,0) {$\frac{1}{5}(1,4) \times \Aff^1$};
\draw [fill=black] (-1,0) circle [radius=.06]; 
\draw [fill=black] (0,1) circle [radius=.06]; 
\node [right] at (0,1) {$v$};
\node [above] at (-1,0) {$\frac{1}{25}(1,4,10)$};
\node at (0,-2.6) {$Y_{26}$};
\end{tikzpicture}
\begin{tikzpicture}[scale=.8,font=\small]
\draw (0,0) circle [radius=2.25];
\draw (0,-1.5) -- (0,1.5);
\draw (-1.5,0) -- (1.5,0); 
\draw [fill=black] (0,0) circle [radius=.06]; 
\node [below right] at (0,0) {$\frac{1}{10}(1,4,5)$};
\draw [fill=black] (0,1) circle [radius=.06]; 
\node [right] at (0,1) {$\frac{1}{4}(1,1,2)$};
\draw [fill=black] (-1,0) circle [radius=.06]; 
\node [above] at (-1,0) {$q$};
\node at (0,-2.6) {$W_{26}$ \textcolor{white}{blahhh}};
\end{tikzpicture}
\begin{tikzpicture}[scale=.8,font=\small]
\draw (0,0) circle [radius=2.25];
\draw (0,-1.5) -- (0,1.5);
\node [below right] at (-.5,0) {$\frac{1}{2}(1,1) \times \Aff^1 $};
\draw [fill=black] (0,1) circle [radius=.06]; 
\node [right] at (0,1) {$\frac{1}{4}(1,1,2)$};
\draw [fill=black] (-1,0) circle [radius=.06]; 
\node [above] at (-1,0) {$p$};
\node at (0,-2.6) {$Z_{26}$};
\end{tikzpicture}
\end{center}

Furthermore, the isolated singularities $p$ and $v$ are themselves smoothable.  Because the local to global deformation theory is unobstructed, we can combine the partial smoothings of each component of the singular locus of $\PP(1,4,10,25)$ in every possible way to obtain seven different partial smoothings.
\end{example}

Let $X_{26}$ be the partial smoothing of $\PP(1,4,10,25)$ with only one singular point, a unique singularity of type $p$.  By a careful dimension count, we obtain the following analogue of Proposition \ref{prop:Xdivisor}. 

\begin{proposition}\label{prop:X26divisor}
    Let $X_{26}$ be the partial smoothing of $\PP(1,4,10,25)$ with a unique singularity of type $p$.  If the pair $\left( X_{26}, \frac{4}{d} D \right)$ is log terminal for the general member $D$ in the $\mathbb{Q}$-linear system $|-\frac{d}{4} K_{X_{26}}|$, then there is a divisor in the moduli space of \he stable pairs of degree $d$ parameterizing surfaces on $X_{26}$.  
\end{proposition}

\begin{remark}
Propositions \ref{prop:Xdivisor} and \ref{prop:X26divisor} are the higher dimensional version of calculations in \cite{hacking}: curves on the surfaces $\PP(1,1,4)$ and $X_{26}$, the smoothing of the $\frac{1}{4}(1,1)$ singularity on $\PP(1,4,25)$, form divisors in the moduli space of degree $d$ plane curves, provided the curves have appropriate singularities. 
\end{remark}

\section{Quintic surfaces}\label{sec:quintics}

Because $\PP(H^0(\PP^3, \calO(5)) \cong \PP^{55}$ and $\dim \Aut \PP^3 = 15$, we have a $40$-dimensional space of quintic surfaces in $\PP^3$.  However, just fixing numerical invariants, we obtain a moduli space of smooth quintic surfaces with an additional component \cite{horikawa}.  Smooth quintic surfaces have numerical invariants $K_S^5 = 5$, $p_g = 4$, and $q = 0$ and the moduli space parameterizing these surfaces has two $40$ dimensional components.  The first component, consisting of \textit{type I} surfaces, parameterizes traditional quintic surfaces $S$ such that $K_S$ is very ample and defines an embedding $S \subset \PP^3$.  The second component parameterizes \textit{type IIa} surfaces such that $|K_S|$ has a base-point and $S$ admits a generically two-to-one morphism to $\PP^1 \times \PP^1$.  The two components meet along a divisor of dimension $39$ parameterizing \textit{type IIb} surfaces such that $|K_S|$ has a base-point and $S$ admits a generically two-to-one morphism to $\F_2$.  For an image of the moduli space and the construction of type II surfaces, see \cite{rana}.  One might naturally ask how the moduli space of pairs defined in this paper encompasses surfaces of type II.  To describe the type II surfaces, we recall how to embed surfaces them into weighted projective spaces, worked out in \cite{quinticsurfaces}.  There is a typographical error in the main theorem in \cite{quinticsurfaces} in the first relation $r_1$, but it is stated correctly below.  

\begin{theorem}[Griffin]\label{thm:griffin}
Let $S$ be a numerical quintic surface of type II.  Then, $$S = \PP(1,1,1,1,2,3,3)/I$$ where $\PP(1,1,1,1,2,3,3)$ has coordinates $(x_0,x_1,x_2,x_3,y,z_1,z_2)$ and $I$ is generated by the relations

\begin{tabular}{l}
     $r_1: x_1x_3 - x_2^2 = - \beta x_0^2 $ \\
     $r_2: x_1y - (x_2 + \beta x_0)(x_3^2 + \gamma x_0 x_3 + \delta x_0^2) = 0$ \\
     $r_3: (x_2 - \beta x_0)y - x_3(x_3^2 + \gamma x_0 x_3 + \delta x_0^2) = 0$ \\
     $r_4: x_1z_2 - (x_2 + \beta x_0)z_1 = 0 $ \\
     $r_5: (x_2 - \beta x_0)z_2 - x_3 z_1  = 0$ \\
     $r_6: z_1y - z_2(x_3^2 + \gamma x_0 x_3 + \delta x_0^2) = 0$\\
     $r_7: z_1^2 - \lambda y x_3^4 - x_1 Q(x_i, y) - x_0 e_1 = 0$ \\
     $r_8: z_1z_2 - \lambda y^2 x_3^2 - x_2Q(x_i, y) - x_0 e_2 = 0$ \\
     $r_9: z_2^2 - \lambda y^3 - x_3 Q(x_i, y) - x_0 e_3 = 0$  \\
\end{tabular}

where $Q$ and $e_i$ are weight 5 polynomials satisfying certain conditions.  The surface $S$ is of type IIb if $\beta = 0$ and type IIa if $\beta \ne 0$. 
\end{theorem}

The simplest example is $\beta = \gamma = \delta = \lambda = e_i = 0$.  In this case, we will show that $S$ is a hypersurface of degree $50$ on $X = \PP(1,4, 10, 25)$, so it satisfies $5K_X + 4S \sim 0$.

\begin{example}\label{example:griffin}
    Let $X = \PP(1,4,10,25)$ with coordinates $a_0, a_1, a_2, a_3$.  First, consider the embedding $$X \to \PP(1,2,5,13,25)$$ given by $$(a_0,a_1,a_2,a_3) \mapsto (a_0^2, a_1, a_2, a_0a_3, a_3^2)$$ so that, if $\PP(1,2,5,13,25)$ has coordinates $b_0,b_1,b_2,b_3,b_4$, then $$X : (b_0b_3 - b_4^2 = 0) \subset \PP(1,2,5,13,25).$$  
    
    Then, consider the embedding $$\PP(1,2,5,13,25) \to \PP(1,1,1,1,2,3,3,5)$$ given by $$(b_0,b_1,b_2,b_3,b_4) \mapsto (b_2, b_0^5, b_0^3b_1, b_0b_1^2, b_1^5, b_0^2b_3, b_1b_3, b_4).$$ 
    
    If $\PP(1,1,1,1,2,3,3,5)$ has coordinates $x_0,x_1,x_2,x_3,y,z_1,z_2,$ and $t$, in the composition $$X \to \PP(1,1,1,1,2,3,3,5)$$ we find that $X$ is defined almost exactly by the equations in Theorem \ref{thm:griffin}, with $\beta = \gamma = \delta = \lambda = e_i = 0$, the only difference is that $t = Q(x_i,y)$.  To obtain the surface $S$, let $S$ be $t = Q(x_i,y)$ in $\PP(1,1,1,1,2,3,3,5)|_X$.  As desired, $S$ has degree $50$ on $X$.

\end{example}

By using various partial smoothings of $\PP(1,4,10,25)$, we can obtain all surfaces of type IIb as a hypersurface on one smoothing of $\PP(1,4,10,25)$, with the generic one on $X_{26}$.  

Returning to $\calM_{\PP^3\s,(5,4)}$, the previous result shows that pairs $(X_{26},S)$ form a divisor $\calD_1$ in the moduli space: it parameterizes the surfaces of type IIb. By Proposition \ref{prop:Xdivisor}, we know there is a second divisor $\calD_2$ in $\calM_{\PP^3\s,(5,4)}$ parameterizing surfaces on $X$, the cone over the anticanonical embedding of the quadric surface.  Because $X$ is a section of $\calO_W(2)$ for the weighted projective space $W = \PP(1,1,1,1,2)$, we compute that $\calO_X(K_X) = \calO_W(-4)|_X$ and generic $D$ satisfying $5 K_X + 4D \sim 0$ is a section of $\calO_X(D) = \calO_W(5)|_X$.  For general $D \in \calO_W(5)|_X$, because $D$ is a complete intersection in $\PP(1,1,1,1,2)$, we can compute the singularities as in \cite[Section 1.7]{wpsfletcher}.  The computation shows that $D$ has a unique $\frac{1}{4}(1,1)$ singularity at the vertex of $X$ (c.f. Proposition \ref{prop:DQ}).  Furthermore, by Proposition \ref{prop:smoothdivisor}, we know $\calM_{\PP^3\s,(5,4)}$ is smooth at the generic point of $\calD_2$.  In \cite[Theorem 1.5]{rana}, it is shown that there is a divisor $\calD'$ in the moduli space of stable quintic surfaces parameterizing surfaces whose unique non Du Val singularity is of type $\frac{1}{4}(1,1)$.  The component $\calD_2$ in $\calM_{\PP^3\s,(5,4)}$ found here parameterizes the surfaces Rana calls `type 1' (appearing as a divisor on the component parameterizing surfaces of type I).  In other words, for general $S$ such that $[S] \in \calD'$ in Rana's work, $S$ appears as a divisor on the threefold $X$ where $[(X,S)] \in \calD_2$ in this interpretation.  

We can describe components of higher codimension using the work in Section \ref{sec:lt}.  From Example \ref{ex:1124}, we know that $\PP(1,1,2,4)$ admits a smoothing to $X$, so should correspond to a higher codimension piece of $\calM_{\PP^3\s,(5,4)}$.  Indeed, a toric computation shows that the projective dimension of the automorphism group of $\PP(1,1,2,4)$ is $17$, and surfaces on $Z = \PP(1,1,2,4)$ satisfying $5K_Z + 4D \sim 0$ are elements of the linear system $| \calO_Z(10)|$.  This linear system has projective dimension $55$, so the space parameterizing surfaces on $\PP(1,1,2,4)$ has dimension 38.  This is a codimension 2 component of $\calM_{\PP^3\s,(5,4)}$ that is codimension 1 inside $\calD_2$.  Furthermore, a computation as in \cite[Section 1.7]{wpsfletcher} shows that the surfaces appearing on $\PP(1,1,2,4)$ have two singularities: $\frac{1}{4}(1,1)$ and $\frac{1}{2}(1,1)$.  
We could continue further: there is a 37-dimensional (or codimension 3) component parameterizing surfaces on $\PP(1,2,9,12)$.  This admits a smoothing to $\PP(1,1,2,4)$ (see Section \ref{sec:lt}) and the surfaces appearing on $\PP(1,2,9,12)$ have an additional $\frac{1}{9}(1,2)$ singularity. 

We can also describe some non-normal threefolds appearing using Propositions \ref{prop:p2} and \ref{prop:sum}.  For instance, considering the mutations going from $(1,1,1,1)$ to $(1,1,2,4)$, from $(1,1,2,4)$ to $(1,4,10,25)$, and from $(1,1,2,4)$ to $(1,2,9,12)$, we obtain the non-normal threefolds $\PP(1,1,1,2)\cup \PP(1,1,1,2)$, $\PP(1,1,2,5) \cup \PP(1,4,2,5)$, and $\PP(1,1,3,4) \cup \PP(1,2,3,4)$.

In fact, we can strengthen the result on non-normal threefolds as follows.  Theorem \ref{thm:slt} implies that any threefold $X$ appearing in one of these pairs is semi dlt.  In the case of quintic surfaces, it is in fact semi plt. 

\begin{theorem}\label{quinticsplt}
The varieties $X$ occurring in a $(\mathbb{P}^3,5,4)$ \he stable pair have at most two components, so each component $(X^\nu, \Delta)$ of the normalization is plt.  
\end{theorem}

\begin{proof}
    If $X$ is not normal and the double locus $\Delta$ on $X^\nu$ has more than one component, $\Delta$ must be connected by \cite[Theorem 17.4]{kollarplus}.  In particular, if $(X^\nu, \Delta)$ is not plt, two components $\Delta_i$ and $\Delta_j$ meet along a curve \cite[Theorem 4.16(2)]{newkollar}.  Call this curve $C$.  By hypothesis, $-(K_{X^\nu} +\Delta)$ is ample, and writing $\Delta' = \Delta \setminus (\Delta_1 + \Delta_2)$, we compute $-(K_{X^\nu} + \Delta_1 + \Delta_2 + \Delta')\cdot C = -2 + \sum(1 - \frac{1}{m_i}) + \Delta'\cdot C < 0$, where $m_i$ is the index of any singularities along $C$.  If $\Delta'$ is not empty, then there is at most one such singularity and $-(K_{X^{\nu}} + \Delta)\cdot C = -\frac{1}{m}$ or $-1$, and $D \cdot C \in \frac{1}{m} \mathbb{Z}$, so the relationship $dK_X + 4D \sim 0$ implies that $d$ is even.  Assume then $\Delta'$ is empty.  If there are no singular points of index $m_i >1$ along $C$, we also get that $d$ is even.  However, there may be one or two singularities along $C$, in which case $-(K_{X^\nu} + \Delta_1 + \Delta_2) \cdot C = -1 - \frac{1}{m}$ or $-\frac{1}{m_1} - \frac{1}{m_2}$.  Note that these do \textit{not} necessarily violate the condition that $dK_X + 4D \sim 0$ (for example, if there is only singular point of index $3$).  Assume that $dK_X + 4D \sim 0$ and $d$ is odd.  Then, we must have $m = 3 \mod 4$ in the case of one singular point, or $m_1+m_2 = 0 \mod 4$.
    
    Consider the first case (one singular point of index $m$) and now assume $d = 5$.  The second is similar.  If $D$ misses the singular point, then $m$ divides $d$ because $dK_X$ is Cartier in a neighborhood of the point, but $m = 3 \mod 4$, so this is impossible.  If $D$ passes through the singular point, the pair $(X, \frac{4}{5}D)$ is slc, so $(X^\nu, \Delta + \frac{4}{5}D^\nu)$ is lc and resolving the point $\pi: Y \to X^\nu$, we find that this is possible only if the curve $D|_{\Delta_i}$ passes through the singular point with multiplicity 1.  Restricting to $\Delta_i$, suppose the singular point is of type $\frac{1}{m}(1,a)$, and $m = 3 \mod 4$.  Using local coordinates $D|_{\Delta_i} = x^iy^j + \dots$ and the relationship $4D|_{\Delta_i} = -5(K_{\Delta_i} + C)$, we find that $5 = 1 \mod m$, but $m = 3 \mod 4$, and this is impossible. 
\end{proof}

\begin{remark}
For larger degree $d$, even when $d$ is odd, the threefolds need not be semi plt: there is a degeneration of $\PP^3$ that is a union of 6 components, isomorphic to $\PP(1,1,2,3)$, glued in a cycle such that the double locus in each component is $\Delta_1 + \Delta_2 $, $\Delta_1 \cong \bP(1,1,3)$ and $\Delta_2 \cong \bP(1,2,3)$.  On each component, $D \in |\mathcal{O} (d)|$ (and this can occur for odd degree $d$; even for instance $d = 7$). 
\end{remark}

For general degree, to explicitly describe all threefolds appearing in $\calM_{\dfours}$, we must complete the classification begun in Section \ref{sec:lt}.  Furthermore, if we denote by $\calM_d^{GIT}$ the GIT moduli space of degree $d$ surfaces, one expects a rational map $$\calM_{\dfours} \dashrightarrow \calM_d^{GIT}$$ although understanding this map would require a better understanding of both $\calM_{\dfours}$ and $\calM_d^{GIT}$.  This will be explored in future work.

\bibliographystyle{alpha}
\bibliography{main}

\end{document}